\documentclass[11pt,a4paper,reqno]{amsart}
\usepackage[latin1]{inputenc}
\usepackage{amsmath}
\usepackage{amsfonts}
\usepackage{amssymb}
\usepackage{amsthm}
\usepackage{anysize}
\usepackage{enumitem}
\usepackage{amscd}
\usepackage{hyperref}
\usepackage[square, comma, numbers, sort&compress]{natbib}
\usepackage{indentfirst}
\marginsize{3.0cm}{3.0cm}{3.0cm}{3.0cm}
\setenumerate{label={\normalfont(\arabic*)}}
\newcommand{\form}[1]{{\langle #1 \rangle }}

\newcommand{\pfister}[1]{{\langle \! \langle #1 \rangle \! \rangle}}

\newcommand{\mydim}[1]{{\mathrm{dim}\!\; #1}}

\newcommand{\Izhdim}[1]{{\mathrm{dim}_{\mathrm{Izh}}\!\; #1}}
\newcommand{\anispart}[1]{#1_{\mathrm{an}}}

\newcommand{\witti}[2]{{\mathfrak{i}_{#1}(#2)}}

\newcommand{\simform}[1]{{#1_{\mathrm{sim}}}}
\newcommand{\normform}[1]{{#1_{\mathrm{nor}}}}
\newcommand{\stb}[0]{\stackrel{\mathrm{stb}}{\sim}}

\newtheorem{theorem}{Theorem}[section]

\newtheorem{lemma}[theorem]{Lemma}
\newtheorem{claim}[theorem]{Claim}

\newtheorem{proposition}[theorem]{Proposition}
\newtheorem{corollary}[theorem]{Corollary}
\newtheorem{conjecture}[theorem]{Conjecture}

\theoremstyle{definition}
\newtheorem{definition}[theorem]{Definition}
\newtheorem{example}[theorem]{Example}
\newtheorem{examples}[theorem]{Examples}

\theoremstyle{remark}
\newtheorem{remark}[theorem]{Remark}
\newtheorem{remarks}[theorem]{Remarks}

\numberwithin{equation}{section}
\begin{document}

\title[Extended Karpenko-Merkurjev theorems for quasilinear quadratic forms]{Extended Karpenko and Karpenko-Merkurjev theorems for quasilinear quadratic forms}
\author{Stephen Scully}
\address{Department of Mathematics and Statistics, University of Victoria}
\email{scully@uvic.ca}


\maketitle

\begin{abstract} Let $p$ and $q$ be anisotropic quasilinear quadratic forms over a field $F$ of characteristic $2$, and let $i$ be the isotropy index of $q$ after scalar extension to the function field of the affine quadric with equation $p=0$. In this article, we establish a strong constraint on $i$ in terms of the dimension of $q$ and two stable birational invariants of $p$, one of which is the well-known ``Izhboldin dimension'', and the other of which is a new invariant that we denote $\Delta(p)$. Examining the contribution from the Izhboldin dimension, we obtain a result that unifies and extends the quasilinear analogues of two fundamental results on the isotropy of non-singular quadratic forms over function fields of quadrics in arbitrary characteristic due to Karpenko and Karpenko-Merkurjev, respectively. This proves in a strong way the quasilinear case of a general conjecture previously formulated by the author, suggesting that a substantial refinement of this conjecture should hold. 
 \end{abstract}

\section{Introduction} \label{SECintroduction} 

An important general problem in the theory of quadratic forms over arbitrary fields is to understand how invariants of quadratic forms can behave under scalar extension to the function field of a quadric. Already of considerable interest here is the behaviour of the most basic invariant, namely the \emph{isotropy index}.\footnote{The maximal dimension of a totally isotropic subspace of the vector space of definition.} From an algebraic-geometric viewpoint, the problem here is to understand when a rational map can exist from one quadric to another, or from one quadric to some higher quadratic Grassmannian of another. 

Let $F$ be an arbitrary field, let $p$ and $q$ be anisotropic quadratic forms of dimension $\geq 2$ over $F$, and let $s$ be the unique integer for which $2^s< \mydim{p} \leq 2^{s+1}$. Let $F(p)$ be the function field of the (integral) affine quadric $X_p$ with equation $p=0$, let $\witti{0}{q_{F(p)}}$ be the isotropy index of $q$ over $F(p)$, and let $k$ denote the (non-negative!) integer $\mydim{q} - 2\witti{0}{q_{F(p)}}$. Since the isotropy index is insensitive to rational extension (\cite[Lem. 7.15]{EKM}), $\witti{0}{q_{F(p)}}$ and $k$ depend only on the stable birational type of $p$ (or more precisely, $X_p$). In studying them, we should therefore search for stable birational invariants of $p$ that are independent of $q$, but still exert some degree of influence. A remarkable observation, originally due to Hoffmann, is that the integer $s$ defined above is such an invariant. This is the outcome of the fundamental \emph{separation theorem}, which asserts that if $\mydim{q} \leq 2^s$, then $q$ cannot become isotropic over $F(p)$.\footnote{This was originally proved over fields of characteristic not 2 by Hoffmann in \cite{Hoffmann1}, and later extended to characteristic 2 by Hoffmann and Laghribi in \cite{HoffmannLaghribi2}.} Looking to the cases where isotropy does occur, we proposed in \cite{Scully2} a strong conjectural generalization of this result:

\begin{conjecture}[\cite{Scully2}] \label{CONJfirstconjecture} If $q_{F(p)}$ is isotropic, then $\mydim{q} = a2^{s+1} + \epsilon$ for some positive integer $a$ and integer $\epsilon \in [-k,k]$ $\big($with $\epsilon \equiv k \pmod{2}\big)$. \end{conjecture}

In short, the more isotropy that occurs, the closer $\mydim{q}$ should be to being divisible by $2^{s+1}$, with divisibility being forced when $\witti{0}{q_{F(p)}}$ attains its largest possible value of $\frac{\mydim{q}}{2}$. This establishes a somewhat unexpected link between the separation theorem and other classical results on the latter extremity (e.g., Fitzgerald's theorem -- see \cite[\S 3.2]{Scully2}).  

When $p$ and $q$ are nonsingular\footnote{This is automatically the case if the characteristic of $F$ is not $2$.}, it was shown in \cite{Scully5} that Conjecture \ref{CONJfirstconjecture} holds in a large number of cases, including the case where $k \leq 2^{s-1} + 2^{s-2}$. This work relies on algebraic-geometric tools that have been at the heart of many of the major advances in the subject (and, more broadly, the study of index-reduction problems for algebraic groups) since the 1990s. While some recent developments have facilitated the extension of certain aspects of the algebraic-geometric approach to the study of singular forms (see \cite{Karpenko3, ScullyZhu}), there remain cases of Conjecture \ref{CONJfirstconjecture} that must be handled by alternative means, and the most apparent of these is that where $q$ is a so-called \emph{quasilinear} form. Indeed, one characterization of quasilinearity is that the projective quadric defined by the form has no smooth points at all, which renders standard algebraic-geometric methods very limited. A more concrete characterization, however, is that quasilinear forms preserve addition of vectors, and this makes the study of these forms more amenable to direct methods, even when the problems of interest are inherently algebraic-geometric. As far as Conjecture \ref{CONJfirstconjecture} is concerned, the case where $q$ is quasilinear reduces to the case where both $p$ and $q$ have this property, since it is easy to see that no isotropy can occur here unless $F(p)$ is inseparable over $F$, forcing the smooth locus of $X_p$ to be empty and $p$ to be quasilinear. We refer to this as the \emph{quasilinear case} of our conjecture. 

The quasilinear case of Conjecture \ref{CONJfirstconjecture} was studied in \cite{Scully4}. Among other things, it was shown there that the statement holds when $k \leq 2^{s-1} + 2^{s-2}$, mirroring the aforementioned result of \cite{Scully5} on the case where $p$ and $q$ are nonsingular. In the present article, we prove the assertion for all values of $k$. In fact, we are able to go much further. More specifically, while the general optimality of Conjecture \ref{CONJfirstconjecture} is known, one expects it to fail outside the cases where $p$ has simplest possible stable birational type. In other words, one expects a refinement involving more informative stable birational invariants of $p$ that separate the simplest types from the others. In the quasilinear case, the desired refinement is achieved with the main result of this paper (Theorem \ref{THMmainresult}), which constrains $\mydim{q}$ in terms of $k$ and two discrete stable birational invariants of $p$ that together capture a considerable amount of non-trivial information. The first of these is the \emph{Izhboldin dimension} $\Izhdim{p}$, defined here as the integer $\mydim{p} - \witti{1}{p}$, where $\witti{1}{p} = \witti{0}{p_{F(p)}}$.\footnote{From an algebraic-geometric perspective, it is more natural to consider the integer $\mydim{p} - \witti{1}{p} -1$, and this is what one finds in much of the literature (e.g., \cite{EKM}). The definition given here will be more convenient for our purposes.} This invariant obviously exists within the general theory, and has long been known to be important for the problem under consideration. It takes values in the interval $[2^s,\mydim{p}-1]$ (and hence sees $s$), and takes the minimal value of $2^s$ when $p$ belongs to the class of quasilinear forms with simplest stable birational type, the \emph{quasi-Pfister neighbours}.\footnote{These are the obvious quasilinear analogues of nonsingular \emph{Pfister neighbours}, which are well-known to have the simplest stable birational types among nonsingular forms.} The second invariant, which we denote $\Delta(p)$, is new, and has no known extension to the general theory. It represents a substantial refinement of the \emph{norm degree} invariant introduced by Hoffmann and Laghribi in \cite{HoffmannLaghribi1}, and comprises a certain set of non-negative integers bounded by the former. For a fixed value of $s$, $\Delta(p)$ detects whether $p$ is a quasi-Pfister neighbour, but also carries much more information beyond, and the results established here indicate that it is of central importance for the kind of problems we are trying to address.

Now a basic motivation for examining the quasilinear case is to develop an idea of what one can expect within the general theory, and while we know of no general substitute for the invariant $\Delta(p)$ introduced here, the invariants $s$ and $\Izhdim{p}$ exist within the wider framework. Upon examining their contribution to our main result, we not only obtain the quasilinear case of Conjecture \ref{CONJfirstconjecture}, but a strong enhancement of it. More specifically, note that Conjecture \ref{CONJfirstconjecture} is vacuously true when $k \geq 2^s$. Since $\Izhdim{p} \geq 2^s$, the quasilinear case therefore amounts to the first part of the following result:

\begin{theorem} \label{THMmaintheoremcorollary} Suppose that $F$ has characteristic 2, $p$ and $q$ are quasilinear, and $q_{F(p)}$ is isotropic with $k<\Izhdim{p}$. For each non-negative integer $r$, let $y_r$ be the largest integer for which $\Izhdim{p} > y_r2^r$. Then:
\begin{enumerate} \item $\mydim{q} = a2^{s+1} + \epsilon$ for some positive integer $a$ and integer $\epsilon \in [-k,k]$;
\item If $p$ is not a quasi-Pfister neighbour, then one of the following holds:
\begin{itemize} \item[$\mathrm{(i)}$] $\mydim{q} = a2^{s+2} + \epsilon$ for some positive integer $a$ and integer $\epsilon \in [-k,k]$;
\item[$\mathrm{(ii)}$] $\Izhdim{p} = 2^s$, and there exist a non-negative integer $r \leq s-2$ with $k \geq 2^s - 2^r$, and positive integers $x \leq 2^{s-2-r}$ and $\epsilon \in [(x-1)2^{r+1} + 2^{s+1} - k, x2^{r+1} + k]$, such that $\mydim{q} = a2^{s+2} \pm \epsilon$ for some non-negative integer $a$;
\item[$\mathrm{(iii)}$] $\Izhdim{p} > 2^s$, and there exist a non-negative integer $r \leq s-1$ with $k \geq y_r2^r$, and positive integers $x < 2^{s+1-r} - y_r$ and $\epsilon \in [(x+y_r)2^{r+1}-k,x2^{r+1} +k]$, such that $\mydim{q} = a2^{s+2} + \epsilon$ for some non-negative integer $a$.  \end{itemize} \end{enumerate} \end{theorem}

Unless additional information is taken into account, this result is in fact the best possible: Modulo the requirements that $\mydim{q} \in k + 2\mathbb{N}$ and $\Izhdim{p} \in [2^s,2^{s+1})$, the only values of the quadruple $(s,\Izhdim{p},\mydim{q}, k)$ that cannot be realized by an appropriate triple $(F,p,q)$ are those excluded by Theorem \ref{THMmaintheoremcorollary} (in particular, the case where $k \geq \Izhdim{p}$ is unrestricted). This is shown in \S 5. When $p$ is not a quasi-Pfister neighbour, a sufficiently small value of $k$ forces us into the more palatable case (i) of (2). More specifically:

\begin{corollary} \label{CORmaintheoremcorollarycritical} Suppose that $F$ has characteristic $2$, $p$ and $q$ are quasilinear and $q_{F(p)}$ is isotropic. Suppose further that $p$ is not a quasi-Pfister neighbour, and that
$$ k < \begin{cases} 2^s + 2^{s-1} & \text{if } \Izhdim{p} > 2^s + 2^{s-1} \\ 2^{s} & \text{if } \Izhdim{p} \in (2^s, 2^{s} + 2^{s-1}] \\ 2^{s-1} + 2^{s-2} & \text{if } \Izhdim{p} = 2^s. \end{cases} $$
Then $\mydim{q} = a2^{s+2} + \epsilon$ for some positive integer $a$ and integer $\epsilon \in [-k,k]$. 
\begin{proof} In all cases, $k<\Izhdim{p}$, and so Theorem \ref{THMmaintheoremcorollary} (2) is applicable. Suppose that $\Izhdim{p} > 2^s$. If $r$ is a positive integer $\leq s-1$, and $y_r$ is the largest integer for which $\Izhdim{p} > y_r2^r$, then $y_r2^r$ is at least $2^s$, and at least $2^s + 2^{s-1}$ in the case where $\Izhdim{p} > 2^s + 2^{s-1}$. By hypothesis, we then have that $k < y_r2^r$, and so we cannot be in case (iii) of Theorem \ref{THMmaintheoremcorollary} (2).  Similarly, if $\Izhdim{p} = 2^s$, then our assumption on $k$ tells us that we are not in case (ii) of Theorem \ref{THMmaintheoremcorollary} (2), and the result again holds. 
\end{proof}
\end{corollary}

In fact, in the conclusion of Corollary \ref{CORmaintheoremcorollarycritical}, the integer $2^{s+2}$ may be replaced with Hoffmann and Laghribi's norm degree (which is at least $2^{s+2}$ here) -- see Corollary \ref{CORmainresult2}. 

To explain the title of the article, we now give three other notable consequences of Theorem \ref{THMmaintheoremcorollary}. The first of these was originally shown in \cite[Thm. 1.3]{Scully1}, and the second by Totaro in \cite[Thm. 5.2]{Totaro}. The third is new, but very closely related to \cite[Cor. 6.18]{Scully3}. 

\begin{corollary} \label{CORintro} Suppose that $F$ has characteristic 2 and that $p$ and $q$ are quasilinear. Let $2^u$ be the largest power of $2$ dividing $\Izhdim{p}$. Then the following hold:
\begin{enumerate} \item $\witti{1}{p} \leq 2^u$;
\item If $q_{F(p)}$ is isotropic, then:
\begin{enumerate} \item[$\mathrm{(i)}$] $\mydim{q} > \Izhdim{p}$;
 \item[$\mathrm{(ii)}$] $\witti{0}{q_{F(p)}} \leq \mathrm{max}\lbrace \mydim{q} - \Izhdim{p} - 2^u, 2^u \rbrace$. \end{enumerate} \end{enumerate} 
\begin{proof} The form $p_{F(p)}$ is isotropic, and setting $q = p$ in the second part of (2) gives the inequality $\witti{1}{p} \leq \mathrm{max}\lbrace \witti{1}{p} - 2^u, 2^u \rbrace$, which is obviously equivalent to (1). It therefore suffices to prove (2). Let us first note that (ii) may be re-written as
\begin{enumerate} \item[(ii)'] $k \geq \mathrm{min}\lbrace 2\Izhdim{p} - (\mydim{q} - 2^{u+1}), \mydim{q} - 2^{u+1} \rbrace.$ \end{enumerate}
 If $\Izhdim{p} = 2^s$, then (i) says that $\mydim{q} > 2^s$, and (ii)' says that $k \geq 2^{s+1} - \mydim{q}$. Both these inequalities are immediate from Theorem \ref{THMmaintheoremcorollary} (1) (the first being the separation theorem), so we can assume that $\Izhdim{p} > 2^s$. As remarked in the discussion preceding Theorem \ref{THMmaintheoremcorollary}, this implies that $p$ is not a quasi-Pfister neighbour. Observe now that (i) and (ii)' are both satisfied if 
\begin{itemize} \item[(a)] $k \geq \Izhdim{p} - 1$, or
\item[(b)] $\mydim{q} \geq 2\Izhdim{p} + 2^{u+1}$.
\end{itemize}
Indeed, if (a) holds then the validity of (ii)' is evident, while (i) holds since $\mydim{q} = k+2\witti{0}{q_{F(p)}} \geq k+2$. On the other hand, if (b) holds, then the validity of (i) is evident, while (ii)' holds since $k \geq 0$. Now since $\mydim{p} \leq 2^{s+1}$, we have $\Izhdim{p} \leq 2^{s+1} - 2^u$, and so $2\Izhdim{p} + 2^{u+1} \leq 2^{s+2}$. We have therefore reduced to the case where $p$ is not a quasi-Pfister neighbour, $k \leq \Izhdim{p} - 2$ and $\mydim{q} < 2^{s+2}$. By Theorem \ref{THMmaintheoremcorollary} (2), there then exist positive integers $r \leq s+1$ and $x \leq 2^{s+1-r}$ such that if $y_r$ is the largest integer with $\Izhdim{p} \geq y_r2^r$, then $k \geq y_r2^r$ and $\mydim{q} \in [(x+y_r)2^{r+1}-k,x2^{r+1} + k]$ (if $\mydim{q} = 2^{s+2} - \epsilon$ for some $\epsilon \in [1,k]$, then we can take $r = s+1$ and $x = 1$). Now, since $x$ is positive, the lower bound for $\mydim{q}$ is at least $(1+y_r)2^{r+1} - k \geq 2\Izhdim{p} - k \geq \Izhdim{p} +2$, and so (i) holds. For (ii)', let us suppose that $k < 2\Izhdim{p} - (\mydim{q} - 2^{u+1})$. Combining this with the lower bound for $\mydim{q}$, we then get that
$$ (x-1)2^r + \Izhdim{p} \leq (x-1)2^r + (1+y_r)2^r = (x+y_r)2^r \leq \frac{\mydim{q}+k}{2} < 2^u + \Izhdim{p}. $$
Now if $u$ were less than $r$, then $2^u + \Izhdim{p} $ would be at most $(1+y_r)2^r$ by the definition of $y_r$. Since $x$ is positive, the preceding inequalities therefore not only imply that $2^u > (x-1)2^r$, but that $u \geq r$, and hence that $2^u \geq x2^r$. The upper bound for $\mydim{q}$ then gives that $k \geq \mydim{q} - x2^{r+1} \geq \mydim{q} - 2^{u+1}$, and so (ii)' holds.  \end{proof} \end{corollary}

Now the point we wish to emphasize is that these statements are not only true when $p$ and $q$ are quasilinear, but also when $p$ and $q$ are nonsingular. Indeed, in that situation, the second part of (2) is \cite[Thm. 4.1] {Scully3}, while (1) and (2)(i) are celebrated results due to Karpenko (\cite{Karpenko1}) and Karpenko-Merkurjev (\cite{KarpenkoMerkurjev,EKM}), respectively.\footnote{This explains the title of the article. We should remark that \cite{Karpenko1} not only assumed nonsingularity, but that the characteristic of the ground field is not 2. This stronger assumption was rendered unnecessary, however, by the work of Primozic on motivic Steenrod operations in positive characteristic (\cite{Primozic}).} In fact, as far as the latter two results are concerned, no restriction on $p$ or $q$ is required: The Karpenko-Merkurjev theorem was extended to the general case by Totaro in \cite{Totaro}, and Karpenko has recently extended his result to the (non-quasilinear) singular case in \cite{Karpenko3}. In view of this situation, we expect that our Theorem \ref{THMmaintheoremcorollary} also extends to the general case, with the term ``quasi-Pfister neighbour'' being replaced with a suitable formulation of ``simplest stable birational type''.\footnote{When $p$ is (sufficiently) nonsingular, this will simply be ``Pfister neighbour".} This will be investigated in forthcoming work. Again, what we prove here for quasilinear forms is actually much stronger (see Theorem \ref{THMmainresult}), but we don't know to what extent one can hope for this kind of enhancement in the general theory.

\section{Preliminaries on Quasilinear Quadratic Forms} \label{SECpreliminaries}

In this section, we collect various preliminary facts on quasilinear quadratic forms that will be used throughout the main part of the text. The basic references here are \cite{Hoffmann1, Scully1,Totaro}. We also establish our notation and terminology, which may, in places, differ slightly from that found in the existing literature. We fix throughout a field $F$ of characteristic 2.

\subsection{Basic Notions} \label{SUBSECbasics} Let $V$ be a finite-dimensional $F$-vector space. A map $\phi \colon V \rightarrow F$ is a \emph{quasilinear quadratic form on $V$} if $\phi(av + w) = a^2\phi(v) + \phi(w)$ for all $a \in F$ and $v,w \in V$. If $\mathfrak{b} \colon V \times V \rightarrow F$ is a symmetric bilinear form, then its restriction to the diagonal is a quasilinear quadratic form on $V$ which we denote $\phi_{\mathfrak{b}}$. Every quasilinear quadratic form on $V$ is of this type, but the bilinear form $\mathfrak{b}$ is far from unique. If $a_1,\hdots,a_n \in F$, then we write $\form{a_1,\hdots,a_n}$ for the quasilinear quadratic form on $F^n$ that sends $(x_1,\hdots,x_n)$ to $\sum_{i=1}^n a_ix_i^2$. By a \emph{quasilinear quadratic form over $F$}, we mean a quasilinear quadratic form on some finite-dimensional $F$-vector space. Isomorphisms of quasilinear quadratic forms over are defined in the standard way, and we use the symbol $\simeq$ to indicate the existence of an isomorphism between given forms. If a quasilinear quadratic form is isomorphic to a non-zero scalar multiple of another, then we say that the two forms are \emph{similar}.
The orthogonal sum and tensor product operations for symmetric bilinear forms give rise to corresponding operations for quasilinear quadratic forms (denoted $\perp$ and $\otimes$, respectively). If $\phi$ and $\psi$ are quasilinear quadratic forms over $F$, then we say that $\psi$ is a \emph{subform of $\phi$}, and write $\psi \subset \phi$, if $\phi \simeq \psi \perp \sigma$ for some quasilinear quadratic form $\sigma$ over $F$. If $\phi \cong \psi \otimes \sigma$ for some $\sigma$, then we say that $\phi$ is \emph{divisible by $\psi$}.

If $\phi$ is a quasilinear quadratic form over $F$, we shall write $V_{\phi}$ for the $F$-vector space on which it is defined. The \emph{dimension of $\phi$}, denoted $\mydim{\phi}$, is the dimension of $V_{\phi}$. If $\lbrace v_1,\hdots,v_n \rbrace$ is a basis of $V_{\phi}$, then $\phi \simeq \form{\phi(v_1),\hdots,\phi(v_n)}$. The set $V_{\phi}^0$ consisting of all $\phi$-isotropic vectors in $V_{\phi}$ is an $F$-linear subspace of $V_{\phi}$. Its dimension is the \emph{isotropy index} $\witti{0}{\phi}$ of \S 1. The restriction of $\phi$ to the quotient space $V_{\phi}/V_{\phi}^0$ is an anisotropic quasilinear quadratic form of dimension $\mydim{\phi} - \witti{0}{\phi}$ over $F$ which we denote $\anispart{\phi}$ and call the \emph{anisotropic part of $\phi$}. The form $\phi$ is isomorphic to the orthogonal sum of $\anispart{\phi}$ and the zero form of dimension $\witti{0}{\phi}$. If $\mydim{\anispart{\phi}} \leq 1$, then we say that $\phi$ is \emph{split}. If $L$ is a field extension of $F$, then we write $\phi_L$ for the quasilinear quadratic form on $V_{\phi} \otimes_F L$ induced by $\phi$. By definition, we then have $\witti{0}{\phi_L} \geq \witti{0}{\phi}$ and $\anispart{(\phi_L)} \subset (\anispart{\phi})_L$. 
 
Modulo the obvious terminological changes, the preceding discussion carries over verbatim to the study of quasilinear quadratic forms on finite-rank free modules over discrete valuation rings of characteristic 2 (defined the same way). In particular, if $R$ is a DVR of characteristic 2 with fraction field $K$ and residue field $k$, and $\phi$ is a quasilinear quadratic form on a finite-rank free $R$-module $M$, then the subset of $M$ on which $\phi$ vanishes is an $R$-linear direct summand of $M$, and so we may define $\witti{0}{\phi}$ to be its rank. Writing $\phi_K$ (resp. $\phi_k$) for the quasilinear quadratic form on the $K$-vector space $M \otimes_R K$ (resp. the $k$-vector space $M \otimes_R k$) induced by $\phi$, then we then clearly have:

\begin{lemma} \label{LEMspecializationlemma} $\witti{0}{\phi_K} = \witti{0}{\phi} \leq \witti{0}{\phi_k}$. \end{lemma}

If $\phi$ is a quasilinear quadratic form over $F$, then the value set $D(\phi): = \lbrace \phi(v)\;|\; v\in V_{\phi} \rbrace$ is a finite-dimensional $F^2$-linear subspace of $F$. For indeterminates $X_1,\hdots,X_n$, Lemma \ref{LEMspecializationlemma} gives the following basic specialization result:

\begin{corollary}[see {\cite[Cor. 3.7]{Hoffmann2}}] \label{CORspecialization} Let $\phi$ be a quasilinear quadratic form over $F$, and let $a_1,\hdots,a_n \in F$. Suppose $f \in F[X_1,\hdots,X_n]_{\mathfrak{m}}$, where $\mathfrak{m} = (X_1-a_1,\hdots,X_n - a_n)$. If $f \in D(\phi_{F(X_1,\hdots,X_n)})$, then $f(a_1,\hdots,a_n) \in D(\phi)$. 
\begin{proof} We may assume that $n=1$ and that $\phi$ is anisotropic. Now, since $f \in D(\phi_{F(X_1)})$, the form $\phi_{F(X_1)} \perp \form{f}$ is isotropic. Applying Lemma \ref{LEMspecializationlemma} with $R = F[X_1]_{\mathfrak{m}}$, we get that $\phi \perp \form{f(a_1)}$ is isotropic over $F$. Since $\phi$ is anisotropic, it then follows that $f(a_1) \in D(\phi)$. 
\end{proof}
\end{corollary}

Up to isomorphism, quasilinear quadratic forms are determined by their isotropy index and value set:

\begin{lemma}[see {\cite[Prop. 2.6]{Hoffmann2}}] \label{LEMsubformtheorem} If $\psi$ and $\phi$ are quasilinear quadratic forms over $F$, then $\anispart{\psi} \subset \anispart{\phi}$ if and only if $D(\psi) \subseteq D(\phi)$. In particular, $\anispart{\psi} \simeq \anispart{\phi}$ if and only if $D(\psi) = D(\phi)$. 
\end{lemma}

If $\phi$ and $\psi$ are quasilinear quadratic forms over $F$, then $D(\phi \perp \psi)$ is the image of the addition map $D(\phi) \oplus D(\psi) \rightarrow F$, and $D(\phi \otimes \psi)$ the image of the multiplication map $D(\phi) \otimes_{F^2} D(\psi) \rightarrow F$. Both maps are $F^2$-linear, and their kernels have dimension $\witti{0}{\phi \perp \psi}$ and $\witti{0}{\phi \otimes \psi}$, respectively. Moreover, we have the following:

\begin{lemma} \label{LEMsubformoftensorproduct} Let $\phi$ and $\psi$ be quasilinear quadratic forms over $F$, with $\phi$ being anisotropic. Then $\phi \subset \anispart{(\phi \perp \psi)}$ and $a \phi \subset \anispart{(\phi \otimes \psi)}$ for all $a \in D(\psi) \setminus \lbrace 0 \rbrace$. 
\begin{proof} By the preceding remarks, we have $D(\phi) \subseteq D(\phi \perp \psi)$ and $D(a\phi) = aD(\phi) \subseteq D(\phi \otimes \psi)$ for all $a \in D(\psi)$. Now apply Lemma \ref{LEMsubformtheorem}. 
\end{proof}
\end{lemma}

If $\phi$ is a quasilinear quadratic form over $F$, then we shall write $X_{\phi}$ for the quadric hypersurface in $\mathbb{A}(V_{\phi})$ defined by the vanishing of $\phi$. If $\phi$ is not split, then $X_{\phi}$ is integral and we write $F(\phi)$ for its function field. The latter can described concretely as follows: Suppose $\phi \simeq \form{a} \perp \phi'$ for some $a \in F \setminus \lbrace 0 \rbrace$ and some $\phi' \subset \phi$. Let $F(V_{\phi'})$ denote the function field of $\mathbb{A}(V_{\phi'})$, and $\phi'(X)$ the element of $F(V_{\phi'})$ represented by $\phi'$. Then $F(\phi)$ is $F$-isomorphic to the field $F(V_{\phi'})\big(\sqrt{a^{-1}\phi'(X)}\big)$.

\subsection{Scalar Extension and Isotropy} \label{SUBSECscalarextension} Let $\phi$ be a quasilinear quadratic form over $F$. If $L$ is a field extension of $F$, then $D(\phi_L)$ is the image of the multiplication map $D(\phi) \otimes_{F^2}L^2 \rightarrow L$. This map is $L^2$-linear, and its kernel has dimension $\witti{0}{\phi_L}$. When $L$ is separable over $F$ (i.e., $L \otimes_F K$ is reduced for every field extension $K$ of $F$), the map is injective, and so:

\begin{lemma}[see {\cite[Prop. 5.3]{Hoffmann2}}] \label{LEManisotropyoverseparable} Let $L$ be a separable field extension of $F$. If $\phi$ is a quasilinear quadratic form over $F$, then $\witti{0}{\phi_L} = \witti{0}{\phi}$. \end{lemma}

\begin{remark} \label{REManisotropyoverfunctionfieldsofgenericallysmoothquadrics} In particular, if $\phi$ is anisotropic, then it remains anisotropic under scalar extension to the function field of any generically smooth variety over $F$. By the proof of \cite[Prop. 22.1]{EKM}, any affine quadric defined by the vanishing of a non-quasilinear quadratic form is generically smooth. Thus, if $\phi$ becomes isotropic over the function field of an affine quadric $X$ over $F$, then $X$ must be the vanishing locus of a quasilinear quadratic form.
\end{remark}

In studying the isotropy behaviour of quasilinear quadratic forms under scalar extension, it follows that only (towers of) inseparable quadratic extensions ultimately matter. Here, we have: 

\begin{lemma}[see {\cite[Lem. 2.27]{Scully1}}, {\cite[Prop. 5.10]{Hoffmann2}}] \label{LEMisotropyoverquadratic} Suppose $a \in F \setminus F^2$. If $\phi$ is an anisotropic quasilinear quadratic form over $F$, then 
$$ 2\witti{0}{\phi_{F(\sqrt{a})}} = \witti{0}{\pfister{a} \otimes \phi} = \mathrm{max}\lbrace \mydim{\tau}\;|\; \tau \subset \phi \text{ and } \tau \text{ is divisible by }\pfister{a} \rbrace. $$
In particular, $\witti{0}{\phi_{F(\sqrt{a})}} \leq \frac{\mydim{\phi}}{2}$. 
\end{lemma}

With the notation introduced at the end of \S \ref{SUBSECbasics}, this gives: 

\begin{lemma} \label{LEMisotropyoverfunctionfieldsofquadrics} Let $\phi$ and $\psi$ be anisotropic quasilinear quadratic forms over $F$ with $\mydim{\psi} \geq 2$. Suppose that $\psi \simeq \form{1} \perp \psi'$ for some subform $\psi' \subseteq \psi$. Then 
$$ 2\witti{0}{\phi_{F(\psi)}} = \witti{0}{\pfister{\psi'(X)} \otimes \phi_{F(V_{\psi'})}} = \mathrm{max}\lbrace \mydim{\tau}\;|\; \tau \subset \phi_{F(V_{\psi'})} \text{ and } \tau \text{ is divisible by }\pfister{\psi'(X)} \rbrace. $$
In particular, $\witti{0}{\phi_{F(\psi)}} \leq \frac{\mydim{\phi}}{2}$. 
\begin{proof} Since $F(\psi)$ is $F$-isomorphic to $F(V_{\psi'})(\sqrt{\psi'(X)})$, the claim follows immediately from Lemmas \ref{LEManisotropyoverseparable} and \ref{LEMisotropyoverquadratic}.
\end{proof} \end{lemma}

We remark that if $\phi$ is an anisotropic quasilinear quadratic form of dimension $\geq 2$ over $F$, then $\phi$ obviously becomes isotropic over $F(\phi)$.

\subsection{Quasi-Pfister Forms} \label{SUBSECquasiPfisters} Given a positive integer $n$ and elements $a_1,\hdots,a_n \in F$, we write $\pfister{a_1,\hdots,a_n}$ for the $2^n$-dimensional quasilinear quadratic form $\form{1,a_1} \otimes \cdots \otimes \form{1,a_n}$. A quasilinear quadratic form over $F$ which is isomorphic to $\form{1}$ or $\pfister{a_1,\hdots,a_n}$ for some $a_1,\hdots,a_n \in F$ is said to be a \emph{quasi-Pfister form}, or an \emph{$n$-fold quasi-Pfister form} when its dimension is $2^n$. These are the quasilinear quadratic forms associated to the well-known Pfister bilinear forms. An anisotropic quasilinear quadratic form over $F$ is a quasi-Pfister form if and only if $D(\phi)$ is a subfield of $F$ (see \cite[Prop. 4.6]{Hoffmann2}). If $\phi$ is a quasi-Pfister form, then $D(\phi)$ is a field, so $\anispart{\phi}$ is again a quasi-Pfister form.

\subsection{The Norm Form and Norm Degree} \label{SUBSECnormform} Let $\phi$ be a quasilinear quadratic form over $F$. The \emph{norm field of $\phi$}, denoted $N(\phi)$, is the smallest subfield of $F$ containing all products $ab$ with $a,b \in D(\phi)$. If $\phi \simeq \form{a_1,\hdots,a_n}$, then $N(\phi) = F^2(a_1a_2,\hdots,a_1a_n)$, so $N(\phi)$ is a finite extension of $F^2$. By Lemma \ref{LEMsubformtheorem}, there then exists, up to isomorphism, a unique anisotropic quasilinear quadratic form $\normform{\phi}$ over $F$ with $D(\normform{\phi}) = N(\phi)$. Since $N(\phi)$ is a field, $\normform{\phi}$ is a quasi-Pfister form. Its dimension, which is a power of $2$, is called the \emph{norm degree of $\phi$}, denoted $\mathrm{ndeg}(\phi)$. In the sequel, it will be more convenient to work with the integer $\mathrm{lndeg}(\phi): = \mathrm{log}_2\big(\mathrm{ndeg}(\phi)\big)$. The form $\normform{\phi}$ may also be characterized as follows:

\begin{lemma}[see {\cite[Lem. 2.11]{Scully1}}] \label{LEMuniversalpropertyofthenormform} Let $\phi$ be a quasilinear quadratic form over $F$. If $\pi$ is an anisotropic quasi-Pfister form over $F$, then $\anispart{\phi}$ is similar to a subform of $\pi$ if and only if $\normform{\phi} \subset \pi$. In particular, $\anispart{\phi}$ is similar to a subform of $\normform{\phi}$. 
\end{lemma}

By the lemma, saying that $\anispart{\phi}$ is similar to a quasi-Pfister form is equivalent to saying that $\anispart{\phi}$ is similar to $\normform{\phi}$. Moreover, if $2^n < \mydim{\anispart{\phi}} \leq 2^{n+1}$ for some integer $n$, then $\mathrm{lndeg}(\phi) \geq n+1$. Equality holds here if and only if $\anispart{\phi}$ is a so-called quasi-Pfister neighbour (see \S \ref{SUBSECstb} below). An upper bound for $\mathrm{lndeg}(\phi)$ is given by $\mydim{\phi} - 1$. This is realized, for instance, by the forms $\form{T_1,\hdots,T_n}$ over $F(T_1,\hdots,T_n)$, where $T_1,\hdots,T_n$ are indeterminates. If $L$ is a field extension of $F$, then it is immediate from the definitions that $\normform{(\phi_L)} \simeq \anispart{((\normform{\phi})_L)}$. In particular, the norm degree does not change under separable field extensions (Lemma \ref{LEManisotropyoverseparable}). For function fields of quasilinear quadrics, we have the following basic but important fact:

\begin{lemma}[see {\cite[Lemma 7.12]{Hoffmann2}}] \label{LEMdropingofnormdegree} Let $\phi$ and $\psi$ be anisotropic quasilinear quadratic forms of dimension $\geq 2$ over $F$. If $\phi_{F(\psi)}$ is isotropic, then $\normform{\psi} \subset \normform{\phi}$ and
$\mathrm{lndeg}\big(\phi_{F(\psi)}\big) = \mathrm{lndeg}(\anispart{(\phi_{F(\psi)})}) = \mathrm{lndeg}(\phi) -1$. 
\end{lemma}

\subsection{Similarity Factors and Divisibility by Quasi-Pfister Forms} Let $\phi$ be a quasilinear quadratic form over $F$. The set $G(\phi) = \lbrace a \in F\setminus \lbrace 0 \rbrace\;|\; a\phi \simeq \phi \rbrace \cup \lbrace 0 \rbrace$ is then a subfield of $F$ containing $F^2$ (see \cite[\S 6]{Hoffmann2}). Its nonzero elements are called the \emph{similarity factors of $\phi$}. One readily checks that $G(\phi)$ is in fact a finite extension of $F^2$. Like $N(\phi)$, it is therefore the value set of an anisotropic quasi-Pfister form over $F$, unique up to isomorphism. We denote this form $\simform{\phi}$. If $\phi$ is a quasi-Pfister form, then $G(\phi) = D(\phi)$ (because $D(\phi)$ is a subfield of $F$), and so $\simform{\phi} \simeq \anispart{\phi}$ by Lemma \ref{LEMsubformtheorem}. In general, we have $\simform{\phi}  = \simform{(\anispart{\phi})}$, and the following lemma then gives a characterization of this form:

\begin{lemma} \label{LEMdivisibilitybyquasiPfister} Let $\phi$ and $\psi$ be anisotropic quasilinear quadratic forms over $F$ with $\mydim{\psi} \geq 2$. Suppose that $\psi \simeq \form{1} \perp \psi'$ for some $\psi' \subset \psi$. Then the following are equivalent:
\begin{enumerate} \item $\phi$ is divisible by $\normform{\psi}$;
\item $D(\psi) \subseteq G(\phi)$;
\item $\psi \subset \simform{\phi}$;
\item $D(\phi)$ is closed under multiplication by arbitrary elements of $D(\psi)$;
\item $\anispart{(\pfister{a} \otimes \phi)} \simeq \phi$ for all $a \in D(\psi)$;
\item $\anispart{(\psi \otimes \phi)} \simeq \phi$;
\item $\anispart{(\pfister{\psi'(X)} \otimes \phi_{F(V_{\psi'})})} \simeq \phi_{F(V_{\psi'})}$.
\item $\witti{0}{\phi_{F(\psi)}} = \frac{\mydim{\phi}}{2}$.
\end{enumerate}
In particular, $\phi$ is divisible by $\simform{\phi}$. 
\begin{proof} $(1) \Rightarrow (2)$: Since $1 \in D(\psi)$, we have that $\psi \subset \normform{\psi}$. Thus, if (1) holds, then $D(\psi) \subseteq D(\normform{\psi}) = G(\normform{\psi}) \subseteq G(\normform{\phi})$.

$(2) \Leftrightarrow (3)$: Apply Lemma \ref{LEMsubformtheorem}.

$(2) \Rightarrow (4)$: Clear.

$(4) \Rightarrow (5)$: Let $a \in D(\psi)$. By the remarks preceding Lemma \ref{LEMsubformoftensorproduct}, $D(\anispart{(\pfister{a} \otimes \phi)}) = D(\phi) + aD(\phi)$. If (4) holds, if follows that $D(\anispart{(\pfister{a} \otimes \phi)}) = D(\phi)$, and so $\anispart{(\pfister{a} \otimes \phi)} \simeq \phi$ by Lemma \ref{LEMsubformtheorem}.

$(5) \Rightarrow (6)$: If $\psi \simeq \form{1,a_1,\hdots,a_n}$, then repeated application of $(5)$ gives that $\anispart{(\psi \otimes \phi)} \subset \anispart{(\pfister{a_1,\hdots,a_n} \otimes \phi)} \simeq \phi$, and so $\anispart{(\psi \otimes \phi)} \simeq \phi$ by Lemma \ref{LEMsubformoftensorproduct}.

$(6) \Rightarrow (7)$: By Lemma \ref{LEManisotropyoverseparable}, $\phi$ remains anisotropic over $F(V_{\psi'})$. Since $\pfister{\psi'(X)}$ is a subform of $\psi_{F(V_{\psi'})}$, (6) and Lemma \ref{LEMsubformoftensorproduct} then imply that $\anispart{(\pfister{\psi'(X)} \otimes \phi_{F(V_{\psi'})})} \simeq \phi_{F(V_{\psi'})}$.

$(7) \Rightarrow (8)$: Apply Lemma \ref{LEMisotropyoverfunctionfieldsofquadrics}.

$(8) \Rightarrow (1)$: See \cite[Thm. 6.10]{Hoffmann2}. \end{proof}
\end{lemma}

Note, in particular, that an anisotropic quasilinear quadratic form $\phi$ over $F$ is a quasi-Pfister form if and only if $\phi \simeq \normform{\phi} \simeq \simform{\phi}$, and that $\phi$ is divisible by the norm form of any of its subforms in this case. We will also need the following:

\begin{lemma} \label{LEManisotropicpartofformsdivisiblebyaquasiPfister} Let $\phi$ be a quasilinear quadratic form over $F$. If $\phi$ is divisible by an anisotropic quasi-Pfister form $\pi$ over $F$, then $\anispart{\phi}$ is also divisible by $\pi$.
\begin{proof} Since $\phi$ is divisible by $\pi$, $G(\phi)$ contains $G(\pi) = D(\pi)$. But $G(\phi) = G(\anispart{\phi})$ by definition, so Lemma \ref{LEMdivisibilitybyquasiPfister} tells us that $\anispart{\phi}$ is also divisible by $\phi$. \end{proof}
\end{lemma}

\subsection{The Knebusch Splitting Tower and Associated Invariants} \label{SUBSECKnebuschtower} Let $\phi$ be an anisotropic quasilinear quadratic form over $F$. We define a finite tower of fields  $F_0 \subset F_1 \subset \cdots \subset F_{\mathrm{lndeg}(\phi)}$ such that $\mathrm{lndeg}(\phi_{F_j}) = \mathrm{lndeg}(\phi) - j$ for all $j$ as follows: Set $F_0 := F$. Suppose now that $F_{j-1}$ has been defined for some $j \in [1,\mathrm{lndeg}(\phi)]$. Since $\mathrm{lndeg}(\phi_{F_{j-1}}) = \mathrm{lndeg}(\phi) - j+1 \geq 1$, $\phi_{F_{j-1}}$ is not split, and so we can set $F_j : = F_{j-1}(\anispart{(\phi_{F_{j-1}})})$. Applying Lemma \ref{LEMdropingofnormdegree} to $\phi_{F_{j-1}}$, we then have $\mathrm{lndeg}(\phi_{F_j}) = \mathrm{lndeg}(\phi_{F_{j-1}}) - 1 = \mathrm{lndeg}(\phi) - j$. For each $j \in [1,\mathrm{lndeg}(\phi)]$, we set $\phi_j : = \anispart{(\phi_{F_j})}$ and $\witti{j}{\phi} : = \witti{0}{\phi_{F_j}} - \witti{0}{\phi_{F_{j-1}}}$. By construction, the dimensions of the $\phi_j$ are strictly decreasing in $j$, and $\mathrm{lndeg}(\phi_j) = \mathrm{lndeg}(\phi) - j$. In particular, $\phi_{F_{\mathrm{lndeg}(\phi)}}$ is split. The integers $\witti{1}{\phi},\hdots,\witti{\mathrm{lndeg}(\phi)}{\phi}$ are the (relative) \emph{higher isotropy indices of $\phi$}. Note that if $\mydim{\phi} \geq 2$, then $\mydim{\phi_1} = \mydim{\phi} - \witti{1}{\phi}$ is the \emph{Izhboldin dimension} $\Izhdim{\phi}$ considered in \S \ref{SECintroduction}. If $\psi$ is a subform of $\phi$ with $\mydim{\psi} > \Izhdim{\phi}$, then $V_{\psi} \otimes_F F(\phi)$ must intersect the $\witti{1}{\phi}$-dimensional subspace of $\phi$-isotropic vectors in $V_{\phi} \otimes_F F(\phi)$, and so $\psi_{F(\phi)}$ is isotropic. The separation theorem (see \S \ref{SECintroduction}) then implies:

\begin{lemma} \label{LEMIzhbound} Let $\phi$ be an anisotropic quasilinear quadratic form of dimension $\geq 2$ over $F$, and let $n$ be the unique integer for which $2^n < \mydim{\phi} \leq 2^{n+1}$. Then $\Izhdim{\phi} \geq 2^n$.
\end{lemma}

\begin{remark} As mentioned in \S 1, we in fact know by \cite{Scully1} that if $u$ is the smallest non-negative integer with $\witti{1}{\phi} \leq 2^u$, then $\Izhdim{\phi} \equiv 0 \pmod{2^u}$. This will be reproved as part of the main result of this paper (see Corollary \ref{CORintro} above). 
\end{remark}

\begin{example} \label{EXsplittingofquasiPfister} If $\pi$ is an anisotropic $n$-fold quasi-Pfister form over $F$ for some positive integer $n$, then $\mathrm{lndeg}(\pi) = n$ and $\witti{j}{\pi} = 2^{n-j}$ for all integers $j \in [1,n]$ (this follows from Lemma \ref{LEMdivisibilitybyquasiPfister}, but see also \cite[Ex. 7.23]{Hoffmann2}). In particular, $\Izhdim{\pi} = 2^{n-1}$.
\end{example}

We will need the following basic observation:

\begin{lemma} \label{LEMisotropyindicesofformsdivisiblebyquasiPfisters} Let $r$ be a positive integer, let $\phi$ be an anisotropic quasilinear quadratic form of dimension $\geq 2$ over $F$, and let $y_r$ be the largest integer for which $\mydim{\phi} > 2^ry_r$. If $\phi$ is divisible by an anisotropic $r$-fold quasi-Pfister form over $F$, then:
\begin{enumerate} \item $\mathrm{lndeg}(\phi) \in [r,r+y_r]$;
\item $\Izhdim{\phi}$ is divisible by $2^r$;
\item $\witti{j}{\phi}$ is divisible by $2^r$ for all $j \in [1,\mathrm{lndeg}(\phi) - r]$.\end{enumerate}
\begin{proof} Let $\pi$ be an anisotropic $r$-fold quasi-Pfister form dividing $\phi$. Since $\phi$ is similar to a subform of $\normform{\phi}$, the same is true of $\pi$, and so $\mathrm{lndeg}(\phi) \geq \mathrm{lndeg}(\pi) = r$. Set $s: = \mathrm{lndeg}(\phi) - r$. If $j \leq s$, then we have $\mathrm{lndeg}(\phi_{j-1}) = \mathrm{lndeg}(\phi) - j +1 > \mathrm{lndeg}(\phi) -s = r$. By Lemma \ref{LEMdropingofnormdegree}, it follows that if $F = F_0 \subset F_1 \subset \cdots \subset F_{\mathrm{lndeg}(\phi)}$ is the Knebusch splitting tower of $\phi$, then $\pi$ remains anisotropic over $F_s$. Applying Lemma \ref{LEManisotropicpartofformsdivisiblebyaquasiPfister}, we then see that $\phi_j$ is divisible by $\pi$ for all $j \leq s$. Since $\witti{j}{\phi} = \mydim{\phi_j} - \mydim{\phi_{j-1}}$ for all $j$, this proves (2) and (3). At the same time, we have $\mathrm{lndeg}(\phi_s) = \mathrm{lndeg}(\phi) - s = r$, and so the preceding remarks imply that $\phi_s$ is similar to $\pi$. In particular, $\mydim{\pi} = 2^r$, so
$$ 2^rs \leq \sum_{j=1}^{s} \witti{j}{\phi} = \mydim{\phi} - \mydim{\phi_s} = \mydim{\phi} - 2^r \leq 2^ry_r, $$
and hence $\mathrm{lndeg}(\phi) = r+ s \leq r + y_r$, proving (1).
\end{proof} \end{lemma}

\subsection{Stable Birational Equivalence and Neighbours} \label{SUBSECstb} Let $\phi$ and $\psi$ be anisotropic quasilinear quadratic forms of dimension $\geq 2$ over $F$. If the quadrics $X_{\phi}$ and $X_{\psi}$ are stably birational as varieties over $F$, then we say that $\phi$ and $\psi$ are \emph{stably birationally equivalent}, and write $\phi \stb \psi$. The results of \cite{Totaro} allow us to characterize this relation as follows:

\begin{theorem}[Totaro] \label{THMstablebirationalequivalence} If $\phi$ and $\psi$ are anisotropic quasilinear quadratic forms of dimension $\geq 2$ over $F$, then the following are equivalent:
\begin{enumerate} \item $\phi \stb \psi$;
\item There exist $F$-places $F(\phi) \rightharpoonup F(\psi)$ and $F(\psi) \rightharpoonup F(\phi)$;
\item There exist rational maps $X_{\psi} \dashrightarrow X_{\phi}$ and $X_{\phi} \dashrightarrow X_{\psi}$;
\item Both $\phi_{F(\psi)}$ and $\psi_{F(\phi)}$ are isotropic. 
\item $\phi_{F(\psi)}$ is isotropic and $\Izhdim{\phi} = \Izhdim{\psi}$. \end{enumerate}
\begin{proof} $(1) \Rightarrow (2)$: Clear.

$ (2) \Rightarrow (3)$: Follows from the completeness of $X_{\phi}$ and $X_{\psi}$ (see, e.g., \cite[P. 408]{EKM}). 

$( 3) \Leftrightarrow (4)$: Clear.

$(4) \Leftrightarrow (5)$: Apply \cite[Theorem 5.2]{Totaro}.

$ (4) \Rightarrow (1)$: Apply \cite[Theorem 6.5]{Totaro}. \end{proof} \end{theorem}

As discussed in \S \ref{SECintroduction}, invariants of quasilinear quadratic forms that respect the relation $\stb$ are at the heart of this article. A new such invariant will be considered in \S \ref{SECDelta} below. As for previously studied examples, we have:

\begin{lemma} \label{LEMknownstablebirationalinvariants} Let $\phi$ and $\psi$ be anisotropic quasilinear quadratic forms of dimension $\geq 2$ over $F$. If $\phi \stb \psi$, then:
\begin{enumerate} \item $\normform{\phi} \simeq \normform{\psi}$;
\item $\mathrm{lndeg}(\phi) = \mathrm{lndeg}(\psi)$;
\item $\Izhdim{\phi} = \Izhdim{\psi}$;
\item $\witti{j}{\phi} = \witti{j}{\psi}$ for all $j \geq 2$. 
\end{enumerate}
\begin{proof} (2) is an immediate consequence of (1), which holds by Lemma \ref{LEMdropingofnormdegree}. (3) is part of Theorem \ref{THMstablebirationalequivalence}, so it only remains to show (4). By hypothesis, $F(\phi)$ and $F(\psi)$ are $F$-linearly embeddable into an extension $L$ of $F$ which is purely transcendental over both. By Lemma \ref{LEManisotropyoverseparable}, we then have that $\anispart{(\phi_L)} \simeq (\phi_1)_L$. At the same time, since $\phi_{F(\psi)}$ is isotropic, \cite[Thm. 6.6]{Scully2} (restated below as Theorem \ref{THMoldtheorem}) and Lemma \ref{LEMsubformoftensorproduct} tell us that $\anispart{(\phi_{F(\psi)})}$ contains a subform similar to $\psi_1$. But (3) says that $\mydim{\phi_1} = \mydim{\psi_1}$, so it follows that $(\phi_1)_L$ and $(\psi_1)_L$ are similar. Now $h(\phi_1) = h(\phi) - 1$ and $h(\psi_1) = h(\psi) - 1$, and since isotropy indices are insensitive to rational extension (Lemma \ref{LEManisotropyoverseparable}), the claim follows.
\end{proof}
\end{lemma}

Let $\phi$ and $\psi$ be anisotropic quasilinear quadratic forms of dimension $\geq 2$ over $F$. If $\psi$ is similar to a subform of $\phi$, and $\mydim{\psi} > \Izhdim{\phi}$, then we say that $\psi$ is a \emph{neighbour of $\phi$}. If $\psi$ is a neighbour of $\phi$, then $\phi_{F(\psi)}$ is evidently isotropic. At the same time, the remarks preceding Lemma \ref{LEMIzhbound} above show that $\psi_{F(\phi)}$ is also isotropic, so $\phi \stb \psi$ by Theorem \ref{THMstablebirationalequivalence}. In general, stable birational equivalence is more complicated, but we do have:

\begin{lemma} \label{LEMquasiPfisterneighbours} Let $\phi$ and $\psi$ be anisotropic quasilinear quadratic forms of dimension $\geq 2$ over $F$ with $\phi \stb \psi$. If $\phi$ is a quasi-Pfister form, then $\psi$ is a neighbour of $\phi$. 
\begin{proof} By Lemma \ref{LEMknownstablebirationalinvariants}, we have $\normform{\psi} \simeq \normform{\phi} = \phi$. By Lemma \ref{LEMuniversalpropertyofthenormform}, it follows that $\psi$ is similar to a subform of $\phi$. At the same time, $\mydim{\psi} > \Izhdim{\psi} = \Izhdim{\phi}$ by Theorem \ref{THMstablebirationalequivalence}, and so $\psi$ is a neighbour of $\phi$.
\end{proof}
\end{lemma}

Neighbours of anisotropic quasi-Pfister forms of dimension $\geq 2$ are called \emph{quasi-Pfister neighbours}. They may be further characterized as follows:

\begin{lemma}[see {\cite[Corollary 3.11]{Scully1}}] \label{LEMcharacterizationofquasiPfisterneighbours} Let $\phi$ be an anisotropic quasilinear quadratic form of dimension $\geq 2$ over $F$, and let $s$ be the unique integer for which $2^s < \mydim{\phi} \leq 2^{s+1}$. Then the following are equivalent:
\begin{enumerate} \item $\phi$ is a quasi-Pfister neighbour;
\item $\phi$ is a neighbour of $\normform{\phi}$;
\item $\mathrm{lndeg}(\phi) = s+1$;
\item $\phi_1$ is similar to a quasi-Pfister form. \end{enumerate}
\end{lemma}

Recall from \S \ref{SUBSECnormform} that if $\phi$ is an anisotropic quasilinear quadratic form of dimension $\geq 2$ with $2^s < \mydim{\phi} \leq 2^{s+1}$, then $s< \mathrm{lndeg}(\phi)<\mydim{\phi}$. Thus, among all anisotropic quasilinear quadratic forms of dimension in a fixed interval of the form $(2^s,2^{s+1}]$, the quasi-Pfister neighbours are those which have minimal norm degree (namely $s+1$). Per \S \ref{SECintroduction}, they should be understood as the forms with simplest stable birational type.

\section{Multiples of Quasi-Pfister Forms and the Invariant $\Delta$} \label{SECDelta}

In this section, we consider anisotropic quasilinear quadratic forms which are divisible by or ``close'' to being divisible by anisotropic quasi-Pfister forms. These considerations lead us to a new stable birational invariant of anisotropic quasilinear quadratic forms which we denote $\Delta$. As indicated in \S \ref{SECintroduction}, this invariant will play a crucial role in our main results. We continue to fix here a field $F$ of characteristic 2.

\subsection{Strong $\pi$-Neighbours} \label{SUBSECstrongpineighbours} To bring some order to the later discussion, we make here the following definition: 

\begin{definition} \label{DEFstrongpineighbour} Let $\eta$ be an anisotropic quasilinear quadratic form of dimension $\geq 2$ over $F$ which is divisible by a quasi-Pfister form $\pi$. By a \emph{strong $\pi$-neighbour of $\eta$}, we mean a neighbour of $\eta$ which has dimension strictly greater than $\mydim{\eta} - \mydim{\pi}$.
\end{definition}

Note that if $\phi$ is an anisotropic quasilinear quadratic form of dimension $\geq 2$ over $F$, then saying that $\phi$ is divisible by an anisotropic quasi-Pfister form $\pi$ is equivalent to saying that $\phi$ is a strong-$\pi$ neighbour of itself. Our definition therefore extends the concept of divisibility by quasi-Pfister forms to one of ``near-divisibility''. In the most extreme case, it recovers the notion of a quasi-Pfister neighbour: By Lemma \ref{LEMcharacterizationofquasiPfisterneighbours}, saying that $\phi$ is a quasi-Pfister neighbour is equivalent to saying that $\phi$ is a strong $\normform{\phi}$-neighbour of $\normform{\phi}$. We make the following additional observations:

\begin{remarks} \label{REMSstrongpineighbours} Let $\eta$ and $\pi$ be as in Definition \ref{DEFstrongpineighbour}, let $\phi$ be a neighbour of $\eta$, and let $n$ be the unique integer for which $2^n < \mydim{\phi} \leq 2^{n+1}$.
\begin{enumerate} \item Let $r$ be the foldness of $\pi$, and let $y_r$ be the largest integer for which $\mydim{\phi} > y_r2^r$. If $\phi$ is a strong $\pi$-neighbour of $\eta$, then we must have that $\mydim{\eta} = (y_r+1)2^r$ (since $\eta$ is divisible by $\pi$, $\mydim{\eta}$ is divisible by $2^r$). At the same time, in order for $\phi$ to be a neighbour of (and hence stably birationally equivalent to) $\eta$, Lemma \ref{LEMIzhbound} tells us that we must have $2^n < \mydim{\eta} \leq 2^{n+1}$. Thus, if $\phi$ is a strong $\pi$-neighbour of $\eta$, then $2^n < \mydim{\eta} \leq 2^{n+1}$ and the foldness of $\pi$ is at most $n+1$. 
\item If $\pi$ has foldness $n+1$, then $\phi$ is a strong $\pi$-neighbour of $\eta$ if and only if $\phi$ is a quasi-Pfister neighbour and $(\pi,\eta) = (\normform{\phi}, a\normform{\phi})$ for some $a \in F\setminus \lbrace 0 \rbrace$. Indeed, the conditions are clearly sufficient, and necessity follows from the preceding observation that $\mydim{\eta} \leq 2^{n+1}$ when $\phi$ is a strong $\pi$-neighbour of $\eta$ (recall that if $\phi$ is a neighbour of an anisotropic quasi-Pfister form $\pi$, then $\pi \simeq \normform{\phi}$).
\end{enumerate} \end{remarks}

More generally, we have the following:

\begin{lemma} \label{LEMambientformofstrongpineighbour} Let $\phi$ be an anisotropic quasilinear quadratic form of dimension $\geq 2$ over $F$, let $n$ be the unique integer for which $2^n < \mydim{\phi} \leq 2^{n+1}$, and let $\pi$ be an anisotropic quasi-Pfister form of foldness $\leq n+1$ over $F$. If $\eta$ is an anisotropic quasilinear quadratic form of dimension $\geq 2$ over $F$, then the following are equivalent:
\begin{enumerate} \item $\phi$ is a strong $\pi$-neighbour of $\eta$;
\item $\mydim{\anispart{( \pi \otimes \phi)}} < \mydim{\phi}+\mydim{\pi}$ and $\eta$ is similar to $\anispart{(\pi \otimes \phi)}$. \end{enumerate}
\begin{proof} Suppose first that (1) holds, i.e., that $\phi$ is a strong $\pi$-neighbour of $\eta$. The inequality $\mydim{\eta} < \mydim{\phi} + \mydim{\pi}$ then holds by definition, so to prove that (2) holds, it suffices to show that $\eta$ is similar to $\anispart{(\pi \otimes \phi)}$. Replacing $\eta$ with a similar form if needed, we can assume that $\phi \subset \eta$. Since $\eta$ is divisible by $\pi$, Lemma \ref{LEMdivisibilitybyquasiPfister} then gives that $\anispart{(\pi \otimes \phi)} \subset \anispart{(\pi \otimes \eta)} \simeq \eta$. To prove the desired assertion, we therefore have to check that $\anispart{(\pi \otimes \phi)}$ and $\eta$ have the same dimension. Note, however, that $\anispart{(\pi \otimes \phi)}$ is also divisible by $\pi$ by Lemma \ref{LEManisotropicpartofformsdivisiblebyaquasiPfister}, and so both dimensions are divisible by $\mydim{\pi}$. If they were different, we would then have that $\mydim{\anispart{(\pi \otimes \phi)}} \leq \mydim{\eta} - \mydim{\pi}$. But $\phi$ is a subform of $\anispart{(\pi \otimes \phi)}$ (Lemma \ref{LEMsubformoftensorproduct}), and has dimension $>\mydim{\eta} - \mydim{\pi}$ by hypothesis, so this is impossible. Thus, (2) holds.

Conversely, suppose that (2) holds. As noted above, $\anispart{(\pi \otimes \phi)}$ is divisible by $\pi$, so the same is true of $\eta$. Since $\mydim{\eta} < \mydim{\phi} + \mydim{\pi}$ by hypothesis, showing that (1) holds therefore amounts to showing that $\phi$ is a neighbour of $\eta$. Since $\phi$ is a subform of $\anispart{(\pi \otimes \phi)}$, it is similar to a subform of $\eta$, so what has to be checked is that $\mydim{\phi} > \Izhdim{\eta}$. If $\eta$ is not similar to $\pi$, then Lemma \ref{LEMisotropyindicesofformsdivisiblebyquasiPfisters} (2) tells us that $\witti{1}{\eta}$ is divisible by $\mydim{\pi}$, and so $\mydim{\phi} > \mydim{\eta} - \mydim{\phi} \geq \Izhdim{\phi}$, as desired. Otherwise, $\eta$ is similar to a quasi-Pfister form of dimension $\leq 2^{n+1}$, and so $\Izhdim{\eta} \leq 2^n$ by Lemma \ref{LEMcharacterizationofquasiPfisterneighbours}, again yielding the desired inequality. This proves the lemma. \end{proof} \end{lemma}

In particular, in the situation of Definition \ref{DEFstrongpineighbour}, the form $\eta$ is determined up to a scalar by $\pi$ and any of its strong $\pi$-neighbours. We therefore introduce the following terminology:

\begin{definition} Let $\phi$ be an anisotropic quasilinear quadratic form of dimension $\geq 2$ over $F$, and let $\pi$ be an anisotropic quasi-Pfister form over $F$. If $\phi$ is a strong $\pi$-neighbour of some anisotropic quasilinear quadratic form of dimension $\geq 2$ over $F$, then we shall say that $\phi$ is a \emph{strong $\pi$-neighbour}. 
\end{definition}

The preceding discussion then amounts to:

\begin{lemma} \label{LEMconcretedescriptionofpineighbour} Let $\phi$ be an anisotropic quasilinear quadratic form of dimension $\geq 2$ over $F$, and let $\pi$ be an anisotropic quasi-Pfister form over $F$. Then $\phi$ is a strong $\pi$-neighbour if and only if $\mydim{\anispart{(\pi \otimes \phi)}} < \mathrm{min}(\mydim{\phi} + \mydim{\pi},2\mydim{\phi})$. Moreover, in this case, $\phi$ is a strong $\pi$-neighbour of $\anispart{(\pi \otimes \phi)}$.
\end{lemma}

Before proceeding, we also need to note the following:

\begin{lemma} \label{LEMpineighbourisotropicoverpi} Let $\phi$ and $\pi$ be anisotropic quasilinear quadratic forms of dimension $\geq 2$ over $F$ with $\pi$ being quasi-Pfister. If $\phi$ is a strong $\pi$-neighbour, then $\phi_{F(\pi)}$ is isotropic and $\pi \subset \normform{\phi}$.
\begin{proof} Let $\pi' \subset \pi$ be such that $\pi = \langle 1 \rangle \perp \pi'$. Set $K: = F(V_{\pi})$, and let $\pi'(X) \in K$ be the generic value of $\pi'$. Then $\pfister{\pi'(X)} \subset \pi_K$, and so $\anispart{(\pfister{\pi'(X)} \otimes \phi_K)} \subset \anispart{(\pi_K \otimes \phi_K)}$. By Lemma \ref{LEMconcretedescriptionofpineighbour}, it follows that $\mydim{\anispart{(\pfister{\pi'(X)} \otimes \phi_K)}} < 2\mydim{\phi}$, and so $\phi_{F(\pi)}$ is isotropic by Lemma \ref{LEMisotropyoverfunctionfieldsofquadrics}. By Lemma \ref{LEMdropingofnormdegree}, we must then also have that $\pi \subset \normform{\phi}$.  \end{proof} \end{lemma}

\subsection{The Invariants $P_r$} \label{SUBSECPrInvariants} Motivated by the above discussion, we introduce a new family of invariants of anisotropic quasilinear quadratic forms as follows:

\begin{definition} Let $\phi$ be an anisotropic quasilinear quadratic form over $F$. For each non-negative integer $r$, we define $P_r(\phi)$ to be the set consisting of all isomorphism classes of anisotropic $r$-fold quasi-Pfister forms $\pi$ over $F$ for which $\mydim{\anispart{(\pi \otimes \phi)}} < \mydim{\phi} + 2^r$. 
\end{definition}

\begin{remark} \label{REMPrdimension} In the situation of the definition, let $y_r$ be the largest integer for which $\mydim{\phi} > y_r2^r$. If $\pi$ is an anisotropic $r$-fold quasi-Pfister form over $F$, then $\anispart{(\pi \otimes \phi)}$ is divisible by $\pi$ (Lemma \ref{LEManisotropicpartofformsdivisiblebyaquasiPfister}), and hence has dimension divisible by $2^r$. Since $\phi$ is a subform of $\anispart{(\pi \otimes \phi)}$ (Lemma \ref{LEMsubformoftensorproduct}), it follows that $[\pi] \in P_r(\phi)$ if and only if $\mydim{\anispart{(\pi \otimes \phi)}} < (y_r+2)2^r$, in which case we must then have $\mydim{\anispart{(\pi \otimes \phi)}} = (y_r+1)2^r$.
\end{remark}

By definition, $P_0$ is the constant invariant $\lbrace [\form{1}] \rbrace$. We are therefore interested in the case where $r \geq 1$. We have here the following basic observations:

\begin{lemma} \label{LEMbasicfactsonPr} Let $\phi$ be an anisotropic quasilinear quadratic form of dimension $\geq 2$ over $F$, let $n$ be the unique integer for which $2^n < \mydim{\phi} \leq 2^{n+1}$, and let $r$ be a positive integer.
\begin{enumerate} \item If $r \geq n+1$, and $\pi$ is an anisotropic $r$-fold quasi-Pfister form over $F$, then $[\pi] \in P_r(\phi)$ if and only if $\normform{\phi} \subset \pi$. In particular, $P_r(\phi) = \emptyset$ for all integers $r \in [n+1,\mathrm{lndeg}(\phi)-1]$ and $P_{\mathrm{lndeg}(\phi)}(\phi) = \lbrace [\normform{\phi}] \rbrace$. 
\item If $r \leq n+1$, and $\pi$ is an anisotropic $r$-fold quasi-Pfister form over $F$, then $[\pi] \in P_r(\phi)$ if and only if $\phi$ is a strong $\pi$-neighbour, in which case $\pi \subset \normform{\phi}$.
\item $P_{n+1}(\phi) \neq \emptyset$ if and only if $\phi$ is a quasi-Pfister neighbour, in which case $P_{n+1}(\phi) = \lbrace [\normform{\phi}] \rbrace$.
\item $P_n(\phi) \neq \emptyset$ if and only if $\phi$ is a quasi-Pfister neighbour, in which case $P_n(\phi) = \lbrace [\pi]\;|\; \pi \text{ is an $n$-fold quasi-Pfister subform of }\normform{\phi} \rbrace$. 
\end{enumerate}
\begin{proof} (1) Suppose that $r \geq n+1$, and let $\pi$ be an anisotropic $r$-fold quasi-Pfister form over $F$. By Remark \ref{REMPrdimension}, we then have that $[\pi] \in P_r(\phi)$ if and only if $\mydim{\anispart{(\pi \otimes \phi)}} = 2^r$. Since $\anispart{(\pi \otimes \phi)}$ is divisible by $\pi$ (Lemma \ref{LEManisotropicpartofformsdivisiblebyaquasiPfister}) this holds if and only if $\anispart{(\pi \otimes \phi)}$ is similar to $\pi$. Now $\phi$ is similar to a subform of $\anispart{(\pi \otimes \phi)}$, so if this holds, then we must have that $\normform{\phi} \subset \pi$ by Lemma \ref{LEMuniversalpropertyofthenormform}. Conversely, if $\normform{\phi} \subset \pi$, then $\phi$ is similar to a subform of $\pi = \simform{\pi}$, and so $\anispart{(\pi \otimes \phi)}$ is similar to $\pi$ by Lemma \ref{LEMdivisibilitybyquasiPfister}. 

(2) The first statement is Lemma \ref{LEMconcretedescriptionofpineighbour}, and the second then holds by Lemma \ref{LEMpineighbourisotropicoverpi}.

(3) This follows from (1) and Lemma \ref{LEMcharacterizationofquasiPfisterneighbours}. 

(4) Let $\pi$ be an $n$-fold quasi-Pfister subform of $\normform{\phi}$. By (2), proving (4) amounts to showing that $\phi$ is a strong $\pi$-neighbour if and only if $\phi$ is a quasi-Pfister neighbour. But if $\phi$ is a strong $\pi$-neighbour of an anisotropic form $\eta$ over $F$, then we must have that $\mydim{\eta} = 2^{n+1}$ by Remark \ref{REMSstrongpineighbours} (1). Since $\eta$ is divisible by $\pi$, it must then be similar to an $(n+1)$-fold quasi-Pfister form, and so $\phi$ is a quasi-Pfister neighbour. Conversely, if $\phi$ is a quasi-Pfister neighbour, then $\mydim{\normform{\phi}} = 2^{n+1}$ (Lemma \ref{LEMcharacterizationofquasiPfisterneighbours}). Since $\pi$ and $\phi$ are similar to subforms of $\normform{\phi}$, and since $\anispart{(\normform{\phi} \otimes \normform{\phi})} \simeq \normform{\phi}$ (Lemma \ref{LEMdivisibilitybyquasiPfister}), we then have that 
$$ \mydim{\anispart{(\pi \otimes \phi)}} \leq \mydim{\normform{\phi}} = 2^{n+1} < \mathrm{min}(\mydim{\phi}  + \mydim{\pi}, 2\mydim{\phi}), $$
and so $\phi$ is a strong $\pi$-neighbour by Lemma \ref{LEMconcretedescriptionofpineighbour}. 
\end{proof}
\end{lemma}

\begin{example} \label{EXPrinvariantsforPfister} Let $n$ be a non-negative integer. If $\phi$ is an anisotropic $n$-fold quasi-Pfister form over $F$, then $P_r(\phi) = \lbrace [\pi]\;|\; \pi \text{ is an $r$-fold quasi-Pfister subform of } \phi \rbrace$ for all integers $r \in [0,n]$. Indeed, $\phi$ is divisible by all its quasi-Pfister subforms, so the claim holds by part (2) of Lemma \ref{LEMbasicfactsonPr}. 
\end{example}

Next, we note that if $P_r(\phi)$ is non-empty for some anisotropic quasilinear quadratic form $\phi$ of dimension $\geq 2$ over $F$ and positive integer $r<\mathrm{lndeg}(\phi)$, then the basic stable birational invariants of $\phi$ are constrained in a non-trivial way:

\begin{lemma} \label{LEMdiscreteconstraintsfromnontrivialityofPr} Let $\phi$ be an anisotropic quasilinear quadratic form of dimension $\geq 2$ over $F$, let $r$ be a positive integer $<\mathrm{lndeg}(\phi)$, and let $y_r$ be the largest non-negative integer for which $\mydim{\phi} > y_r2^r$. If $P_r(\phi) \neq \emptyset$, then:

\begin{enumerate} \item $\mathrm{lndeg}(\phi) \in [r,r+y_r]$;
\item $\Izhdim{\phi}$ is divisible by $2^r$;
\item $\witti{j}{\phi}$ is divisible by $2^r$ for all $j \in [2,\mathrm{lndeg}(\phi)-r]$. \end{enumerate}

\begin{proof} Let $\pi$ be an anisotropic $r$-fold quasi-Pfister form over $F$ such that $[\pi] \in P_r(\phi)$, and set $\eta: = \anispart{(\pi \otimes \phi)}$. By Lemma \ref{LEMbasicfactsonPr} (2), $\phi$ is a strong $\pi$-neighbour of $\eta$. In particular, we have $\phi \stb \eta$.

(1) Since $\phi \stb \eta$, we have $\mathrm{lndeg}(\phi) = \mathrm{lndeg}(\eta)$ by Lemma \ref{LEMknownstablebirationalinvariants} (2). But since $\eta$ is divisible by $\pi$, we have $\mathrm{lndeg}(\eta) \in [r,r+y_r]$ by Lemma \ref{LEMisotropyindicesofformsdivisiblebyquasiPfisters} (1). 

(2) Since $\phi \stb \eta$, we have $\Izhdim{\phi} = \Izhdim{\eta}$ by Lemma \ref{LEMknownstablebirationalinvariants} (3). Now $\eta$ is divisible by $\pi$, but not similar to $\pi$ (since $r< \mathrm{lndeg}(\phi)$, we have $\mydim{\pi} < \mydim{\phi}$ by Lemma \ref{LEMbasicfactsonPr}), and so $\Izhdim{\eta}$ is divisible by $2^r$ by Lemma \ref{LEMisotropyindicesofformsdivisiblebyquasiPfisters} (2). 

(3) Again, since $\phi \stb \eta$, we have $\witti{j}{\phi} = \witti{j}{\eta}$ for all $j \geq 2$ by Lemma \ref{LEMknownstablebirationalinvariants} (4). Applying Lemma \ref{LEMisotropyindicesofformsdivisiblebyquasiPfisters} (3) to $\eta$ then gives the result. 
\end{proof} \end{lemma}  

In particular:

\begin{corollary} \label{CORPremptyforgenericforms} Let $\phi$ be an anisotropic quasilinear quadratic form of dimension $\geq 2$ over $F$. If $P_r(\phi) \neq \emptyset$ for some positive integer $r< \mathrm{lndeg}(\phi)$, then $\mathrm{lndeg}(\phi) \leq \frac{\mydim{\phi}+1}{2}$. 
\begin{proof} Since $r< \mathrm{lndeg}(\phi)$, Lemma \ref{LEMbasicfactsonPr} (1) tells us that $r < \mathrm{log}_2(\mydim{\phi})$. For each $k\leq r$, let $y_k$ be the largest integer for which $\mydim{\phi} > y_k2^k$. Then $y_r < y_{r-1} < \cdots < y_1$, and so $r + y_r \leq 1 + y_1 \leq \frac{\mydim{\phi}+1}{2}$. The claim then follows from the first part of Lemma \ref{LEMdiscreteconstraintsfromnontrivialityofPr}.
\end{proof} \end{corollary}

\begin{example} \label{EXPrforgenericforms} If $X_1,\hdots,X_n$ are $n\geq 4$ indeterminates, and $\phi$ is the $n$-dimensional (anisotropic) quasilinear quadratic form $\form{X_1,\hdots,X_n}$ over $F(X_1,\hdots,X_n)$, then $\mathrm{lndeg}(\phi) = n-1 > \frac{n+1}{2}$. By Corollary \ref{CORPremptyforgenericforms}, it follows that $P_r(\phi) = \emptyset$ for all positive integers $r \leq n-2$. 
\end{example}

Next, we observe the following:

\begin{proposition} \label{PROPmultiplyingPrs} Let $\phi$ be an anisotropic quasilinear quadratic form of dimension $\geq 2$ over $F$, let $r \leq s$ be non-negative integers, and let $y_r$ and $y_s$ be the largest integers for which $\mydim{\phi} > y_r 2^r$ and $\mydim{\phi} > y_s 2^s$, respectively. Suppose that $\sigma$ and $\pi$ are anisotropic quasi-Pfister forms over $F$ such that $[\sigma] \in P_{r}(\phi)$ and $[\pi] \in P_{s}(\phi)$. Set $\eta: = \anispart{(\pi \otimes \sigma)}$. 
\begin{enumerate} \item $\eta$ is a $t$-fold quasi-Pfister form for some $t \geq s$.
\item For all integers $i \in [s,t]$, there is an $i$-fold quasi-Pfister subform $\eta \subset \eta$ such that $[\eta] \in P_i(\phi)$. In particular, $P_i(\phi) \neq \emptyset$ for all $i \in [s+1,t]$. 
\item If either of the following holds, then $\sigma \subset \pi$:
\begin{itemize} \item $P_{s+1}(\phi) = \emptyset$;
\item $y_s$ is even and $\mydim{\phi} > (y_{r}+1)2^{r-1} + y_s2^{s-1}$. \end{itemize}
\item If either of the two conditions in $(3)$ holds, or if $y_s$ is odd, then $\anispart{(\sigma \otimes \phi)} \subset \anispart{(\pi \otimes \phi)}$.
\end{enumerate}
\begin{proof} We can assume that $r>0$. Since $D(\eta) = D(\pi \otimes \sigma')$ is a subfield of $F$, $\eta$ is a $t$-fold quasi-Pfister form for some positive integer $t$. Since $\pi \subset \sigma$ (Lemma \ref{LEMsubformoftensorproduct}), we have $t \geq s$, and so (1) holds. To prove (2), (3) and (4), we need the following:

\begin{claim} \label{LEMmultiplyingbybinary} Suppose, in the above situation, that $a \in D(\sigma) \setminus D(\pi)$. Then $[\pfister{a} \otimes \pi] \in P_{s+1}(\phi)$, and $\mydim{\anispart{(\pfister{a} \otimes \pi \otimes \phi)}} \leq (y_{r}+1)2^{r} + 2\big((y_s+1)2^s - \mydim{\phi}\big)$. 
\begin{proof} Set $\tau: = \pfister{a} \otimes \pi$. Since $a \notin D(\pi)$, $\tau$ is an anisotropic $(r+1)$-fold quasi-Pfister form. Now, since $\phi \subset \anispart{(\pi \otimes \phi)}$ (Lemma \ref{LEMsubformoftensorproduct}), we can write $\anispart{(\pi \otimes \phi)} \simeq \phi \perp \psi$ for some form $\psi$ of dimension $(y_s+1)2^s - \mydim{\phi}<2^s$ over $F$. By Lemma \ref{LEMsubformtheorem}, we then have that
$$ \anispart{(\tau \otimes \phi)} \simeq \anispart{(\pfister{a} \otimes \phi \perp \pfister{a} \otimes \psi)} \subset \anispart{(\pfister{a} \otimes \phi)} \perp \anispart{(\pfister{a} \otimes \psi)} \subset \anispart{(\sigma \otimes \phi)} \perp \anispart{(\pfister{a} \otimes \psi)}, $$
and so $\mydim{\anispart{(\tau \otimes \phi)}} \leq (y_{r} + 1)2^{r} + 2\mydim{\psi}$. This proves the inequality in the statement. At the same time, since $r \leq s$, and since $\mydim{\psi} < 2^s$, we get that $\mydim{\anispart{(\tau \otimes \phi)}} \leq (y_s+2)2^s$. By Remark \ref{REMPrdimension}, this implies that $[\tau] \in P_{s+1}(\phi)$, and so the claim is proved.
\end{proof}
\end{claim}

We return to the proof of the proposition:

(2) If $\sigma \subset \pi$, then it follows from Lemma \ref{LEMdivisibilitybyquasiPfister} that $\eta \simeq \pi$, and the statement holds trivially. Suppose now that $\sigma \not \subset \pi$. By Lemma \ref{LEMsubformtheorem}, there then exists an element $a \in D(\sigma) \setminus D(\pi)$. Set $\eta_{s+1} = \pfister{a} \otimes \pi$. By Claim \ref{LEMmultiplyingbybinary}, $[\sigma_{s+1}] \in P_{s+1}(\phi)$. Observe now that since $\pfister{a} \subset \sigma$, we have $\anispart{(\pfister{a} \otimes \sigma)} \simeq \sigma$ (again use Lemma \ref{LEMdivisibilitybyquasiPfister}), and so $\anispart{(\eta_{s+1} \otimes \sigma)} \simeq \anispart{\pi \otimes \pfister{a} \otimes \sigma)} \simeq \anispart{(\pi \otimes \sigma)} = \eta$. Thus, if $t \geq s+2$, then we can repeat the preceding argument with $\sigma_{s+1}$ replacing $\pi$ to find an $(s+2)$-fold quasi-Pfister subform $\eta_{s+2} \subset \eta$ such that $[\eta_{s+2}] \in P_{s+2}(\phi)$. Continuing in this way, we obtain the desired result.

(3) If $P_{s+1}(\phi) = \emptyset$, then it follows from (2) that $t =s$, and so $\eta \simeq \pi$. By Lemma \ref{LEMdivisibilitybyquasiPfister}, we then have that $\sigma \subset \pi$. Suppose now that $\mydim{\phi} > (y_{r} +1)2^{r-1} +y_s2^{s-1}$. Let $a \in D(\sigma)$. If $a \notin D(\pi)$, then Claim \ref{LEMmultiplyingbybinary} tells us that $[\pfister{a} \otimes \pi] \in P_{s+1}(\phi)$, and so $\anispart{(\pfister{a} \otimes \pi \otimes \phi)}$ has dimension divisible by $2^{s+1}$. On the other hand, the second statement of the claim, together with our assumed lower bound for $\mydim{\phi}$, gives that
$$ \mydim{\anispart{(\pfister{a} \otimes \pi \otimes \phi)}} \leq (y_{r} + 1)2^{r} + 2\big((y_s+1)2^s - \mydim{\phi}\big) < (y_s + 2)2^s. $$
But then $(y_s+1)2^s$ must be divisible by $2^{s+1}$, i.e., $y_s$ is odd.

(4) If $\sigma \subset \pi$, there is nothing to show. Suppose now that $\sigma \not \subset \pi$, and let $a \in D(\sigma) \setminus D(\pi)$. By Claim \ref{LEMmultiplyingbybinary}, we then have that $[\pfister{a} \otimes \pi] \in P_{s+1}(\phi)$. If $y_s$ is odd, it then follows that $\mydim{\anispart{(\pfister{a} \otimes \pi \otimes \phi)}} = (y_s+1)2^s$. But $\anispart{(\pfister{a} \otimes \pi \otimes \phi)}$ contains $\anispart{(\pi \otimes \phi)}$ as a subform, and since the latter also has dimension $(y_s+1)2^s$, we then have that $\anispart{(\pfister{a} \otimes \pi \otimes \phi)} \simeq \anispart{( \pi \otimes \phi)}$. By Lemma \ref{LEMsubformoftensorproduct}, it follows that $\anispart{(\pfister{a} \otimes \phi)} \subset \anispart{(\pi \otimes \phi)}$. Since this obviously also holds when $a \in D(\pi)$, it holds for all $a \in D(\sigma)$.  Since $D(\sigma \otimes \phi)$ is generated as an $F^2$-vector space by products $ab$ with $a \in D(\sigma)$ and $b \in D(\phi)$ (see the remarks preceding Lemma \ref{LEMsubformoftensorproduct}), Lemma \ref{LEMsubformtheorem} then gives that $\anispart{(\sigma \otimes \phi)} \subset \anispart{(\pi \otimes \phi)}$.
 \end{proof}
\end{proposition}

In particular, we obtain the following:

\begin{corollary} \label{CORquasiPfisterinvariant} Let $\phi$ be an anisotropic quasilinear quadratic form of dimension $\geq 2$ over $F$, and let $n$ be the unique integer for which $2^n < \mydim{\phi} \leq 2^{n+1}$. Then there exists a positive integer $m \in [0,n+1]\setminus \lbrace n \rbrace$ and an anisotropic $m$-fold quasi-Pfister form $\pi$ over $F$ such that:
\begin{enumerate} \item $P_{r}(\phi) = \emptyset$ for all $r \in [m+1, \mathrm{lndeg}(\phi)-1]$;
\item $P_m(\phi) = \lbrace [\pi] \rbrace$;
\item If $r \leq m$, and $\sigma$ is an anisotropic quasi-Pfister form over $F$ such that $[\sigma] \in P_{r}(\phi)$, then $\sigma \subset \pi$;
\item $m = n+1$ if and only if $\phi$ is a quasi-Pfister neighbour, in which case $\pi \simeq \normform{\phi}$. \end{enumerate}
\begin{proof} Take $m$ to be the largest non-negative integer $\leq n+1$ for which $P_m(\phi) \neq \emptyset$. If $m = n+1$, then $\phi$ is a quasi-Pfister neighbour and we can take $\pi = \normform{\phi}$ (Lemma \ref{LEMbasicfactsonPr} (3) and Lemma \ref{LEMpineighbourisotropicoverpi}). If not, then $P_{m+1}(\phi) = \emptyset$, and the existence of the desired form $\pi$ is then an immediate consequence of Proposition \ref{PROPmultiplyingPrs} (3). 
\end{proof}
\end{corollary}

\begin{remark} In the case where $\phi$ is not a quasi-Pfister neighbour, the quasi-Pfister form $\pi$ of Corollary \ref{CORquasiPfisterinvariant} seems to be a new invariant of $\phi$.
\end{remark}

In the next subsection, we will use the preceding discussion to define a new stable birational invariant of an arbitrary quasilinear quadratic form of dimension $\geq 2$. To prove its stable birational invariance, however, we will need to observe that, for any non-negative integer $r$, non-triviality of the invariant $P_r$ on a particular form can be detected under purely transcendental base change. This is a consequence of the following lemma:

\begin{lemma} \label{LEMdescendingoverrationalextension} Let $\sigma$ be an anisotropic quasilinear quadratic form of dimension $n \geq 1$ over $F(X)$, where $X$ is an indeterminate. If $D(\sigma) \subseteq F(X^2)$, then there exist polynomials $f_1,\hdots,f_n \in F[X]$ such that $\sigma \simeq \form{f_1(X^2),\hdots,f_n(X^2)}$ and the quasilinear quadratic form $\form{f_1(0),\hdots,f_n(0)}$ is anisotropic over $F$. 
\begin{proof} Since $D(\sigma) \subseteq F(X^2)$, there certainly exist $n$-tuples of polynomials $(f_1,\hdots,f_n) \in F[X]^n$ such that $\sigma \simeq \form{f_1(X^2),\hdots,f_n(X^2)}$. Let us choose one for which the integer $\sum_{i=1}^n \mathrm{deg}(f_i)$ is minimal. Reordering the $f_i$ if necessary, we can assume that $\mathrm{deg}(f_1) \leq \cdots \leq \mathrm{deg}(f_n)$. We claim that $\form{f_1(0),\hdots,f_n(0)}$ is anisotropic. Suppose, to the contrary, that this is not the case. There then exists an integer $r \in [1,n]$ such that $f_r(0)$ lies in the $F^2$-linear span of $f_1(0),\hdots, f_{r-1}(0)$, say $f_r(0) = \sum_{i=1}^{r-1} \lambda_i^2 f_i(0)$ with $\lambda_1,\hdots,\lambda_{r-1} \in F$. Set $g_r : = f_r + \sum_{i=1}^{r-1} \lambda_i^2 f_i$, and set $g_i : = f_i$ for all $i \neq r$. By construction, we then have that:
\begin{itemize} \item $\form{g_1(X^2),\hdots,g_n(X^2)} \simeq \form{f_1(X^2),\hdots,f_n(X^2)} \simeq \sigma$ (apply Lemma \ref{LEMsubformtheorem});
\item $\sum_{i=1}^n \mathrm{deg}(g_i) \leq \sum_{i=1}^n \mathrm{deg}(f_1)$ \big(since $\mathrm{deg}(f_1) \leq \cdots \mathrm{deg}(f_r)$\big);
\item $g_r(0) = 0$. \end{itemize}
Replacing $(f_1,\hdots,f_n)$ with $(g_1,\hdots,g_n)$, we can therefore assume that $f_r(0) = 0$, i.e., that $f_r = Xf_r'$ for some $f_r' \in F[X]$ with $\mathrm{deg}(f_r') = \mathrm{deg}(f_r)-1$. Then $\form{f_r(X^2)} \simeq \form{X^2f_r'(X^2)} \simeq \form{f_r'(X^2)}$, and so $\sigma \simeq \form{f_1(X^2),\hdots,f_{r-1}(X^2),f_r'(X^2),f_{r+1}(X^2),\hdots,f_n(X^2)}$. But $\mathrm{deg}(f_r') + \sum_{i \neq r} \mathrm{deg}(f_i) = \big(\sum_{i=1}^n \mathrm{deg}(f_i)\big) - 1$, so this contradicts our choice of $(f_1,\hdots,f_n)$. 
\end{proof} \end{lemma}

\begin{proposition} \label{PROPtrivialityofPrinvariantunderrationalextensions} Let $\phi$ be an anisotropic quasilinear quadratic form of dimension $\geq 2$ over $F$, and let $r$ be a non-negative integer. If $L$ is a purely transcendental field extension of $F$, then $P_r(\phi) \neq \emptyset$ if and only if $P_r(\phi_L) \neq \emptyset$.
\begin{proof} Let $\pi$ be an anisotropic $r$-fold quasi-Pfister form over $F$. Since $L/F$ is separable, $\pi$ remains anisotropic over $L$ (Lemma \ref{LEManisotropyoverseparable}). It is then immediate that $[\pi_L] \in P_r(\phi_L)$ whenever $[\pi] \in P_r(\phi)$. Thus, $P_r(\phi_L)$ is non-empty if $P_r(\phi)$ is. For the reverse implication, we may assume that $L = F(X)$ for a single indeterminate $X$. Under this assumption, suppose that $\pi$ is an anisotropic $r$-fold quasi-Pfister form over $L$ such that $[\pi] \in P_r(\phi_L)$. By Lemma \ref{LEMpineighbourisotropicoverpi}, $\pi$ is a subform of $\normform{(\phi_{L})}$. Since $L/F$ is separable, however, we have $\normform{(\phi_L)} \simeq (\normform{\phi})_{L}$ (see the remarks directly preceding Lemma \ref{LEMdropingofnormdegree}), and so the elements of $D(\pi)$ are $L^2$-linear combinations of elements of $F$. Since $L^2 = F^2(X^2)$, it follows that $D(\pi) \subseteq F(X^2)$. By Lemma \ref{LEMdescendingoverrationalextension}, we then have that $\pi \simeq \form{f_1(X^2),\hdots,f_{2^r}(X^2)}$ for some $f_1,\hdots,f_{2^r} \in F[X]$ such that the quasilinear quadratic form $\tau : = \form{f_1(0),\hdots,f_{2^r}(0)}$ is anisotropic over $F$. We claim that $\tau$ is a quasi-Pfister form. Since $\tau$ is anisotropic, this amounts to showing that $N(\tau) = D(\tau)$, i.e., that $D(\tau)$ is closed under multiplication. By additivity, it suffices to show that $f_i(0)f_j(0) \in D(\tau)$ for all $i,j \in [1,2^r]$. But since $\pi$ is a quasi-Pfister form, we have $f_i(X^2)f_j(X^2) \in D(\pi)$. By the Cassels-Pfister theorem (which is valid for quasilinear quadratic forms, see \cite[Thm. 17.3]{EKM})  it follows that $f_i(X^2)f_j(X^2) = \sum_{k=1}^{2^r} g_k(X)^2f_k(X^2)$ for some $g_1,\hdots,g_{2^r} \in F[X]$. Evaluating at $0$, we get $f_i(0)f_j(0) = \sum_{k=1}^{2^r} g_k(0)^2 f_k(0) \in D(\tau)$, and so our claim holds. We now claim that $[\tau] \in P_r(\phi)$. Let $y_r$ be the largest integer for which $\mydim{\phi} >y_r2^r$. By Remark \ref{REMPrdimension}, proving our claim amounts to showing that $\mydim{\anispart{(\tau \otimes \phi)}} < (y_r+2)2^r$. But since $\pi \in P_r(\phi_{L})$, we have $\mydim{\anispart{(\pi \otimes \phi_{L})}} < 2^r(y_r+2)$. Since $\pi$ is isomorphic to the form $\form{f_1(X^2),\hdots,f_{2^r}(X^2)}$, and since the latter is defined over $F[X]$, the claim then follows by applying Lemma \ref{LEMspecializationlemma} to $\form{f_1(X^2),\hdots,f_{2^r}(X^2)} \otimes \phi$ viewed as form over the discrete valuation ring $F[X]_{(X)}$.  \end{proof}
\end{proposition}

Although Proposition \ref{PROPtrivialityofPrinvariantunderrationalextensions} is all we shall need in the sequel, one might wonder whether its conclusion remains valid if we allow $L$ to be an arbitrary separable extension of $F$. While we are unable to address this in general, we can at least show that the answer is positive in a number of cases. We first note:

\begin{lemma} \label{LEMdescentofGaloisinvariantforms} Let $L$ be an algebraic field extension of $F$, and let $\psi$ be an anisotropic quasilinear quadratic form over $L$. If $L/F$ is Galois, and $D(\psi)$ is stable under the canonical action of $\mathrm{Gal}(L/F)$ on $L$, then $\psi \simeq \tau_L$ for some quasilinear quadratic form $\tau$ over $F$, unique up to isomorphism. Moreover, if $\psi$ is a quasi-Pfister form, then so is $\tau$.
\begin{proof} The extension $L^2/F^2$ is also Galois, and the restriction map $\mathrm{Aut}(L) \rightarrow \mathrm{Aut}(L^2)$ identifies its Galois group with that of $L/F$. The action of $\mathrm{Gal}(L/F)$ on $D(\psi)$ then determines an action of $\mathrm{Gal}(L^2/F^2)$ which is evidently $L^2$-semilinear. Thus, if $U$ is locus of $\mathrm{Gal}(L/F)$-fixed points in $D(\psi)$, then $U$ is finite-dimensional as an $F^2$-vector space, and the multiplication map $D \otimes_{F^2} L^2 \rightarrow U$ is an $L^2$-linear isomorphism. Up to isomorphism, there is then a unique quasilinear quadratic form $\tau$ over $F$ such that $D(\tau) = U$ and $\tau_L \cong \psi$ (see Lemma \ref{LEMsubformtheorem}). Finally, if $\psi$ is quasi-Pfister, then $D(\psi)$ is a subfield of $L$, and so the fixed-point locus $U$ is a subfield of $F$, so that $\tau$ is also quasi-Pfister. \end{proof} \end{lemma}

We now have:

\begin{proposition}  \label{PROPstabilityofPrunderseparable} Let $\phi$ be an anisotropic quasilinear quadratic form of dimension $\geq 2$ over $F$, let $r$ be a non-negative integer, and let $y_r$ be the largest integer for which $\mydim{\phi} > y_r2^r$. Let $L$ be a separable extension of $F$. Suppose that any of the following hold:
\begin{enumerate} \item $P_{r+1}(\phi_L) = \emptyset$;
\item $\mydim{\phi} > 2^r(y_r+1) - 2^{r-1}$;
\item $y_r$ is odd;
\item $r \in \lbrace 0,1,n,m \rbrace$, where $n$ is the unique integer for which $2^n < \mydim{\phi} \leq 2^{n+1}$, and $m = \mathrm{max}\lbrace r\;|\; r \leq n+1 \text{ and } P_r(\phi) \neq \emptyset \rbrace$. 
\end{enumerate}
Then $P_r(\phi) \neq \emptyset$ if and only if $P_r(\phi_L) \neq \emptyset$. 
\begin{proof} The validity of the implication $P_r(\phi) \neq \emptyset \Rightarrow P_r(\phi_L) \neq \emptyset$ has already been noted in the proof of Proposition \ref{PROPtrivialityofPrinvariantunderrationalextensions}. For the converse, we can assume that $L/F$ is finitely generated. Proposition \ref{PROPtrivialityofPrinvariantunderrationalextensions} then allows us to reduce to the case where $L/F$ is finite and separable. Since the `only if' implication holds, we can in fact assume that $L/F$ is finite and Galois. Under this assumption, let $\pi$ be an anisotropic $r$-fold quasi-Pfister form over $F$ such that $[\pi] \in P_r(\phi)$, and set $\eta: = \anispart{(\pi \otimes \phi)}$. For each $\sigma \in \mathrm{Gal}(L/F)$, and each quasilinear quadratic form $\mu$ over $L$, let us write $\sigma_*(\mu)$ for the quasilinear quadratic form $\sigma \circ \mu$ over $L$. Note that $D\big(\sigma_*(\mu)\big) = \sigma\big(D(\mu)\big)$, so $\sigma_*$ commutes with the formation of anisotropic parts and sends quasi-Pfister forms to quasi-Pfister forms. It is also clear that $\sigma_*$ commutes with the formation of tensor products, so
$$ \sigma_*(\eta) \simeq \sigma_*\big(\anispart{(\pi \otimes \phi_L)}\big) \simeq \anispart{\big(\sigma_*(\pi) \otimes \sigma_*(\phi_L)\big)} = \anispart{(\sigma_*(\pi) \otimes \phi_L)}. $$
Since $[\pi] \in P_r(\phi)$, we then have that $\sigma_*(\pi) \in P_r(\phi)$ also. Consider now the anisotropic quasi-Pfister form $\psi: = \anispart{\big(\otimes_{\sigma} \sigma_*(\pi)\big)}$, where the tensor product is taken over all $\sigma \in \mathrm{Gal}(L/F)$. By the remarks preceding Lemma \ref{LEMsubformoftensorproduct}, $D(\psi)$ is the image of the $L^2$-linear multiplication map $\bigotimes_{\sigma} \sigma\big(D(\pi)\big) \rightarrow L$, and is thus a $\mathrm{Gal}(L/F)$-stable subfield of $L$. By Lemma \ref{LEMdescentofGaloisinvariantforms}, it follows that there exists an anisotropic quasi-Pfister form $\tau$ over $F$ such that $\tau_L \simeq \psi$. Set $s := \mathrm{lndeg}(\tau)$. Since $L/F$ is separable, we have $s = \mathrm{lndeg}(\psi)$ (Lemma \ref{LEManisotropyoverseparable}). Since $\sigma_*(\pi) \in P_r(\phi)$, Proposition \ref{PROPmultiplyingPrs} then tells us that $s \geq r$, $\psi \in P_s(\phi_L)$, and $P_i(\phi_L) \neq \emptyset$ for all $r \leq i \leq s$. Again, using Lemma \ref{LEManisotropyoverseparable}, we then have that
$$ \mydim{\anispart{(\tau \otimes \phi)}} = \mydim{\anispart{(\tau_L \otimes \phi_L)}} = \mydim{\anispart{(\psi \otimes \phi_L)}}, $$
and so $[\tau] \in P_s(\phi)$. We now consider the four situations in the statement.

(1) If $P_{r+1}(\phi_L) = \emptyset$, then we must have that $s = r$, and so $[\tau] \in P_r(\phi)$. 

(2) Suppose that $\mydim{\phi} > 2^r(y_r+1) -2^{r-1}$. If $y_r$ is even, then then Proposition \ref{PROPmultiplyingPrs} (3) tells us that $\sigma_*(\pi) \subset \pi$ for all $\sigma \in \mathrm{Gal}(L/F)$, and so $\eta \simeq \pi$ by Lemma \ref{LEMdivisibilitybyquasiPfister}. Thus, $s = r$, and we again have that $[\tau] \in P_r(\phi)$. Thus, to complete case (2), it suffices to cover case (3), which we now do.

(3) Let $\sigma \in \mathrm{Gal}(L/F)$. If $y_r$ is odd, then Proposition \ref{PROPmultiplyingPrs} (4) tells us that $\anispart{(\sigma_*(\pi) \otimes \phi_L)} \simeq \eta$. By Lemma \ref{LEManisotropicpartofformsdivisiblebyaquasiPfister}, it follows that $\eta$ is divisible by $\sigma_*(\pi)$. By Lemma \ref{LEMdivisibilitybyquasiPfister}, we then have that $D(\psi) \subseteq G(\eta)$, and hence that $\eta$ is divisible by $\psi \simeq \tau_L$. Let $\tau'$ be an $r$-fold quasi-Pfister subform of $\tau$. By Lemmas \ref{LEManisotropyoverseparable} and \ref{LEMdivisibilitybyquasiPfister}, we then have that 
$$ \mydim{\anispart{(\tau' \otimes \phi)}} = \mydim{\anispart{(\tau'_L \otimes \phi_L)}} \leq \mydim{\anispart{(\psi \otimes \eta)}} = \mydim{\eta}, $$
and so $[\tau'] \in P_r(\phi)$. 

(4) The assertion is clear when $r = 0$, and when $r = n$, it is immediate consequence of Lemma \ref{LEMbasicfactsonPr} (4). When $r = m$, it follows from (1) (if $m<n+1$) and Lemma \ref{LEMbasicfactsonPr} (3) (if $m = n+1$). It remains to consider the case where $r = 1$. By (2), we can assume here that $\mydim{\phi}$ is odd, so that $\mydim{\eta} = \mydim{\phi} + 1$. We claim that there exists a $1$-fold quasi-Pfister subform $\tau' \subset \tau$ such that $[\tau'] \in P_1(\phi)$. We argue by induction on $\mydim{\phi}$. If $\mydim{\phi} = 1$, there is nothing to show, so assume that $\mydim{\phi} \geq 3$. To do the induction step, we use a result to be stated and proved in the next section. Specifically, since $\anispart{(\tau \otimes \phi)}$ is divisible by $\tau$ (Lemma \ref{LEMdivisibilitybyquasiPfister}), Corollary \ref{CORinductivepropositionweak} and Remark \ref{REMinductivepropositiontensorbyPfister} below tell us that there exists an anisotropic quasilinear quadratic form $\hat{\phi}$ over $F$ such that 
\begin{itemize} \item[(i)] $\phi \perp \hat{\phi} \simeq \anispart{(\tau \otimes \phi)}$, and 
\item[(ii)] $\mydim{\hat{\phi}} - 2\witti{0}{\hat{\phi}_{M(\nu)}} = \mydim{\phi} - 2\witti{0}{\phi_{M(\nu)}}$ for every separable extension $M/F$ and subform $\nu \subset \tau_M$ of dimension $\geq 2$. \end{itemize} 
Note that when $\nu$ is a 1-fold quasi-Pfister form, the equality in (ii) may be rewritten as $\mydim{\anispart{(\nu \otimes \hat{\phi}_M)}} - \mydim{\hat{\phi}} = \mydim{\anispart{(\nu \otimes \phi_M)}} - \mydim{\phi}$ (Lemma \ref{LEMisotropyoverfunctionfieldsofquadrics}).  In particular, taking $(M,\nu) = (L,\pi)$, we get that $\mydim{\anispart{(\pi \otimes \hat{\phi}_L)}} - \mydim{\hat{\phi}_L} = 1$, and so $[\pi] \in P_1(\hat{\phi}_L)$. Now, since $[\tau] \in P_s(\phi)$, (i) implies that $\mydim{\hat{\phi}} < \mydim{\phi}$ (see Corollary \ref{CORquasiPfisterinvariant}). By the induction hypothesis, it follows that there exists a 1-fold quasi-Pfister subform $\tau' \subset \tau$ such that $[\tau] \in P_1(\hat{\phi})$. Applying the above equality with $(M,\nu) = (F, \tau')$, we then get that $\mydim{\anispart{(\tau' \otimes \phi)}} - \mydim{\phi} = 1$, and so $[\tau'] \in P_1(\phi)$, completing the proof.
\end{proof}
\end{proposition}

\subsection{The Invariant $\Delta$} \label{SUBSECDeltaInvariant} We now come to the main point of this section. Recall that if $\phi$ is a quasilinear quadratic form of dimension $\geq 2$ over $F$, then we write $\phi_1$ for the anisotropic part of $\phi$ over $F(\phi)$. The dimension of $\phi_1$ is equal to the Izhboldin dimension $\Izhdim{\phi}$. We make the following definition:

\begin{definition} For any quasilinear quadratic form $\phi$ of dimension $\geq 2$ over $F$, we set $\Delta(\phi): = \lbrace r\;|\; r<\mathrm{lndeg}(\phi) \text{ and } P_r(\phi_1) \neq \emptyset \rbrace$. 
\end{definition}

\begin{remark} \label{REMDeltasimilaritytype} In the above situation, the sets $P_r(\phi_1)$ clearly only depend on the similarity type of $\phi_1$ over $F(\phi)$, so the same is true of $\Delta(\phi)$.
\end{remark}

The discussion of the previous subsection gives us the following:

\begin{proposition} \label{PROPpropertiesofDelta} Let $\phi$ be an anisotropic quasilinear quadratic form of dimension $\geq 2$ over $F$, and let $n$ be the unique integer for which $2^n < \mydim{\phi} \leq 2^{n+1}$. 

\begin{enumerate} \item $\Delta(\phi)$ contains $0$ and $\mathrm{lndeg}(\phi)-1$.
\item If $ r\in \Delta(\phi)$, then $r \notin [n+1,\mathrm{lndeg}(\phi)-2]$.
\item $n+1 \in \Delta(\phi)$ if and only if $\mathrm{ndeg}(\phi) = n+2$ and $\Izhdim{\phi} > 2^n$.
\item $n \in \Delta(\phi)$ if and only if one of the following holds:
\begin{itemize} \item $\phi$ is a quasi-Pfister neighbour $($i.e., $\mathrm{lndeg}(\phi) = n+1)$;
\item $\mathrm{lndeg}(\phi) = n+2$ and $\Izhdim{\phi} > 2^n$. \end{itemize}
\item If $r \in \Delta(\phi)$ for some non-negative integer $r \leq \mathrm{lndeg}(\phi) - 2$, then:
\begin{itemize} \item[$\mathrm{(i)}$] $\mathrm{lndeg}(\phi) \in [r+1,r+1+y_r]$, where $y_r$ is the largest integer for which $\Izhdim{\phi} > y_r2^r$;
\item[$\mathrm{(ii)}$] $\Izhdim{\phi} - \witti{2}{\phi}$ is divisible by $2^r$;
\item[$\mathrm{(iii)}$] $\witti{j}{\phi}$ is divisible by $2^r$ for all $j \in [3, \mathrm{lndeg}(\phi) - r - 1]$.
\end{itemize}
\item If $\mathrm{lndeg}(\phi) \geq \frac{\mydim{\phi}}{2} + 2$, then $\Delta(\phi) = \lbrace 0,\mathrm{lndeg}(\phi)-1 \rbrace$. 
\end{enumerate}
\begin{proof} We may assume that $n \geq 1$, so that $\mydim{\phi_1} \geq 2$ (if $\mydim{\phi} = 2$, then $\Delta(\phi) = \lbrace 0 \rbrace$ by definition).

(1) It is clear that $0 \in \Delta(\phi)$, while $\mathrm{lndeg}(\phi) - 1 = \mathrm{lndeg}(\phi_1) \in \Delta(\phi)$ by Lemma \ref{LEMbasicfactsonPr} (1). 

(2) Again, since $\mathrm{lndeg}(\phi_1) = \mathrm{lndeg}(\phi) - 1$, this holds by Lemma \ref{LEMbasicfactsonPr} (1).

(3) By Lemma \ref{LEMbasicfactsonPr} (1), $n+1 \in \Delta(\phi)$ if and only if $\phi_1$ is a neighbour of an anisotropic $(n+1)$-fold quasi-Pfister form, which holds if and only if $\mathrm{lndeg}(\phi_1) = n+1$ and $\mydim{\phi_1} > 2^n$. The claim then follows since $\mathrm{lndeg}(\phi_1) = \mathrm{lndeg}(\phi) - 1$ and $\mydim{\phi_1} = \Izhdim{\phi}$.

(4) By Lemma \ref{LEMIzhbound}, we have $\mydim{\phi_1} = \Izhdim{\phi} \geq 2^n$. By Lemma \ref{LEMbasicfactsonPr} (2), it follows that $n \in \Delta(\phi)$ if and only if one of the following holds:

\begin{itemize} \item $\phi_1$ is similar to an $n$-fold quasi-Pfister.
$\phi_1$ is a neighbour of an anisotropic $(n+1)$-fold quasi-Pfister form;
\item $\phi_1$ is similar to an $n$-fold quasi-Pfister form.
\end{itemize} 

Now the first condition holds if and only if $\phi$ is a quasi-Pfister neighbour (Lemma \ref{LEMcharacterizationofquasiPfisterneighbours}), and we have just seen that the second holds if and only if $\mathrm{lndeg}(\phi) = n+2$ and $\Izhdim{\phi} > 2^n$.

(5) Since $\mathrm{lndeg}(\phi_1) = \mathrm{lndeg}(\phi) - 1$ and $\mydim{\phi_1} = \Izhdim{\phi}$, all three statements hold by Lemma \ref{LEMdiscreteconstraintsfromnontrivialityofPr}.

(6) As above, this holds by Corollary \ref{CORPremptyforgenericforms}.
\end{proof}
\end{proposition}

\begin{examples} \label{EXSDeltavalues} \begin{enumerate}[leftmargin=*] \item If $\phi$ is a quasi-Pfister neighbour, then $\Delta(\phi) = \lbrace 0,1,\hdots,n \rbrace$, where $n$ is the unique integer with $2^n < \mydim{\phi} \leq 2^{n+1}$. Indeed, in this case, $\phi_1$ is similar to an $n$-fold quasi-Pfister form (Lemma \ref{LEMcharacterizationofquasiPfisterneighbours}), and so the claim holds by Example \ref{EXPrinvariantsforPfister}. 
\item If $X_1,\hdots,X_n$ are $n \geq 5$ indeterminates, and $\phi$ is the (anisotropic) quasilinear quadratic form $\form{X_1,\hdots,X_n}$ over $F(X_1,\hdots,X_n)$, then $\mathrm{lndeg}(\phi) = n-1 \geq \frac{n}{2}+2$, and so $\Delta(\phi) = \lbrace 0, n-1 \rbrace$ by part (6) of Proposition \ref{PROPpropertiesofDelta}. 
\end{enumerate} \end{examples}

By Proposition \ref{PROPtrivialityofPrinvariantunderrationalextensions}, we also have:

\begin{proposition} \label{PROPDeltadoesntchangeunderrationalextensions} Let $\phi$ be an anisotropic quasilinear quadratic form of dimension $\geq 2$ over $F$. If $L$ is a purely transcendental field extension of $F$, then $\Delta(\phi_L) = \Delta(\phi)$.
\end{proposition}

This allows us to prove that $\Delta$ respects stable birational equivalence:

\begin{proposition} Let $\phi$ and $\psi$ be quasilinear quadratic forms of dimension $\geq 2$ over $F$. If $\phi \stb \psi$, then $\Delta(\phi) = \Delta(\psi)$.
\begin{proof} By hypothesis, $F(\phi)$ and $F(\psi)$ are $F$-linearly embeddable into an extension $L$ of $F$ which is purely transcendental over both. Moreover, as noted in the proof of Lemma \ref{LEMknownstablebirationalinvariants} (4), the forms $(\phi_1)_L$ and $(\psi_1)_L$ are similar. In view of Remark \ref{REMDeltasimilaritytype}, the claim then follows immediately from Proposition \ref{PROPDeltadoesntchangeunderrationalextensions}. \end{proof}
\end{proposition}

Before proceeding, it will be convenient to make one further definition (its purpose will become clear at the beginning of the next section).

\begin{definition} Let $\phi$ be an anisotropic quasilinear quadratic form of dimension $\geq 2$ over $F$. If $\mathrm{lndeg}(\phi) = 1$ (resp. $\mathrm{lndeg}(\phi) = 2$), then we set $c(\phi): = \frac{3}{4}$ (resp. $c(\phi) : = \frac{3}{2}$). Otherwise, we set $c(\phi)$ equal to the largest integer $<\Izhdim{\phi}$ which is divisible by $2^m$, where $m: = \mathrm{max}\lbrace r\;|\; r \leq \mathrm{lndeg}(\phi)-3 \text{ and } r \in \Delta(\phi) \rbrace$.
\end{definition}

Since $\mathrm{lndeg}$ and $\Delta$ are stable birational invariants, the same is true of $c$. We have the following basic observations:

\begin{lemma} \label{LEMbasicpropertiesofc} Let $\phi$ be an anisotropic quasilinear quadratic form of dimension $\geq 2$ over $F$, and let $n$ be the unique integer for which $2^n < \mydim{\phi} \leq 2^{n+1}$.
\begin{enumerate} \item If $\Izhdim{\phi} > 2^n + 2^{n-1}$, then $c(\phi) \geq 2^n + 2^{n-1}$.
\item If $\Izhdim{\phi} > 2^n$, then $c(\phi) \geq 2^n$.
\item If $\Izhdim{\phi} = 2^n$, then $c(\phi) = 2^n - 2^m$ for some integer $m \in [0,n-2]$. 
\item If $\phi$ is a quasi-Pfister neighbour, then $c(\phi) = 2^n - 2^{n-2}$. 
\item If $\mathrm{lndeg}(\phi) \geq \frac{\mydim{\phi}}{2} + 2$, then $c(\phi) = \mydim{\phi}-2$.  
 \end{enumerate}
\begin{proof} If $\mathrm{lndeg}(\phi) = 1$, then $n = 0$, $\Izhdim{\phi} = 1 = 2^0$ and $c(\phi) = \frac{3}{4} = 2^0 - 2^{-2}$. Similarly, if $\mathrm{lndeg}(\phi) = 2$, then $n = 1$, $\Izhdim{\phi} = 2 = 2^1$ and $c(\phi) = \frac{3}{2} = 2^1 - 2^{-1}$. Thus, in these cases, (3) holds and none of the other statements are applicable. We can therefore assume that $\mathrm{lndeg}(\phi) \geq 3$.

(1,2,3) If $r \in \Delta(\phi)$ and $r \leq \mathrm{lndeg}(\phi) - 3$, then $r \leq n-1$ by parts (2), (3) and (4) of Proposition \ref{PROPpropertiesofDelta}. In particular, (1) and (2) hold. Suppose now that $\Izhdim{\phi} = 2^n$. By the preceding remarks, we then have that $c(\phi) = 2^n - 2^m$, where $m$ is the largest element of $\Delta(\phi)$ less than or equal to $\mathrm{max} \lbrace \mathrm{lndeg}(\phi)-3,n-1 \rbrace$. By part (5)(i) of Proposition \ref{PROPpropertiesofDelta}, however, we have $n-1 \in \Delta(\phi)$ only if $\mathrm{lndeg}(\phi) \leq n+y_{n-1}$, where $y_{n-1}$ is the largest integer for which $\Izhdim{\phi} > y_{n-1}2^{n-1}$. Since $\Izhdim{\phi} = 2^n$, $y_{n-1} = 1$, and so $\mathrm{lndeg}(\phi) = n+1$ in this case. We therefore have that $m \leq n-2$, and so (3) holds.

(4) If $\phi$ is a quasi-Pfister neighbour, then $\Izhdim{\phi} = 2^n$ (Lemma \ref{LEMcharacterizationofquasiPfisterneighbours}) and $n-2 \in \Delta(\phi)$ \big(Example \ref{EXSDeltavalues} (1)\big). Since $\mathrm{lndeg}(\phi) \geq n+1$, it follows that $c(\phi) \leq 2^n - 2^{n-2}$, and equality then holds by (3).

(5) If $\mathrm{lndeg}(\phi) \geq \frac{\mydim{\phi}}{2} + 2$, then $\Delta(\phi) = \lbrace 0, \mathrm{lndeg}(\phi)-1 \rbrace$ by Corollary \ref{CORPremptyforgenericforms}, and so $c(\phi) = \Izhdim{\phi} - 1$. At the same time, we shall see in Proposition \ref{PROPdivisibilityofp1} below that we must also have that $\witti{1}{p} = 1$ in this case, and so $\Izhdim{p} = \mydim{p}-1$. \end{proof} \end{lemma}

\begin{example} \label{EXcforgeneric} If $X_1,\hdots,X_n$ are $n \geq 5$ indeterminates, and $\phi$ is the (anisotropic) quasilinear quadratic form $\form{X_1,\hdots,X_n}$ over $F(X_1,\hdots,X_n)$, then $\mathrm{lndeg}(\phi) = n-1 \geq \frac{n}{2}+2$, and so $c(\phi) = n-2$ by part (5) of the proposition.
\end{example}

\section{Main Results}

We now come to our general results on the isotropy of quasilinear quadratic forms over function fields of quasilinear quadrics. The invariants $\Delta$ and $c$ introduced in the previous subsection play a key role here. Fix a field $F$ of characteristic 2. The main result is:

\begin{theorem} \label{THMmainresult} Let $p$ and $q$ be anisotropic quasilinear quadratic forms of dimension $\geq 2$ over $F$, and let $n$ be the unique integer for which $2^n < \Izhdim{p} \leq 2^{n+1}$. Suppose that $q_{F(p)}$ is isotropic, and set $k := \mydim{q} - 2\witti{0}{q_{F(p)}}$. For each non-negative integer $r$, let $y_r$ be the largest integer for which $\Izhdim{p} > y_r2^r$. If $k < \Izhdim{p}$, then either:
\begin{enumerate} \item $\mydim{q} = a2^{\mathrm{lndeg}(p)} + \epsilon$ for some positive integer $a$ and integer $\epsilon \in [-k,k]$; or
\item $p$ is not a quasi-Pfister neighbour, and there exist non-negative integers $r,r' \in \Delta(p)$ and a positive integer $x \geq r' - r + 1$ such that the following hold:
\begin{itemize} \item $k \geq y_r2^r$ and $\mydim{q} = a2^{\mathrm{lndeg}(p)} \pm \epsilon$ for some non-negative integer $a$ and positive integer $\epsilon \in [(x+y_r)2^{r+1} - k,x2^{r+1} +k]$;
\item $r \leq n-1$, $r' \in [r,n]$ and  $x \leq \mathrm{min}\lbrace 2^{n-1-r}, (y_{r'}+1)2^{r'-r} - y_r \rbrace$;
\item If $\mathrm{lndeg}(p) = n+2$, then $x2^r \leq 2^n - y_r2^{r-1}$. Otherwise, we have $x2^r \leq y_r2^r - c(p_1) \leq y_r2^r - (2^{n-1} + 2^{n-2})$.  \end{itemize} \end{enumerate}
\end{theorem}

In many situations, we are forced into the simpler case (1) due to the non-existence of integers $r$ and $r'$ satisfying the conditions in (2). Here, the integer $\epsilon$ in the formula for $\mydim{q}$ is subject to the same constraint as in Conjecture \ref{CONJfirstconjecture}, but the exponent of the $2$-power will typically be much larger than that in the latter, so the conclusion is significantly stronger (see, e.g., Example \ref{EXmainresultforgeneric} below). There are two basic reasons why we may be forced into case (1): The first is that the invariant $\Delta(p)$ may be too constrained, and the second is that the value of $k$ may be too small to allow the inequality $k \geq y_r2^r$ in (2) to be satisfied for any $r \leq n-1$. To the second point, the invariant $c$ allows us to make a precise statement:

\begin{corollary} \label{CORmainresult} Let $p$ and $q$ be anisotropic quasilinear quadratic forms of dimension $\geq 2$ over $F$. Suppose that $q_{F(p)}$ is isotropic, and set $k := \mathrm{dim}(q) - 2\witti{0}{q_{F(p)}}$. If $k < c(p)$, then $\mydim{q} = a2^{\mathrm{lndeg}(p)} + \epsilon$ for some positive integer $a$ and integer $\epsilon \in [-k,k]$. 
\begin{proof} By definition, $c(p)$ is strictly less than $\Izhdim{p}$. Since $k<c(p)$, Theorem \ref{THMmainresult} is therefore applicable. Let $n$ be as in the statement of the latter. If we were not in case (1), then there would exist a non-negative integer $r \in \Delta(p)$ such that $r \leq n-1$ and $k \geq y_r2^r$, where $y_r$ is the largest integer for which $\Izhdim{p} > y_r2^r$. By the definition of $c(p)$, this would imply that $\mathrm{lndeg}(p) \leq r+2 \leq n+1$. But since $\mydim{p_1} = \Izhdim{p} > 2^n$, we have $\mathrm{lndeg}(p) = \mathrm{lndeg}(p_1) +1 \geq n+2$, so we must in fact be in case (1).
\end{proof}
\end{corollary}

By the first three parts of Lemma \ref{LEMbasicpropertiesofc}, this gives:

\begin{corollary} \label{CORmainresult2} Let $p$ and $q$ be anisotropic quasilinear quadratic forms of dimension $\geq 2$ over $F$, and let $s$ be the unique integer for which $2^s < \mydim{p} \leq 2^{s+1}$. Suppose that $q_{F(p)}$ is isotropic, and set $k := \mathrm{dim}(q) - 2\witti{0}{q_{F(p)}}$. Suppose further that
$$ k < \begin{cases} 2^s + 2^{s-1} & \text{if } \Izhdim{p} > 2^s + 2^{s-1} \\ 2^{s} & \text{if }\Izhdim{p} \in (2^s,2^s + 2^{s-1}] \\
2^{s-1}+2^{s-2} & \text{if } \Izhdim{p} =2^s. \end{cases} $$
Then $\mydim{q} = a2^{\mathrm{lndeg}(p)} + \epsilon$ for some positive integer $a$ and integer $\epsilon \in [-k,k]$. 
\end{corollary}

Another situation where Corollary \ref{CORmainresult} effectively applies is that where $p$ is ``sufficiently generic''. More specifically, combining Corollary \ref{CORmainresult} with part (5) of Proposition \ref{PROPpropertiesofDelta} gives us the following:

\begin{corollary} \label{CORmainresultgeneric} Let $p$ and $q$ be anisotropic quasilinear quadratic forms of dimension $\geq 2$ over $F$ with $\mathrm{lndeg}(p) \geq \frac{\mydim{p}}{2}+2$. Suppose that $q_{F(p)}$ is isotropic, and set $k := \mathrm{dim}(q) - 2\witti{0}{q_{F(p)}}$. If $k \leq \mydim{p} - 3$, then $\mydim{q} = a2^{\mathrm{lndeg}(p)} + \epsilon$ for some positive integer $a$ and integer $\epsilon \in [-k,k]$. 
\end{corollary}

The case where $p$ is actually generic looks as follows:

\begin{example} \label{EXmainresultforgeneric}  Let $X_1,\hdots,X_n$ be $n \geq 5$ indeterminates, and let $p$ be the (anisotropic) quasilinear quadratic form $\form{X_1,\hdots,X_n}$ over $F(X_1,\hdots,X_n)$. Let $q$ be an anisotropic quasilinear quadratic form of dimension $\geq 2$ over $F(X_1,\hdots,X_n)$, and let $k  = \mydim{q} - 2\witti{0}{q_{F(p)}}$. If $k \leq n-3$, then it follows from Example \ref{EXcforgeneric} and Corollary \ref{CORmainresultgeneric} that $\mydim{q} = a2^{n-1} + \epsilon$ for some non-negative integer $a$ and some $\epsilon \in [-k,k]$.
\end{example}

\begin{remark} One might imagine that the case of generic forms is also accessible in the non-singular theory. However, if $p$ is a generic non-singular quadratic form of dimension $\geq 2$ over a field $K$, then even the cases where extreme isotropy occurs over $K(p)$ seem to be poorly understood (aside from the known results on Conjecture \ref{CONJfirstconjecture}). In particular, little seems to be known about the structure of the kernel of the restriction homomorphism from the quadratic Witt group of $K$ to the quadratic Witt group of $K(p)$.  
\end{remark}

Now, analyzing the contribution of the Izhboldin dimension in our main result, we obtain Theorem \ref{THMmaintheoremcorollary} from the introduction:

\begin{proof}[Proof of Theorem \ref{THMmaintheoremcorollary}] Let $p$, $q$, $s$ and $k$ be as in the statement of the theorem. Note that $\mathrm{lndeg}(p) =s+1$ when $p$ is a quasi-Pfister neighbour, with $\mathrm{lndeg}(p) \geq s+2$ otherwise (Lemma \ref{LEMcharacterizationofquasiPfisterneighbours}). Now, since $k < \Izhdim{p}$, Theorem \ref{THMmainresult} is applicable. If $p$ is a quasi-Pfister neighbour then the theorem says that (1) holds. Assume now that $p$ is not a quasi-Pfister neighbour. If $\mydim{q} = a2^{\mathrm{lndeg}(p)} + \epsilon$ for some positive integer $a$ and integer $\epsilon \in [-k,k]$, then we are in case (i) of (2), and (1) also holds. We may therefore suppose that we are in the second case allowed by Theorem \ref{THMmainresult}. Let $r,y_r$ and $x$ be as in the statement of the latter. Since $2^s \leq \Izhdim{p} < 2^{s+1}$, the following then hold:
\begin{itemize} \item[(a)] $k \geq y_r2^r$ and $\mydim{q} = a2^{s+2} \pm \epsilon$ for some non-negative integer $a$ and positive integer $\epsilon \in [(x+y_r)2^{r+1} - k, x2^{r+1} + k]$;
\item[(b)] $r \leq s-1$ with $r \leq s-2$ when $\Izhdim{p} = 2^s$, and $x \leq \mathrm{min}\lbrace 2^{s-1-r}, 2^{s+1-r}-y_r \rbrace$ with $x \leq 2^{s-2-r}$ when $\Izhdim{p} = 2^s$. \end{itemize}
To prove that (1) holds in this case, it suffices to show that $\epsilon + k \geq 2^{s+1}$ in (a). But since $\Izhdim{p} \geq 2^s$, and since $x$ is positive, we have 
$$ \epsilon \geq (x+y_r)2^{r+1} -k \geq (1+y_r)2^{r+1} - k \geq 2\Izhdim{p} - k \geq 2^{s+1} - k, $$
and so the claim indeed holds. Now, if $\Izhdim{p} = 2^s$, then $y_r2^r = 2^s - 2^r$, so (a) and (b) tell us that (2)(ii) is satisfied. Suppose finally that $\Izhdim{p} > 2^s$. If $\mydim{q} = a2^{s+2} + \epsilon$ in (a), then (2)(iii) is satisfied. If not, then $a \geq 1$ and $\mydim{q} = a'2^{s+2} + \epsilon'$, where $a' = a-1 \geq 0$ and $\epsilon' = 2^{s+2} - \epsilon$. Set $x' : = 2^{s+1-r} - y_r - x$. Since $x$ is positive, we have $x' < 2^{s+1-r} - y_r$. Moreover, $x'$ is non-negative by (b). Since $\epsilon \in [(x+y_r)2^{r+1}-k,, x2^{r+1} + k]$, we have $\epsilon' =2^{s+2} - \epsilon \in [(x'+y_r)2^{r+1} -k, x'2^{r+1} + k]$. Now $y_r2^{r+1} \leq 2k$, so if $x'= 0$, then $\epsilon' \in [-k,k]$. Since $\mydim{q} = a'2^{s+1} + \epsilon'$, (2)(i) is then satisfied ($a'$ must be positive in this case, since the isotropy of $q_{F(p)}$ forces $\mydim{q}$ to be greater than $k$). If $x' \neq 0$, on the other hand, then $x' \in [1,2^{s+1-r}-y_r]$, and the equality $\mydim{q} = a'2^{s+2} + \epsilon'$, together with (b), shows that (2)(iii) is satisfied. This completes the proof. 
\end{proof}

Now, the key ingredient in the proof of Theorem \ref{THMmainresult} the following result from \cite{Scully1}:

\begin{theorem}[{\cite[Thm. 6.4]{Scully1}}] \label{THMoldtheorem} Let $p$ and $q$ be anisotropic quasilinear quadratic forms of dimension $\geq 2$ over $F$. Then there exists an anisotropic quasilinear quadratic form $\tau$ of dimension $\witti{0}{q_{F(p)}}$ over $F(p)$ such that $\anispart{(\tau \otimes p_1)} \subset \anispart{(q_{F(p)})}$. 
\end{theorem}

Note here that the integer $\mydim{\anispart{(q_{F(p)})}} - \mydim{\tau}$ coincides with $k = \mydim{q} - 2\witti{0}{q_{F(p)}}$. Thus, given the conclusion of Theorem \ref{THMoldtheorem}, proving Theorem \ref{THMmainresult} becomes a matter of understanding something about the dimension of the form $\anispart{(\tau \otimes p_1)}$. We achieve this with the technical Theorem \ref{THMrefinedtensorproducttheorem} below. First, we shall need some additional preliminaries. It will convenient here to introduce the following notation:

\begin{definition} For any quasilinear quadratic form $\phi$ over $F$, we set $d(\phi): = \mydim{\phi} - 2\witti{0}{\phi}$. 
\end{definition}

Note that in the situtation of Theorem \ref{THMmainresult}, the integer $k$ is nothing else but $d(q_{F(p)})$. By Lemmas \ref{LEMisotropyoverfunctionfieldsofquadrics} and \ref{LEMdivisibilitybyquasiPfister}, we have the following:

\begin{lemma} \label{LEMdisnonnegative} Let $\phi$ and $\psi$ be anisotropic quasilinear quadratic forms $F$ with $\mydim{\psi} \geq 2$. Then $d(\phi_{F(\psi)}) \geq 0$, and equality holds if and only if $\phi$ is divisible by $\normform{\psi}$. 
\end{lemma}

The proof of Theorem \ref{THMrefinedtensorproducttheorem} will be inductive. To achieve the induction step, we will use a trick with symmetric bilinear forms that was not observed by the author at the time of \cite{Scully1} (and which we already used in the proof of Proposition \ref{PROPstabilityofPrunderseparable} (4) above).

\subsection{Tool for Induction} \label{SUBSECinductivetool}

Recall that if $a \in F \setminus \lbrace 0 \rbrace$, then $\mathbb{M}_a$ denotes the non-degenerate symmetric bilinear form on $F^2$ given by the assignment $\big((x_1,y_1),(x_2,y_2)\big) \mapsto a(x_1y_1 + x_2y_2)$. The non-degenerate symmetric bilinear form on $F^2$ given by the assignment $\big((x_1,y_1),(x_2,y_2)\big) \mapsto x_1y_2 + y_1x_2$ is the hyperbolic plane $\mathbb{H}$. Although both forms are isotropic, they are not isomorphic due to the characteristic assumption on $F$. The Witt decomposition theorem for symmetric bilinear forms therefore admits the following refinement in characteristic 2 (see \cite[(2.1)]{LaghribiMammone}): Let $\mathfrak{b}$ be a non-degenerate symmetric bilinear form over $F$. Then there exist an anisotropic symmetric bilinear form $\anispart{\mathfrak{b}}$ over $F$, non-negative integers $r$ and $s$, and elements $a_1,\hdots,a_r \in F^\times$ such that the following hold:

\begin{itemize} \item[(i)] $\mathfrak{b} \cong \anispart{\mathfrak{b}} \perp \mathbb{M}_{a_1} \perp \cdots \perp \mathbb{M}_{a_r} \perp s \cdot \mathbb{H}$;
\item[(ii)] The quasilinear quadratic form $\phi_{\anispart{b}} \perp \form{a_1,\hdots,a_r}$ is anisotropic.\end{itemize}

Furthermore, the integers $r$ and $s$ are unique, and $\anispart{\mathfrak{b}}$ is unique up to isomorphism. We set $\mathfrak{i}_{h}(\mathfrak{b}) = s$ and $\mathfrak{i}_W(\mathfrak{b}) = r+s$. If $\mathfrak{i}_W(\mathfrak{b})  = \frac{\mydim{\mathfrak{b}}}{2}$, then we say that $\mathfrak{b}$ is \emph{split} (this amounts to saying that $\anispart{\mathfrak{b}}$ has dimension 0, or that $\mathfrak{b}$ represents the zero element in the Witt ring of $F$). Note that the quasilinear quadratic form associated to $\mathbb{H}$ is the form $\langle 0,0 \rangle$. By (ii), it follows that $\anispart{(\phi_\mathfrak{b})} \simeq \phi_{\anispart{\mathfrak{b}}} \perp \form{a_1,\hdots,a_r}$, and so $\witti{0}{\phi_\mathfrak{b}} = 2\mathfrak{i}_h(\mathfrak{b}) + r = \mathfrak{i}_W(\mathfrak{b}) + \mathfrak{i}_h(\mathfrak{b})$. Since $\mydim{\phi_\mathfrak{b}} = \mydim{\mathfrak{b}} = \mydim{\anispart{\mathfrak{b}}} + 2\mathfrak{i}_W(\mathfrak{b})$, this gives:

\begin{lemma} \label{LEMformulafordofdiagonalpartofbilinear} If $\mathfrak{b}$ is a non-degenerate symmetric bilinear form over $F$, then $d(\phi_\mathfrak{b}) = \mydim{\anispart{\mathfrak{b}}} - 2\mathfrak{i}_h(\mathfrak{b})$. 
\end{lemma}

We now observe the following:

\begin{lemma} \label{LEMvalueofdwithnohyperbolicplanes} Let $\mathfrak{b}$ be a non-degenerate symmetric bilinear form over $F$, and let $\nu$ be an anisotropic quasilinear quadratic form of dimension $\geq 2$ over $F$. If $\mathfrak{b}$ is a subform of an anisotropic symmetric bilinear form over $F$ that splits over $F(\nu)$, then $\mathfrak{i}_h(\mathfrak{b}_{F(\nu)}) = 0$ and $d((\phi_\mathfrak{b})_{F(\nu)}) = \mydim{\anispart{(\mathfrak{b}_{F(\nu)})}}$.
\begin{proof} By Lemma \ref{LEMformulafordofdiagonalpartofbilinear}, it suffices to prove the first assertion. Now if $\mathfrak{c}$ is a subform of a non-degenerate symmetric bilinear form $\mathfrak{d}$ over a field of characteristic $2$, then it is clear from the definitions that $\mathfrak{i}_h(\mathfrak{d}) \geq \mathfrak{i}_h(\mathfrak{c})$. To prove what we need, we may therefore assume that $\mathfrak{b}$ splits over $F(\nu)$. Lemma \ref{LEMformulafordofdiagonalpartofbilinear} then tells us that $\mathfrak{i}_h(\mathfrak{b}_{F(\nu)}) = -d((\varphi_\mathfrak{b})_{F(\nu)})$. But since $\mathfrak{b}$ is anisotropic, the integer $d((\varphi_\mathfrak{b})_{F(\nu)})$ is non-negative (Lemma \ref{LEMdisnonnegative}), and so we must then have that $\mathfrak{i}_h(\mathfrak{b}_{F(\nu)}) = 0$. 
\end{proof} \end{lemma}

In particular, we get:

\begin{corollary} \label{CORcomplementarybilinearforms} Let $\mathfrak{b}$ and $\mathfrak{c}$ be non-degenerate symmetric bilinear forms over $F$ such that $\mathfrak{b} \perp \mathfrak{c}$ is anisotropic, and let $\nu$ be an anisotropic quasilinear quadratic form of dimension $\geq 2$ over $F$. If $(\mathfrak{b} \perp \mathfrak{c})_{F(\nu)}$ is split, then $d((\phi_\mathfrak{b})_{F(\nu)}) = d((\phi_\mathfrak{c})_{F(\nu)})$.
\begin{proof} Since $\mathfrak{b} \perp \mathfrak{c}$ splits over $F(\nu)$, the anisotropic forms $\anispart{(\mathfrak{b}_{F(\nu)})}$ and $\anispart{(\mathfrak{c}_{F(\nu)})}$ are Witt equivalent, and hence isomorphic (\cite[Prop. 2.4]{EKM}). In particular, they have the same dimension, and so the claim follows from Lemma \ref{LEMvalueofdwithnohyperbolicplanes} (applied to both $\mathfrak{b}$ and $\mathfrak{c}$). \end{proof} \end{corollary}

Before stating the main consequence, we need the following obvious statement:

\begin{lemma} \label{LEMsubformsofbilinearforms} Let $\psi$ be a quasilinear quadratic form over $F$, and let $\mathfrak{d}$ be a symmetric bilinear form over $F$ such that $\psi \subset \phi_\mathfrak{d}$. Then $\mathfrak{d}$ admits a subform $\mathfrak{b}$ such that $\psi \simeq \phi_\mathfrak{b}$.
\begin{proof} Let $V$ be the $F$-vector space on which $\mathfrak{d}$ is defined. We may assume that $V_\psi$ is a subspace of $V$ and that $\psi$ is the restriction of $\phi_{\mathfrak{d}}$ to this subspace. The restriction of $\mathfrak{d}$ to $V_{\psi}$ then has the desired property.
\end{proof} \end{lemma}

The result we want is now the following:

\begin{proposition} \label{PROPinductivepropositionweak} Let $\psi$ and $\nu$ be anisotropic quasilinear quadratic forms over $F$ with $\mydim{\nu} \geq 2$. Suppose there exists an anisotropic symmetric bilinear form $\mathfrak{d}$ over $F$ such that $\psi \subset \phi_{\mathfrak{d}}$. If $\mathfrak{d}_{F(\nu)}$ is split, then there exists an anisotropic quasilinear quadratic form $\hat{\psi}$ over $F$ such that $\psi \perp \hat{\psi} \simeq \phi_{\mathfrak{d}}$ and $d(\hat{\psi}_{F(\nu)}) = d(\psi_{F(\nu)})$. 
\begin{proof} By Lemma \ref{LEMsubformsofbilinearforms}, there exists a subform $\mathfrak{b}$ of $\mathfrak{d}$ such that $\psi \simeq \phi_\mathfrak{b}$. Let $\mathfrak{c}$ be the complementary subform (so that $\mathfrak{d} \simeq \mathfrak{b} \perp \mathfrak{c}$). The form $\hat{\psi} := \phi_\mathfrak{c}$ then has the desired properties by Corollary \ref{CORcomplementarybilinearforms}.
\end{proof}
\end{proposition} 

We will specifically use the following special case:

\begin{corollary} \label{CORinductivepropositionweak} Let $\psi$ and $\eta$ be quasilinear quadratic forms over $F$ such that $\psi \subset \eta$. If $\eta$ is divisible by an anisotropic quasi-Pfister form $\pi$ of dimension $\geq 2$ over $F$, then there exists an anisotropic quasilinear quadratic form $\hat{\psi}$ over $F$ such that $\psi \perp \hat{\psi} \simeq \eta$ and $d(\hat{\psi}_{F(\nu)}) = d(\psi_{F(\nu)})$ for any subform $\nu \subset \pi$ of dimension $\geq 2$. 
\begin{proof} Let $\mathfrak{d}'$ be a bilinear Pfister form over $F$ with $\phi_{\mathfrak{d}'} \simeq \pi$. We then have that $\eta \simeq \phi_{\mathfrak{d}}$ for some anisotropic symmetric bilinear form $\mathfrak{d}$ over $F$ which is divisible by $\mathfrak{d}'$. Since an isotropic bilinear Pfister form is split, $\mathfrak{d}'$ splits over the function field of any subform of $\pi$ having dimension $\geq 2$. Since $\mathfrak{d}$ is divisible by $\mathfrak{d}'$, the same is then true of $\mathfrak{d}$, and so we can apply Proposition \ref{PROPinductivepropositionweak} to get the desired conclusion.
\end{proof}
\end{corollary}

\begin{remark} \label{REMinductivepropositiontensorbyPfister} More generally, the form $\hat{\psi}$ constructed here has the property that $d(\hat{\psi}_{M(\nu)}) = d(\psi_{M(\nu)})$ for every separable field extension $M$ of $F$ and every subform $\nu \subset \pi_M$ of dimension $\geq 2$. Indeed, this is implicit in the proof in view of Lemma \ref{LEManisotropyoverseparable}. 
\end{remark}

We will actually need the following extension of the previous result:

\begin{corollary} \label{CORinductivepropositiontensorbyPfister} Let $\psi$ and $\sigma$ be subforms of an anisotropic quasilinear quadratic form $\eta$ over $F$. If $\eta$ is divisible by an anisotropic quasi-Pfister form $\pi$ of dimension $\geq 2$ over $F$, then there exist anisotropic quasilinear quadratic forms $\hat{\psi}$ and $\hat{\sigma}$ over $F$ such that the following hold:
\begin{enumerate} \item $\psi \perp \hat{\psi} \simeq \eta \simeq \sigma \perp \hat{\sigma}$;
\item $d(\hat{\psi}_{F(\nu)}) = d(\psi_{F(\nu)})$ and $d(\hat{\sigma}_{F(\nu)}) = d(\sigma_{F(\nu)})$ for every subform $\nu \subset \pi$ of dimension $\geq 2$;
\item If $Y$ is an indeterminate, then $d\big((\hat{\psi} \perp Y\hat{\sigma})_{F(Y)(\nu)}\big) = d\big((\sigma \perp Y\psi)_{F(Y)(\nu)}\big)$ for every subform $\nu \subset \pi_{F(Y)} \otimes \pfister{Y}$ of dimension $\geq 2$. 
\end{enumerate}
\begin{proof} As above, let $\mathfrak{d}'$ be a bilinear Pfister form over $F$ with $\phi_{\mathfrak{d}'} \simeq \pi$, and let $\mathfrak{d}$ be an anisotropic symmetric bilinear form over $F$ which is divisible by $\mathfrak{d}'$ and which satisfies $\phi_{\mathfrak{d}} \simeq \eta$. By Lemma \ref{LEMsubformsofbilinearforms}, there exists subforms $\mathfrak{b}$ and $\mathfrak{b}'$ of $\mathfrak{d}$ with $\phi_{\mathfrak{b}} \simeq \psi$ and $\phi_{\mathfrak{b}'} \simeq \sigma$. Let $\mathfrak{c}$ and $\mathfrak{c}'$ be complementary subforms of $\mathfrak{b}$ and $\mathfrak{b}'$ in $\mathfrak{d}$, respectively, and set $\hat{\psi} := \phi_{\mathfrak{c}}$ and $\hat{\sigma}: = \phi_{\mathfrak{c}'}$. The proof of Corollary \ref{CORinductivepropositionweak} then shows that $\hat{\psi}$ and $\hat{\sigma}$ satisfy (1) and (2). Now, replacing $F$ with $F(Y)$, we have that $(\mathfrak{b}' \perp Y\mathfrak{b}) \perp (\mathfrak{c}' \perp Y \mathfrak{b})$ is isomorphic to the (anisotropic) form $\mathfrak{d}\perp Y\mathfrak{d}$, which is divisible by the (anisotropic) Pfister form $\mathfrak{d}' \perp Y \mathfrak{d}$ (for anisotropy, apply Lemmas \ref{LEMisotropyoverquadratic} and \ref{LEManisotropyoverseparable}). The same arguments as above then show that $d\big((\hat{\sigma} \perp Y\hat{\psi})_{F(\nu)}\big) = d\big((\sigma \perp Y\psi)_{F(\nu)}\big)$ for every subform $\nu \subset \pi \otimes \pfister{Y} = \phi_{\mathfrak{d}' \perp Y \mathfrak{d}'}$. Since $\hat{\psi} \perp Y\hat{\sigma} \simeq Y(\hat{\sigma} \perp Y \hat{\psi})$, (3) then also holds. 
\end{proof} \end{corollary}

\subsection{A Corollary of Theorem \ref{THMoldtheorem}} Before proceeding to the key technical result, it will be convenient to record an immediate consequence of Theorem \ref{THMoldtheorem} for the invariant $c$. First, we have the following (which already implies part (1) of Corollary \ref{CORintro}):

\begin{proposition}[{\cite[Cor. 6.14]{Scully1}}] \label{PROPdivisibilityofp1} If $\phi$ is an anisotropic quasilinear quadratic form of dimension $\geq 2$ over $F$, then $\phi_1$ is divisible by an anisotropic quasi-Pfister form of dimension $\geq \witti{1}{\phi}$.
\begin{proof} Applying Theorem \ref{THMoldtheorem} with $q = p = \phi$, we see that there exists an anisotropic quasilinear quadratic form $\tau$ of dimension $\witti{1}{\phi}$ over $F(\phi)$ such that $\anispart{(\tau \otimes \phi_1)} \subset \phi_1$. But $\mydim{\anispart{(\tau \otimes \phi_1)}} \geq \mydim{\phi_1}$ by Lemma \ref{LEMsubformoftensorproduct}, so we must then have that $\anispart{(\tau \otimes \phi_1)} \simeq \phi_1$. By Lemma \ref{LEMdivisibilitybyquasiPfister}, $\phi_1$ is then divisible by $\normform{\tau}$, which has dimension $\geq \mydim{\tau} = \witti{1}{\phi}$. 
\end{proof}
\end{proposition}

Now, the value of $c(\phi)$ for an arbitrary quasi-Pfister neighbour $\phi$ was determined in Lemma \ref{LEMbasicpropertiesofc} (4). For non-quasi-Pfister neighbours, the above result gives:

\begin{corollary} \label{CORcvaluefornonquasiPfister} Let $\phi$ be an anisotropic quasilinear quadratic form over $F$ which is not a quasi-Pfister neighbour, and let $n$ be the unique positive integer for which $2^n < \mydim{\phi} \leq 2^{n+1}$. Then:
\begin{enumerate} \item $c(\phi) \leq \Izhdim{\phi} - 2^u$, where $u$ is the smallest integer for which $\witti{1}{\phi} \leq 2^u$;
\item If $\Izhdim{\phi} > 2^n$, then $c(\phi) \geq 2^n$;
\item If $\Izhdim{\phi} = 2^n$, then $c(\phi) = 2^n - 2^m$ for some integer $m \in [0,n-2]$ with $\mydim{\phi} \leq 2^n + 2^m$;
\item $c(\phi) \geq \frac{\mydim{\phi}}{2}$;
\item $\witti{1}{\phi} \leq \frac{\mydim{\phi}}{4}$.
\end{enumerate}
\begin{proof} (1) Since $\phi$ is not a quasi-Pfister neighbour, we have $\mathrm{lndeg}(\phi) \geq n+2 \geq 3$. By definition, we then have $c(\phi)  = \Izhdim{\phi} - 2^m$, where $m = \mathrm{max} \lbrace r\;|\; r \leq \mathrm{lndeg}(\phi)-3 \text{ and } r \in \Delta(\phi)\rbrace$. Now Proposition \ref{PROPdivisibilityofp1} tells us that $u \in \Delta(\phi)$, so to prove the claim, it will be enough to show that $u \leq n-1$. But since $2^n \leq \Izhdim{\phi} < 2^{n+1}$, $u$ is at most $n$. Moreover, if $u$ were equal to $n$, then $\phi_1$ (which has dimension $\Izhdim{\phi}$) would have be similar to an $n$-fold quasi-Pfister form, contradicting the fact that $\phi$ is not a quasi-Pfister neighbour (Lemma \ref{LEMcharacterizationofquasiPfisterneighbours}). The claim therefore holds.

(2) This is part of Lemma \ref{LEMbasicpropertiesofc} (2).

(3) If $\Izhdim{\phi} = 2^n$, then Lemma \ref{LEMbasicpropertiesofc} (3) tells us that $c(\phi) = 2^n - 2^m$ for some $m \in [0, n-2]$. But $\mydim{\phi} = 2^n + \witti{1}{\phi}$ in this case, and so (1) then implies that $\mydim{\phi} \leq 2^n + 2^m$.

(4) This follows immediately from (2) and (3).

(5) By (1), we have $c(\phi)  \leq \mydim{\phi} - 2\witti{1}{\phi}$, and (4) then gives $\witti{1}{\phi} \leq \frac{\mydim{\phi}}{4}$.
\end{proof} \end{corollary}

\subsection{A Theorem on Tensor Products and the Proof of Theorem \ref{THMmainresult}} Now, recall that to deduce Theorem \ref{THMmainresult} from Theorem \ref{THMoldtheorem}, we need to understand something about highly isotropic tensor products of quasilinear quadratic forms. The key technical result we shall prove here is the following:

\begin{theorem} \label{THMrefinedtensorproducttheorem} Let $\psi$ and $\phi$ be non-zero anisotropic quasilinear quadratic forms over $F$, and set $\sigma  := \anispart{(\psi \otimes \phi)}$ and $d := \mydim{\sigma} - \mydim{\psi}$. If $d < \mydim{\phi}$, then either $\sigma$ is divisible by $\normform{\phi}$, or there exist anisotropic quasilinear quadratic forms $\alpha_1,\hdots,\alpha_m$ and a non-negative integer $r \leq \mathrm{lndeg}(\phi)-m-1$ such that the following hold:
\begin{enumerate}
\item $\alpha_m$ is an $r$-fold quasi-Pfister form;
\item $\alpha_m$ divides $\alpha_1,\hdots,\alpha_{m-1},\sigma$, and $[\alpha_m] \in P_r(\psi) \cap P_r(\phi)$;
\item Set $\alpha_{-1} := \anispart{(\alpha_m \otimes \psi)}$ and $\alpha_0 : = \anispart{(\alpha_m \otimes \phi)}$. For each $i \in [-1,m-1]$, set $\beta_{i+1} : = \anispart{(\alpha_i \otimes \alpha_{i+1})}$, and let $n_i$ $($resp. $v_i)$ be the largest integer for which $\mydim{\alpha_i} > 2^{n_i}$ $($resp. $\mydim{\alpha_i} > v_i2^{\mathrm{lndeg}(\alpha_{i+1})})$. Then: 
\begin{itemize} \item[$\mathrm{(i)}$] $\beta_0 \simeq \sigma$;
\item[$\mathrm{(ii)}$] For all $i \in [0,m-1]$, $\mydim{\alpha_{i+1}} \leq 2^{n_i-1}$ and $\mathrm{lndeg}(\alpha_{i+1}) < \mathrm{lndeg}(\alpha_{i})$;
\item[$\mathrm{(iii)}$] For all $i \in [-1,m-1]$, $[\normform{(\alpha_{i+1})}] \in P_{\mathrm{lndeg}(\alpha_{i+1})}(\alpha_i)$;
\item[$\mathrm{(iv)}$] For all $i \in [0,m-1]$, $\mydim{\beta_i} = \mydim{\alpha_{i-1}} + \mydim{\alpha_i} -2^r$ and
$$ \mydim{\alpha_{i+1}} = \mathrm{min} \lbrace \mydim{\alpha_{i-1}} - v_{i-1}2^{\mathrm{lndeg}(\alpha_i)}, (v_{i-1}+1)2^{\mathrm{lndeg}(\alpha_i)} - \mydim{\beta_i} \rbrace;$$
\item[$\mathrm{(v)}$]For all $i \in [1,m-1]$, $d\big((\beta_i)_{F(\alpha_{i-1})}\big) \leq \mydim{\alpha_{i-1}} - \mydim{\alpha_i} - 2^r$.
\end{itemize}
\end{enumerate}
\end{theorem}

We make some comments:

\begin{remarks} \label{REMStensorproducttheorem} \begin{enumerate}[leftmargin=*] \item Suppose we are in the case where $\sigma$ is not divisible by $\normform{\phi}$. Let $y_r$ and $y_r'$ be the largest integers for which $\mydim{\phi} > y_r2^r$ and $\mydim{\psi} > y_r'2^r$, respectively. Since $[\alpha_m] \in P_r(\psi) \cap P_r(\phi)$, we then have that $\mydim{\alpha_{-1}} = (y_r'+1)2^r$ and $\mydim{\alpha_0} = (y_r+1)2^r$ (Remark \ref{REMPrdimension}). At the same time, since $r<\mathrm{lndeg}(\phi)$, the first two parts of Lemma \ref{LEMbasicfactsonPr} tell us that $\alpha_0 \stb \phi$. In particular, we have $\normform{(\alpha_0)} \simeq \normform{\phi}$ and $\mathrm{lndeg}(\alpha_0) = \mathrm{lndeg}(\phi)$ (Lemma \ref{LEMknownstablebirationalinvariants}). Part (3) then tells us, in particular, that the following hold:
\begin{itemize} \item $\mydim{\sigma} = (y_r + y_r' + 1)2^r$ and $d \geq y_r2^r$;
\item $[\normform{\phi}] \in P_{\mathrm{lndeg}(\phi)}(\psi)$ (because $\anispart{(\normform{\phi} \otimes \psi)} \simeq \anispart{(\normform{(\alpha_0)} \otimes \alpha_{-1})}$);
\item If $v$ is the largest integer for which $\mydim{\psi} > v2^{\mathrm{lndeg}(\phi)}$, then
$$ \mydim{ \alpha_1} = \mathrm{min} \lbrace (y_r'+1)2^r - v2^{\mathrm{lndeg}(\phi)}, (v + 1)2^{\mathrm{lndeg}(\phi)} - \mydim{\sigma} \rbrace $$ 
(since $r < \mathrm{lndeg}(\phi)$, $v$ coincides with the integer $v_0$ in the statement).
 \end{itemize}

\item If $\mathrm{lndeg}(\phi) \leq 1$ (i.e., if $\mydim{\phi} \in \lbrace 1,2 \rbrace$), then the conclusion of the theorem is that $\sigma$ is divisible by $\phi$, which we have already observed in Lemma \ref{LEManisotropicpartofformsdivisiblebyaquasiPfister}. Suppose now that $\mathrm{lndeg}(\phi) \geq 2$. Let $t$ be the largest integer $\leq \mathrm{lndeg}(\phi)-2$ for which $P_t(\phi) \neq \emptyset$, and let $c$ be the largest integer less than $\mydim{\phi}$ which is divisible by $2^t$. If $d < c$, then $\sigma$ must be divisible by $\normform{\phi}$, since otherwise the first bullet-point in the previous remark would be invalidated. By Lemma \ref{LEMbasicfactsonPr} (1), this applies when $d < 2^n < \mydim{\phi}$ for some non-negative integer $n$. We shall make use of this in the proof (which is inductive). \end{enumerate}
\end{remarks}

Theorem \ref{THMmainresult} is a direct consequence of Theorems \ref{THMoldtheorem} and \ref{THMoldtheorem}. In fact, we have:

\begin{proposition} \label{PROPtensorproducttheoremimpliesmainresult} Let $l$ be a positive integer. If Theorem \ref{THMrefinedtensorproducttheorem} holds whenever $d<l$, then Theorem \ref{THMmainresult} holds whenever $k<l$. 
\begin{proof} Let $p$, $q$ and $k$ and $n$ be as in the statement of Theorem \ref{THMmainresult}, and suppose that $k \leq l$. By Theorem \ref{THMoldtheorem}, there exists an anisotropic quasilinear quadratic form $\tau$ of dimension $\witti{0}{q_{F(p)}}$ over $F(p)$ such that $\anispart{(\tau \otimes p_1)} \subset \anispart{(q_{F(p)})}$. Set $\sigma: = \anispart{(\tau \otimes p_1)}$. By definition, we have 
\begin{equation} \label{eq1} \mydim{q} = 2\witti{0}{q_{F(p)}} + k = 2\mydim{\tau} +k. \end{equation}
On the other hand, since $\mydim{\anispart{(q_{F(p)})}} = \mydim{q} - \witti{0}{q_{F(p)}} = \frac{\mydim{q}+k}{2}$, we also have that
\begin{equation} \label{eq2} \mydim{q} \geq 2\mydim{\anispart{(q_{F(p)})}} -k \geq 2\mydim{\sigma} -k. \end{equation} 
Together, \eqref{eq1} and \eqref{eq2} give that $\mydim{\sigma} - \mydim{\tau} \leq k$. Since $k$ is at most $\Izhdim{p} - 1$ and strictly less than $l$, we are in a position to apply Theorem \ref{THMrefinedtensorproducttheorem} with $\psi = \tau$ and $\phi = p_1$ \big(we are now working over the field $F(p)$\big). If $\sigma$ is divisible by $\normform{(p_1)}$, then $\mydim{\sigma}$ is divisible by $2^{\mathrm{lndeg}(p_1)} = 2^{\mathrm{lndeg}(p)-1}$. Since $\mydim{\tau} \leq \mydim{\sigma}$ (Lemma \ref{LEMsubformoftensorproduct}), \eqref{eq1} and \eqref{eq2} then give that $\mydim{q} = a2^{\mathrm{lndeg}(p)} + \epsilon$ for some positive integer $a$ and integer $\epsilon \in [-k,k]$. Note that this holds when $p$ is a quasi-Pfister neighbour, since $p_1$ is then similar to $\normform{(p_1)}$ (Lemma \ref{LEMcharacterizationofquasiPfisterneighbours}), and the latter then divides $\sigma$ by Lemma \ref{LEManisotropicpartofformsdivisiblebyaquasiPfister}. We can therefore suppose that $p$ is not a quasi-Pfister neighbour, and that we are in the second case of Theorem  \ref{THMrefinedtensorproducttheorem}. Let $\alpha_1,\hdots,\alpha_m$ and $r$ be as in the statement of the latter, and set $\alpha: = \alpha_1$, $\pi: = \alpha_m$ and $\beta:= \anispart{(\alpha \otimes p_1)}$. If $y_r$ the largest integer for which $\Izhdim{p} > y_r2^r$, $y_r'$ the largest integer for which $\mydim{\tau} > y_r'2^r$, and $v$ is the largest integer for which $\mydim{\tau} > v2^{\mathrm{lndeg}(p_1)}= v2^{\mathrm{lndeg}(p)-1}$, the following then hold (we are using Remark \ref{REMStensorproducttheorem} (1) as well as the statement of the theorem here):
\begin{itemize}  \item[(a)] $\pi$ is an anisotropic $r$-fold quasi-Pfister form;
\item[(b)] $\alpha$ is divisible by $\pi$ and $[\pi] \in P_r(p_1)$;
\item[(c)] $\mydim{\alpha} \leq 2^{n-1}$ and $\mathrm{lndeg}(\alpha) < \mathrm{lndeg}(p_1) = \mathrm{lndeg}(p)-1$;
\item[(d)] $[\normform{\alpha}] \in P_{\mathrm{lndeg}(\alpha)}(p_1)$;
\item[(e)] $\mydim{\sigma} = (y_r' + y_r+1)2^r$;
\item[(f)] $\mydim{\alpha} = \mathrm{min}\lbrace (y_r'+1)2^r - v2^{\mathrm{lndeg}(p)-1}, (v+1)2^{\mathrm{lndeg}(p)-1} - \mydim{\sigma} \rbrace$;
\item[(g)] $\mydim{\beta} = \mydim{\alpha} + y_r2^r$;
\item[(h)] $d(\beta_{F(p)(p_1)}) \leq y_r2^r- \mydim{\alpha}$. 
\end{itemize}
Set $r' : = \mathrm{lndeg}(\alpha)$. By (b), (c) and (d), $r$ and $r'$ lie in $\Delta(p)$. By (b), there exists a quasilinear quadratic form $\alpha'$ over $F(p)$ such that $\alpha \simeq \pi \otimes \alpha'$. Set $x: = \mydim{\alpha'} = \frac{\mydim{\alpha}}{2^r}$. By Lemma \ref{LEMisotropyindicesofformsdivisiblebyquasiPfisters} (1), we then have $r' \leq r + x-1$, or $x \geq r'-r+1$. We now claim that the integers $r$, $r'$ and $x$ satisfy all the remaining conditions required in case (2) of our theorem. In other words, we claim the following hold:
\begin{itemize} \item[(i)] $k \geq y_r2^r$ and $\mydim{q} = a2^{\mathrm{lndeg}(p)} \pm \epsilon$ for some non-negative integer $a$ and positive integer $\epsilon \in [(x+y_r)2^{r+1} k, x2^{r+1} + k]$;
\item[(ii)] $r \leq n-1$, $r' \in [r,n]$ and $x \leq \mathrm{min} \lbrace 2^{n-1-r}, (y_{r'} + 1)2^{r'-r} - y_r \rbrace$ ($y_{r'}$ being the largest integer for which $\Izhdim{p} > y_{r'}2^{r'}$);
\item[(iii)] If $\mathrm{lndeg}(p) = n+2$, then $x2^r \leq 2^n - y_r2^{r-1}$. Otherwise, we have $x2^r \leq y_r2^r - c(p_1) \leq y_r2^r - (2^{n-1} + 2^{n-2})$. \end{itemize}
Let us check these one by one.

(i) Statement (f) above gives two possible values for $x2^r=  \mydim{\alpha}$. Suppose first that $x2^r = (y_r'+1)2^r - v2^{\mathrm{lndeg}(p)-1}$. By the definition of $y_r'$, we then have that $\mydim{\tau} \leq v2^{\mathrm{lndeg}(p)-1} + x2^r$, and so $\mydim{q} \leq v2^{\mathrm{lndeg}(p)} + x2^{r+1}+k$ by \eqref{eq1}. On the other hand, (e) gives that 
$$ \mydim{\sigma} = (y_r'+1)2^r + y_r2^r =  v2^{\mathrm{lndeg}(p)-1} + (x+y_r)2^r, $$
and so $\mydim{q} \geq v2^{\mathrm{lndeg}(p)} + (x+y_r)2^{r+1} - k$ by \eqref{eq2}. Setting $a = v$, we then get that $\mydim{q} = a2^{\mathrm{lndeg}(p)} + \epsilon$ for some $\epsilon \in [(x+y_r)2^{r+1}-k,x2^{r+1}+k]$, and so (i) holds (note that $k$ must be $\geq y_r2^r$ in order for the interval containing $\epsilon$ to be non-empty). Now the other possibility is that $x2^r = (v+1)2^{\mathrm{lndeg}(p)-1} - \mydim{\sigma}$. In this case, \eqref{eq2} becomes $\mydim{q} \geq (v+1)2^{\mathrm{lndeg}(p)} - (x2^{r+1}+k)$. At the same time, (e) gives that
$$ \mydim{\tau} \leq (1+y_r')2^r = \mydim{\sigma} - y_r2^r = (v+1)2^{\mathrm{lndeg}(p)-1} - (x+y_r)2^r, $$
and so $\mydim{q} \leq (v+1)2^{\mathrm{lndeg}(p)} - \big((x+y_r2^r)-k\big)$ by \eqref{eq1}. Setting $a = (v+1)$, we then get that $\mydim{q} = a2^{\mathrm{lndeg}(p)} - \epsilon$ for some $\epsilon \in [(x+y_r)2^{r+1}-k,x2^{r+1}+k]$, and so (i) again holds. 

(ii) By (b), $x2^r \leq 2^{n-1}$, so $r \leq n-1$ and $x \leq 2^{n-1-r}$. Since $\alpha$ is divisible by $\pi$, $r'$ is at least $r$. At the same time, $r' < \mathrm{lndeg}(p_1)$, and (d) and Lemma \ref{LEMbasicfactsonPr} (1) then imply that $r' \leq n$. To prove (ii), it then only remains to show that $x \leq (y_r'+1)2^{r'-r} - y_r$, or that $(x+y_r)2^r \leq (y_r'+1)2^{r'}$. But (g) says that $(x+y_r)2^r = \mydim{\beta}$, and $\beta$ is similar to a subform of $\anispart{(\normform{\alpha} \otimes p_1)}$, which has dimension $(y_r'+1)2^{r'}$ by (d) and Remark \ref{REMPrdimension}. 

(iii) Since $\mydim{p_1} = \Izhdim{p} > 2^n$, we have $\mathrm{lndeg}(p) = \mathrm{lndeg}(p_1) +1 \geq n+2$. If $\mathrm{lndeg}(p) = n+2$, then (e) and (f) give that
\begin{eqnarray*} x2^{r+1} =2\mydim{\alpha} & \leq & \big((y_r'+1)2^r - v2^{n+1}\big) + \big((v+1)2^{n+1} - \mydim{\sigma}\big) \\ &=& 2^{n+1} + \big((y_r'+1)2^r - \mydim{\sigma}\big) \\ &=& 2^{n+1} - y_r2^r, \end{eqnarray*}
and so $x2^r \leq 2^n - y_r2^{r-1}$. Suppose now that $\mathrm{lndeg}(p) \geq n+3$. By (h), $d\big(\beta_{F(p)(p_1)}\big) \leq (y_r-x)2^r$. If $(y_r-x)2^r$ were less than $c(p_1)$, then it would follow from the preceding discussion (with $\beta$ replacing $q$ and $p_1$ replacing $p$) that $\mydim{\beta}$ lies with $(y_r-x)2^r$ of an integer multiple of $2^{\mathrm{lndeg}(p_1)} = 2^{\mathrm{lndeg}(p)-1} \geq 2^{n+2}$. But $\mydim{\beta} = (x+y_r)2^r$ by (c), and since $(x+y_r)2^r + (y_r-x)2^r = y_r2^{r+1} < 2\Izhdim{p} \leq 2^{n+2}$, we see that this is not the case. We must therefore have $(y_r - x)2^r \geq c(p_1)$, or $x2^r \leq y_r2^r - c(p_1)$. This proves that (iii) holds, since $c(p_1) \geq 2^{n-1} + 2^{n-2}$ by Lemma \ref{LEMbasicpropertiesofc}. 
\end{proof} \end{proposition}

\begin{remark} \label{REMsuboptimalityoftensorproducttheorem} Note that when we were applying Theorem \ref{THMrefinedtensorproducttheorem} here, we made no use of the forms $\alpha_2,\hdots,\alpha_{m-1}$ appearing in the second case. Nevertheless, we have stated Theorem \ref{THMrefinedtensorproducttheorem} as it is because our calculations indicate that these forms can also be made subject to at least some non-trivial constraints coming from the invariant $\Delta(p)$. Furthermore, when we are in this case of the theorem, the given conditions on the $\alpha_i$ show that the dimension of $\sigma = \anispart{(\psi \otimes \phi)}$ is expressible in terms of $\mydim{\alpha_0}$ and $v_i2^{\mathrm{lndeg}(\alpha_{i+1})}$ ($i \in [-1,m-1]$). While the ambiguity in the formula for $\mydim{\alpha_{i+1}}$ prevents us from giving an exact expression, it may be possible to get something more concrete by refining the construction. Any improvements here would of ultimately lead to a strengthening of Theorem \ref{THMmainresult}.
\end{remark}

We now proceed to the proof of Theorem \ref{THMrefinedtensorproducttheorem}:

\begin{proof}[Proof of Theorem \ref{THMrefinedtensorproducttheorem}] We induct on $d$. If $d = 0$, then $\psi$ is similar to $\sigma$ (Lemma \ref{LEMsubformoftensorproduct}). By Lemma \ref{LEMdivisibilitybyquasiPfister}, we then get that both $\psi$ and $\sigma$ are divisible by $\normform{\phi}$ (the case where $\mydim{\phi}=1$ is trivial). Suppose now that $d \geq 1$, and that the result holds whenever the relevant dimension difference is less than $d$. By Proposition \ref{PROPtensorproducttheoremimpliesmainresult}, Theorem \ref{THMmainresult} then holds whenever $k<d$, and the same is then true of Corollary \ref{CORmainresult} (we shall make use of this below). Now, modifying $\psi$ and $\phi$ by scalars if needed, we can assume that $1 \in D(\psi) \cap D(\phi)$. By Lemma \ref{LEMsubformoftensorproduct}, both $\psi$ and $\phi$ are then subforms of $\sigma$. If $\sigma$ is divisible by $\normform{\phi}$, then there is nothing to prove, so assume otherwise. Since $d \geq 1$, we have $\mydim{\phi} \geq 2$, and so the field $F(\phi)$ is defined. We have:
 
 \begin{lemma} \label{LEMnonzerods} Both $d(\psi_{F(\phi)})$ and $d(\sigma_{F(\phi)})$ are non-zero. 
 \begin{proof} Since $\phi \subset \normform{\phi}$, we have $\anispart{(\normform{\phi} \otimes \phi)} \simeq \normform{\phi}$ (Lemma \ref{LEMdivisibilitybyquasiPfister}). In particular, if $\psi$ were divisible by $\normform{\phi}$, then the same would be true of $\sigma$. Since this is not the case, neither $\psi$ nor $\sigma$ are divisible by $\normform{\phi}$, and the statement then follows from Lemma \ref{LEMdisnonnegative}. 
 \end{proof}  \end{lemma}

Now, the first main step of the proof is to show that $[\normform{\phi}] \in P_{\mathrm{lndeg}(\phi)}(\psi)$. For this, we will make us of the following observations:

\begin{lemma} \label{LEMbasicfactsontensorproduct} The following hold:
\begin{enumerate} \item $d(\psi_{F(\phi)}) + d(\sigma_{F(\phi)}) \leq d$;
\item $\mydim{\anispart{(\sigma_{F(\phi)})}} \leq \witti{0}{\psi_{F(\phi)}} + d = \mydim{\anispart{(\psi_{F(\phi)})}} + d - d(\psi_{F(\phi)})$;
\item For any $a \in D(\psi) \setminus \lbrace 0 \rbrace$, there exists a subform $\nu$ of $\phi_1$ such that $\mydim{\nu} \geq \frac{d(\psi_{F(\phi)}) + \mydim{\phi} - d}{2}$ and $a\nu \subset \anispart{(\psi_{F(\phi)})}$. \end{enumerate}

\begin{proof} Let $\phi' \subset \phi$ be such that $\phi \simeq \form{1} \perp \phi'$, let $K = F(V_{\phi'})$, and let $\phi'(X) \in K$ be the generic value of $\phi'$. As an extension of $F$, we can then identity $F(\phi)$ with $K(\sqrt{\phi'(X)})$. Consider the form $\eta: = \anispart{(\psi \otimes \pfister{\phi'(X)})}$ over $K$. Since $\pfister{\phi'(X)}$ is a subform of $\phi_K$, $\eta$ is a subform of $\sigma_K$ (which is anisotropic by Lemma \ref{LEManisotropyoverseparable}). By Lemma \ref{LEMisotropyoverfunctionfieldsofquadrics}, we have $d(\psi_{F(\phi)}) = \mydim{\eta} - \mydim{\psi}$. Since $\eta$ is divisible by $\pfister{\phi'(X)}$ (Lemma \ref{LEManisotropicpartofformsdivisiblebyaquasiPfister}), we then have that
$$ \mydim{\anispart{(\eta_{F(\phi)})}} = \frac{\mydim{\eta}}{2} = \frac{\mydim{\psi} + d(\psi_{F(\phi)})}{2} = \mydim{\anispart{(\psi_{F(\phi)})}}, $$
and so $\anispart{(\eta_{F(\phi)})} \simeq \anispart{(\psi_{F(\phi)})}$. We now verify the three desired assertions.

(1) By Lemma \ref{LEMisotropyoverfunctionfieldsofquadrics}, $d(\sigma_{F(\phi)}) = \mydim{\sigma} - m$, where $m$ is the maximal dimension of a subform of $\sigma_K$ which is divisible by $\pfister{\phi'(X)}$. Since $\eta$ is a subform of $\sigma_K$ divisible by $\pfister{\phi'(X)}$, it follows that 
$$ d(\sigma_{F(\phi)}) \leq \mydim{\sigma} - \mydim{\eta} = \mydim{\sigma} - \big(\mydim{\psi} + d(\psi_{F(\phi)})\big) = d - d(\psi_{F(\phi)}),$$
as desired.

(2) Since $\mydim{\sigma}  = \mydim{\psi} + d = \mydim{\eta} + d - d(\psi_{F(\phi)})$, and since $\eta$ is a subform of $\sigma_K$, we have that $\mydim{\anispart{(\sigma_{F(\phi)})}} \leq \mydim{\anispart{(\eta_{F(\phi)})}} + d - d(\psi_{F(\phi)})$. But we noted above that $\anispart{(\eta_{F(\phi)})} \simeq \anispart{(\psi_{F(\phi)})}$, and so the desired inequality holds (note that $d(\psi_{F(\phi)}) = \mydim{\anispart{(\psi_{F(\phi)})}}  - \witti{0}{\psi_{F(\phi)}}$ by definition, so the equality in the statement is immediate). 

(3) By Lemma \ref{LEMsubformoftensorproduct}, we have $a\phi \subset \sigma$, and so $aD(\phi_K)$ is a $K^2$-linear subspace of $D(\sigma_K)$. By dimension count, it intersects $D(\eta)$ in a $K^2$-linear subspace of dimension at least 
$$ \mydim{\eta} + \mydim{\phi} - \mydim{\sigma} = d(\psi_{F(\phi)}) + \mydim{\phi} - d. $$
By Lemma \ref{LEMsubformtheorem}, this means that $\phi_K$ admits a subform $\rho$ of dimension at least $d(\psi_{F(\phi)}) + \mydim{\phi} - d$ such that $a\rho \subset \eta$. Setting $\nu := \anispart{(\rho_{F(\phi)})}$, we then get that $a\nu \subset \anispart{(\eta_{F(\phi)})} \simeq \anispart{(\psi_{F(\phi)})}$. Finally, since $F(\phi) = K(\sqrt{\phi'(X)})$, Lemma \ref{LEMisotropyoverquadratic} tells us that $\mydim{\nu} \geq \frac{\mydim{\rho}}{2} \geq \frac{d(\psi_{F(\phi)}) + \mydim{\phi} - d}{2}$, and so $\nu$ has the desired property.
\end{proof} \end{lemma}

We note the following consequence of the first part:

\begin{corollary} \label{CORatleastonedissmall} If $\phi$ is not a quasi-Pfister neighbour, then at least one of $d(\psi_{F(\phi)})$ and $d(\sigma_{F(\phi)})$ is less than $c(\phi)$. 
\begin{proof} If this were not the case, then Lemma \ref{LEMbasicfactsontensorproduct} (1) would imply that $2c(\phi) \leq d < \mydim{\phi}$, contradicting Lemma \ref{CORcvaluefornonquasiPfister} (4). 
\end{proof} \end{corollary}

We can now complete the first main step:

\begin{proposition} \label{PROPlndegliesinP} $[\normform{\phi}] \in P_{\mathrm{lndeg}(\phi)}(\psi)$. 

\begin{proof} For ease of notation, let $\psi'$ and $\sigma'$ be the anisotropic parts of $\psi$ and $\sigma$ over $F(\phi)$, respectively. By Theorem \ref{THMoldtheorem}, there exists an anisotropic quasilinear quadratic form $\tau$ of dimension $\witti{0}{\psi_{F(\phi)}}$ over $F(\phi)$ such that $\anispart{(\tau \otimes \phi_1)} \subset \psi'$. We set $\psi'' : = \anispart{(\tau \otimes \phi_1)}$. By Lemma \ref{LEMsubformtheorem}, $\tau$ is a subform of $\psi''$, so $\mydim{\tau} \leq \mydim{\psi''} \leq \mydim{\psi'}$, with $\mydim{\psi'} - \mydim{\tau}$ being equal to $d(\psi_{F(\phi)}) = \mydim{\psi} - 2\witti{0}{\psi_{F(\phi)}}$. 

Now, set $l+1 : = \mathrm{lndeg}(\phi)$, and let $v$ be the largest integer for which $\mydim{\psi} > v2^{l+1}$. To prove what we want, we have to show that $\mydim{\anispart{(\normform{\phi} \otimes \psi)}} < (v+2)2^{l+1}$ (see Remark \ref{REMPrdimension}). Note, however, that the anisotropic part of $\anispart{(\normform{\phi} \otimes \psi)}$ over $F(\phi)$ coincides with $\anispart{(\normform{(\phi_1)} \otimes \psi')}$ (see the remarks preceding Lemma \ref{LEMdropingofnormdegree}). In particular, if we can show that the latter has dimension $< (v+2)2^l = \frac{1}{2}\big((v+2)2^{l+1}\big)$, then the desired assertion will follow from Lemma \ref{LEMisotropyoverfunctionfieldsofquadrics}. We first show:

\begin{claim} \label{LEMboundsoverfunctionfield} In the above situation, we have $v2^l < \mydim{\psi'} \leq (v+1)2^l$. 
\begin{proof} Since $\mydim{\psi} > v2^{l+1}$, the lower bound is again a consequence of Lemma \ref{LEMisotropyoverfunctionfieldsofquadrics}. Suppose now that $\mydim{\psi'} = (v+1)2^l+ y$ for some non-negative integer $y$. We have to show that $y = 0$. Suppose for the sake of contradiction that $y>0$, and let $x\in [0,2^{l+1}-1]$ be such that $\mydim{\psi} = (v+1)2^{l+1} -x$. We then have that $\mydim{\tau} = \witti{0}{\psi_{F(\phi)}} = (v+1)2^l - (x+y)$ and $\mydim{\psi'} - \mydim{\tau} = d(\psi_{F(\phi)}) =x+2y$. We claim that $\psi''$ is divisible by $\normform{(\phi_1)}$ and has dimension $ \geq (v+1)2^l$. Before proving this claim, let us first explain how it gives what we want: Let $a \in D(\psi) \setminus \lbrace 0 \rbrace$. Since $d(\psi_{F(\phi)}) = x+2y$, and since $d < \mydim{\phi}$, Lemma \ref{LEMbasicfactsontensorproduct} (3) tells us that $\phi_1$ admits a subform $\nu$ such that $\mydim{\nu} > y$ and $a \nu \subset \psi'$. Since $\mydim{\psi'} = (v+1)2^l + y$, $a D(\nu) = D(a \nu)$ must then have non-zero intersection with $D(\psi'')$. But since $\psi''$ is divisible by $\normform{(\phi_1)}$, it is closed under multiplication by arbitrary elements of $D(\phi_1)$ (Lemma \ref{LEMdivisibilitybyquasiPfister}), and hence by arbitrary elements of $D(\nu)$. Since the latter is closed under inversion of non-zero elements, it then follows that $a \in D(\psi'')$. But $D(\psi')$ is generated as an $F(\phi)^2$-vector space by $D(\psi)$, so Lemma \ref{LEMsubformtheorem} then gives that $\psi' \simeq \psi''$. In particular, $\psi'$ is divisible by $\normform{(\phi_1)}$. Since $\mathrm{lndeg}(\phi_1) = \mathrm{lndeg}(\phi) - 1 =l$, $\mydim{\psi'}$ is then divisible by $2^l$, so the same must be true of $y$. But Lemma \ref{LEMbasicfactsontensorproduct} (1) gives that $2y \leq x+2y = d(\psi_{F(\phi)}) \leq d < \mydim{\phi} \leq 2^{l+1}$, so this forces $y=0$, contradicting our standing assumption. It remains to prove our claim about $\psi''$. Before proceeding, we note that both $d(\psi_{F(\phi)})$ and $d(\sigma_{F(\phi)})$ are strictly less than $d$ by Lemmas \ref{LEMnonzerods} and \ref{LEMbasicfactsontensorproduct} (1). We now separate two cases.
\vspace{.5 \baselineskip}

\noindent {\it Case 1}. $\phi$ is a quasi-Pfister neighbour, i.e., $\mydim{\phi} > 2^l$. In this case, $\phi_1$ is an $l$-fold quasi-Pfister form (Lemma \ref{LEMcharacterizationofquasiPfisterneighbours}), and so $\normform{(\phi_1)} = \phi_1$. Since $\psi'' = \anispart{(\tau \otimes \phi_1)}$, the divisibility assertion then holds by Lemma \ref{LEManisotropicpartofformsdivisiblebyaquasiPfister}. In particular, $\mydim{\psi''}$ is divisible by $2^l$. Since $\mydim{\psi''} \geq \mydim{\tau}$, the dimension claim will follow if we can show that $\mydim{\tau} > v2^l$. But another application of Lemma \ref{LEManisotropicpartofformsdivisiblebyaquasiPfister} tells us that $\sigma' \simeq \anispart{(\psi' \otimes \phi_1)}$ is also divisible by $\phi_1$, and hence has dimension divisible by $2^l$. Since it contains $\psi'$ as a subform (Lemma \ref{LEMsubformoftensorproduct}), and since $\mydim{\psi'} = (v+1)2^l + y > (v+1)2^l$, it follows that $\mydim{\sigma'} \geq (v+2)2^l$. By Lemma \ref{LEMbasicfactsontensorproduct} (2), this gives that $\mydim{\tau} = \witti{0}{\psi_{F(\phi)}} \geq (v+2)2^l - d$. But since $d < \mydim{\phi} \leq 2^{l+1}$, we then get that $\mydim{\tau} > v2^l$, as desired. \vspace{.5 \baselineskip}

\noindent {\it Case 2}. $\phi$ is not a quasi-Pfister neighbour, i.e., $\mydim{\phi} \leq 2^l$. In this case, we claim that $d(\psi_{F(\phi)}) < c(\phi)$. Suppose otherwise. By Corollary \ref{CORatleastonedissmall}, we then have that $d(\sigma_{F(\phi)}) < c(\phi)$. Since $d(\sigma_{F(\phi)}) < d$, and since Corollary \ref{CORmainresult} holds when $k<d$ (recall that this is implied by the induction hypothesis), it follows that $\mydim{\sigma}$ lies within $d(\sigma_{F(\phi)})$ of an integer multiple of $2^{l+1}$. But both $d(\sigma_{F(\phi)})$ and $\mydim{\sigma} - (v+1)2^{l+1}$ are less than $d$, and since $d < \mydim{\phi} \leq 2^l$, we must then have that $\mydim{\sigma} \leq (v+1)2^{l+1} + d(\sigma_{F(\phi)})$. But Lemma \ref{LEMbasicfactsontensorproduct} (1) tells us that
$$ d(\sigma_{F(\phi)}) \leq d - d(\psi_{F(\phi)}) = \mydim{\sigma} - \big(\mydim{\psi} + d(\psi_{F(\phi)})\big) = \mydim{\sigma} - \big((v+1)2^{l+1}+2y\big), $$
so this contradicts our assumption that $y >0$. The claim therefore holds, i.e., $d(\psi_{F(\phi)}) < c(\phi)$. In particular, $\mydim{\psi''} - \mydim{\tau} \leq \mydim{\psi'} - \mydim{\tau} = d(\psi_{F(\phi)}) < \mathrm{min} \lbrace c(\phi),d \rbrace$, and so an application of the induction hypothesis to the pair $(\tau,\phi_1)$ gives that $\psi''$ is divisible by $\normform{(\phi_1)}$ \big(see Remark \ref{REMStensorproducttheorem} (2)\big). At the same time, we have $x+y < x + 2y = d(\psi_{F(\phi)}) < d < \mydim{\phi} \leq 2^l$, and so $\mydim{\psi''} > \mydim{\tau} = (v+1)2^l - (x+y) > v2^l$. Since $\mathrm{lndeg}(\phi_1) = l$, the divisibility of $\psi''$ by $\normform{(\phi_1)}$ then also gives that $\mydim{\psi''} \geq (v+1)2^l$, as desired.
\end{proof} \end{claim}

Returning to the proof of the proposition, our goal now is to show that $[\normform{(\phi_1)}] \in P_l(\psi')$ (as noted above, we have $\mathrm{lndeg}(\phi_1) = l$). Since $v2^l < \mydim{\psi'} \leq (v+1)2^l$, this will prove that $\anispart{(\normform{(\phi_1)} \otimes \psi')}$ has dimension $(v+1)2^l < (v+2)2^l$, which is exactly what we wanted to show (see the remarks preceding Claim \ref{LEMboundsoverfunctionfield}). But $\sigma' \simeq \anispart{(\psi' \otimes \phi_1)}$, and we have $\mydim{\sigma'} - \mydim{\psi'} \leq d - d(\psi_{F(\phi)}) < d$ by Lemmas \ref{LEMnonzerods} and \ref{LEMbasicfactsontensorproduct} (2). In particular, if $\mydim{\sigma'} - \mydim{\psi'} < \mydim{\phi_1}$, then our claim follows from an application of the induction hypothesis to the pair $(\psi',\phi_1)$. We may therefore assume henceforth that $\mydim{\sigma'} - \mydim{\psi'} \geq \mydim{\phi_1}$. By the preceding remarks, we then have that
$$ d(\psi_{F(\phi)}) \leq d - \mydim{\phi_1} < \mydim{\phi} - \mydim{\phi_1} = \witti{1}{\phi}.$$
Since $\witti{1}{\phi} \leq 2^l$ (Lemma \ref{LEMIzhbound}), it follows that
$$ \mydim{\psi''} \geq \mydim{\tau} = \mydim{\psi'} - d(\psi_{F(\phi)}) > v2^l - 2^l = (v-1)2^l. $$
We claim that $\psi''$ is divisible by $\normform{(\phi_1)}$. If $\phi$ is a quasi-Pfister neighbour, then $\normform{(\phi_1)} = \phi_1$ (Lemma \ref{LEMcharacterizationofquasiPfisterneighbours}) and the claim holds by Lemma \ref{LEManisotropicpartofformsdivisiblebyaquasiPfister}. If not, then $\mydim{\phi} \geq 4$ and $d(\psi_{F(\phi)}) < \witti{1}{\phi} < c(\phi)$ by parts (4) and (5) of Corollary \ref{CORcvaluefornonquasiPfister}. Since $\mydim{\psi''} - \mydim{\tau} \leq \mydim{\psi'} - \mydim{\tau} = d(\psi_{F(\phi)}) <d$, an application of the induction hypothesis to the pair $(\tau,\phi_1)$ then gives the claim. Since $\mydim{\psi''} > (v-1)2^l$, we also get that $\mydim{\psi''}$ is an integer multiple of $2^l$ greater than or equal to $v2^l$. Now $\mydim{\psi'} \leq (v+1)2^l$ (Claim \ref{LEMboundsoverfunctionfield}), so if $\mydim{\psi''} \geq (v+1)2^l$, then $\psi' \simeq \psi''$, and so $\psi'$ is divisible by $\normform{(\phi_1)}$. In particular, $[\normform{(\phi_1)}] \in P_l(\psi')$. Suppose now that $\mydim{\psi''} = v2^l$. Then $\psi''$ is a proper subform of $\psi'$. Since $D(\psi')$ is generated by $D(\psi)$ as an $F(\phi)^2$-vector space, we can find an element $a \in D(\psi)$ such that $a \notin D(\psi'')$. Consider now the form $\psi''' : = \psi'' \perp a \normform{(\phi_1)}$ of dimension $(v+1)2^l$. Since $\psi''$ is divisible by $\normform{(\phi_1)}$, the same is true of $\psi''$. Moreover, $\psi''$ is anisotropic. Indeed, if $\psi''$ were isotropic, then there would exist a non-zero element $b \in D\big(\normform{(\phi_1)}\big)$ such that $ab \in D(\psi'')$ (because $\psi''$ is anisotropic). But since $\psi''$ is divisible by $\normform{(\phi_1)}$, $D(\psi'')$ is closed under multiplication by arbitrary elements of $D\big(\normform{(\phi_1)}\big)$. Since the latter is closed under inversion of non-zero elements, it would then follow that $a \in D(\psi'')$, a contradiction. Now, let $b \in D(\psi')$. By Lemma \ref{LEMsubformtheorem}, the form $\anispart{(\psi'' \perp a\phi_1 \perp b \phi_1)}$ is a subform of $\sigma' \simeq \anispart{(\psi' \otimes \phi_1)}$. By Lemma \ref{LEMbasicfactsontensorproduct} (2), we have
$$ \mydim{\sigma'} \leq \witti{0}{\phi_{F(\phi)}} + d = \mydim{\tau} + d \leq \mydim{\psi''} + d < \mydim{\psi''} + \mydim{\phi}. $$
On the other hand, Lemma \ref{LEMIzhbound} tells us that $2\mydim{\phi_1} = 2\Izhdim{\phi_1} \geq \mydim{\phi}$, and so $\mydim{(\psi'' \perp a\phi_1 \perp b \phi_1)} \geq \mydim{\psi''} + \mydim{\phi}$. The form $\psi'' \perp a\phi_1 \perp b \phi_1$ must therefore be isotropic. Since $\psi'' \perp a \phi_1$ is anisotropic (because $\psi'''$ is), it follows that there exists a non-zero element $c \in D(\phi_1)$ such that $bc \in D(\psi'' \perp a\phi_1) \subseteq D(\psi''')$. But $\psi'''$ is divisible by $\normform{(\phi_1)}$, and $D(\phi_1) \subset D\big(\normform{(\phi_1)}\big)$, so the same reasoning as above then gives that $b \in D(\psi''')$. Since $b$ was an arbitrary element of $D(\psi')$, this shows that $\psi' \subset \psi'''$ (Lemma \ref{LEMsubformtheorem}). Since $\psi'''$ is divisible by $\normform{(\phi_1)}$, we then have that $\anispart{(\normform{(\phi_1)} \otimes \psi''')} \simeq \psi'''$ (Lemma \ref{LEMdivisibilitybyquasiPfister}), and so $\anispart{(\normform{(\phi_1)} \otimes \psi')} \subset \psi'''$. This proves what we want, since $\mydim{\tau'''} = (v+1)2^l < \mydim{\psi'} +2^l$.
\end{proof} \end{proposition}

With this established, we will now be able to use the construction of \S \ref{SUBSECinductivetool} to produce forms $\alpha_1,\hdots,\alpha_m$ satisfying the conditions in the statement of the theorem. First, however, it will be convenient to make a small adjustment to the pair $(\psi,\phi)$. 

\begin{lemma} \label{LEMmaximaldividingform} There exists a unique non-negative integer $r < \mathrm{lndeg}(\phi)$ and an anisotropic $r$-fold quasi-Pfister form $\pi \subset \normform{\phi}$ such that:
\begin{itemize} \item $\pi$ divides $\sigma$;
\item $[\pi] \in P_r(\phi)$;
\item If $\pi'$ is an anisotropic $s$-fold quasi-Pfister form for some integer $r<s<\mathrm{lndeg}(\phi)$, and $\pi'$ divides $\sigma$, then $[\pi'] \notin P_s(\phi)$. 
\end{itemize}
Moreover, the following hold:
\begin{enumerate} \item $[\pi] \in P_r(\psi)$;
\item $\sigma \simeq \anispart{(\widetilde{\psi} \otimes \widetilde{\phi})}$, where $\widetilde{\psi} := \anispart{(\pi \otimes \psi)}$ and $\widetilde{\phi} := \anispart{(\pi \otimes \phi)}$. \end{enumerate}
\begin{proof} Choose an anisotropic quasi-Pfister form $\pi$ of largest possible dimension over $F$ such that $\pi$ divides $\sigma$ and $[\pi] \in P_r(\phi)$, where $r = \mathrm{lndeg}(\phi)$. Since $\sigma$ is not divisible by $\normform{\phi}$, Lemma \ref{LEMbasicfactsonPr} (1) implies that $r < \mathrm{lndeg}(\phi)$. By part (2) of the same lemma, we then have that $\pi \subset \normform{\phi}$. Now the pair $(r,\pi)$ clearly satisfies the first part of the statement, and we just have to check that (1) and (2) hold. Since $\sigma$ is divisible by $\pi$, we have $\anispart{(\pi \otimes \sigma)} \simeq \sigma$ by Lemma \ref{LEMdivisibilitybyquasiPfister}. Thus, if we set $\widetilde{\psi} := \anispart{(\pi \otimes \psi)}$ and $\widetilde{\phi} := \anispart{(\pi \otimes \phi)}$, then
$$ \anispart{(\widetilde{\psi} \otimes \widetilde{\phi})} \simeq \anispart{\big((\pi \otimes \psi) \otimes (\pi \otimes \phi)\big)} \simeq \anispart{\big(\pi \otimes (\pi \otimes \sigma)\big)} \simeq \anispart{(\pi \otimes \sigma)} \simeq \anispart{\sigma}, $$
and (2) holds. For (1), we have to show that $\mydim{\widetilde{\psi}} < \mydim{\psi} + 2^r$. Suppose that this is not the case. Then $\mydim{\sigma} - \mydim{\widetilde{\psi}} < \mydim{\sigma} - \mydim{\psi} - 2^r =d -2^r$. Since $d< \mydim{\phi}$, we have $d - 2^r <y2^r$, where $y$ is the largest integer such that $\mydim{\phi} > y2^r$. By the induction hypothesis, it then follows that there exists an integer $s>r$ and an anisotropic $s$-fold quasi-Pfister form $\pi'$ over $F$ such that $\sigma$ is divisible by $\pi'$ and $[\pi'] \in P_s(\phi)$. But this contradicts our choice of $\pi$, so the claim must in fact hold. \end{proof} \end{lemma}

Now, let $r$ and $\pi$ be as in Lemma \ref{LEMmaximaldividingform}, and set $\widetilde{\psi} : =  \anispart{(\pi \otimes \psi)}$ and $\widetilde{\phi} : = \anispart{(\pi \otimes \phi)}$. When we construct the forms $\alpha_1,\hdots,\alpha_m$ below, the quasi-Pfister form $\alpha_m$ will coincide with $\pi$. Taking this for granted, statement (2) in the lemma now allows us to replace $\psi$ with $\widetilde{\psi}$ and $\phi$ with $\widetilde{\phi}$ without affecting anything in the statement of the theorem. Furthermore, after making this replacement, the pair $(r,\pi)$ still has the properties stated in Lemma \ref{LEMmaximaldividingform}. Indeed, $\sigma$ does not change, and if $s > r$, then it follows from Remark \ref{REMPrdimension} that any element of $P_s(\widetilde{\phi})$ is also an element of $P_s(\phi)$. Making the replacement, we can therefore assume that $\psi$, $\phi$ and $\sigma$ are all divisible by $\pi$. The integer $d = \mydim{\sigma} - \mydim{\psi}$ is then divisible by $2^r$. Since $d< \mydim{\phi}$, this gives:

\begin{lemma} \label{LEMboundond} $d \leq \mydim{\phi} - 2^r$.
\end{lemma}

With this modfication, we now apply the construction of \S \ref{SUBSECinductivetool} to get:

\begin{proposition} \label{PROPinductivestepinproofoftheorem}  Let $n$ be the unique integer for which $2^n< \mydim{\phi} \leq 2^{n+1}$. Then there exist non-zero anisotropic quasilinear quadratic forms $\hat{\psi}$ and $\hat{\sigma}$ over $F$ such that the following hold:
\begin{enumerate} \item $\psi \perp \hat{\psi} \simeq \anispart{(\normform{\phi} \otimes \psi)} \simeq \sigma \perp \hat{\sigma}$;
\item $\mydim{\hat{\psi}} - \mydim{\hat{\sigma}} = d$;
\item $d(\hat{\psi}_{F(\nu)}) = d(\psi_{F(\nu)})$ and $d(\hat{\sigma}_{F(\nu)}) = d(\sigma_{F(\nu)})$ for every subform $\nu \subseteq \normform{\phi}$ of dimension $\geq 2$;
\item If $\pi'$ is a quasi-Pfister subform of $\normform{\phi}$, then $\hat{\psi}$ $($resp. $\hat{\sigma})$ is divisible by $\pi
$ if and only if $\psi$ $($resp. $\sigma)$ is divisible by $\pi'$;
\item $\anispart{(\hat{\sigma} \otimes \phi)} \simeq \hat{\psi}$;
\item $\hat{\psi}$ and $\hat{\sigma}$ are similar to proper subforms of $\normform{\phi}$;
\item If $\mydim{\hat{\psi}} + \mydim{\hat{\sigma}} \leq 2^{\mathrm{lndeg}(\phi)}$, then $\mydim{\hat{\sigma}} \leq 2^{n-1}$ and $\mathrm{lndeg}(\hat{\sigma}) < \mathrm{lndeg}(\phi)$. \end{enumerate}

\begin{proof} Set $\eta: = \anispart{(\normform{\phi} \otimes \psi)}$. By Proposition \ref{PROPlndegliesinP}, we have $\mydim{\eta} - \mydim{\psi} < 2^{\mathrm{lndeg}(\phi)}$. Now, since $\phi \subset \normform{\phi}$, we have $\anispart{(\normform{\phi} \otimes \phi)} \simeq \normform{\phi}$ by Lemma \ref{LEMdivisibilitybyquasiPfister}. Then $\anispart{(\normform{\phi} \otimes \sigma)} \simeq \anispart{(\normform{\phi} \otimes \psi \otimes \phi)} \simeq \eta$, and so both $\psi$ and $\sigma$ are subforms of $\eta$. Applying Corollary \ref{CORinductivepropositiontensorbyPfister}, we obtain quasilinear quadratic forms $\hat{\psi}$ and $\hat{\sigma}$ such that the following hold:
\begin{enumerate} \item[(i)] $\psi \perp \hat{\psi} \simeq \eta \simeq \sigma \perp \hat{\sigma}$;
\item[(ii)] $d(\hat{\psi}_{F(\nu)}) = d(\psi_{F(\nu)})$ and $d(\hat{\sigma}_{F(\nu)}) = d(\sigma_{F(\nu)})$ for every subform $\nu \subseteq \normform{\phi}$ of dimension $\geq 2$;
\item[(iii)] If $Y$ is an indeterminate, then $d\big((\hat{\psi} \perp Y\hat{\sigma})_{F(Y)(\nu)}\big) = d\big((\sigma \perp Y\psi)_{F(Y)(\nu)}\big)$ for every subform $\nu \subset \pi_{F(Y)} \otimes \pfister{Y}$ of dimension $\geq 2$. 
\end{enumerate}
We claim that $\hat{\psi}$ and $\hat{\sigma}$ have the desired properties. Note first that (1) and (3) are just (i) and (ii) above, and that (2) is an immediate consequence of (1). Since $\eta$ is divisible by $\normform{\phi}$ (Lemma \ref{LEManisotropicpartofformsdivisiblebyaquasiPfister}), but $\sigma$ is not, both $\sigma$ and $\psi$ are proper subforms of $\eta$, and hence $\hat{\psi}$ and $\hat{\sigma}$ are non-zero. Moreover, since $\mydim{\eta} - \mydim{\psi} < 2^{\mathrm{lndeg}(\phi)}$, both forms have dimension $<2^{\mathrm{lndeg}(\phi)}$. We now show that (4), (5) and (6) hold.

(4) If $\mydim{\pi'} = 1$, then there is nothing to show. Otherwise, Lemma \ref{LEMdisnonnegative} tells us that an anisotropic quasilinear quadratic form $\rho$ over $F$ is divisible by $\pi'$ if and only if $d(\rho_{F(\pi')}) = 0$, and so the claims hold by (3).

(5) Set $K := F(V_{\phi})$, and let $\phi(X) \in K$ be the generic value of $\phi$. Consider the form $\nu = \form{Y} \perp \phi_{F(Y)}$ over the rational function field $F(Y)$. The field $F(Y)(\nu)$ is $K$-isomorphic to $K(Y)(\sqrt{Y^{-1}\phi(X)})$, which is a purely transcendental extension of $K$, and hence $F$. By Lemma \ref{LEManisotropyoverseparable}, $\sigma$ then remains anisotropic over $F(Y)(\nu)$. On the other hand, since $Y^{-1}\phi(X)$ is a square in $F(Y)(\nu)$, we have $(\sigma \perp Y\psi)_{F(Y)(\nu)} \simeq (\sigma_K \perp \phi(X)\psi_K)_{F(Y)(\nu)}$. But $\phi(X)\psi_K \subset \sigma_K$ by the definition of $\sigma$, and so the isotropy index of the form $(\sigma \perp Y\psi)_{F(Y)(\nu)}$ is equal to $\mydim{\psi}$. In particular, we have $d\big((\sigma \perp Y\psi)_{F(Y)(\nu)}\big) = \mydim{\sigma} - \mydim{\psi} = d$. Since $\nu \subset \phi_{F(Y)} \otimes \pfister{Y} \subset (\normform{\phi})_{F(Y)} \otimes \pfister{Y}$, we must then also have that $d\big((\hat{\psi} \perp Y\hat{\sigma})_{F(Y)(\nu)}\big) = d$ by statement (iii) above. By (2), this means that the isotropy index of $(\hat{\psi} \perp Y\hat{\sigma})_{F(Y)(\nu)}$ is equal to $\mydim{\hat{\sigma}}$. But since $F(Y)(\nu)$ is purely transcendental over $F$, $\hat{\psi}$ remains anisotropic over $F(Y)(\nu)$ (Lemma \ref{LEManisotropyoverseparable}), so we must then have that $Y\hat{\sigma}_{F(Y)(\nu)} \subset \hat{\psi}_{F(Y)(\nu)}$. As above, this may be rexpressed as the subform containment $\phi(X)\hat{\sigma}_{F(Y)(\nu)} \subset \hat{\psi}_{F(Y)(\nu)}$. Since $F(Y)(\nu)$ is also purely transcendental over $K$, we then deduce that $\phi(X)\hat{\sigma}_K \subset \hat{\psi}_K$. By Corollary \ref{CORspecialization}, it follows that $D(\hat{\psi})$ contains all products of the form $ab$ with $a \in D(\hat{\sigma})$ and $b \in D(\phi)$. Since these products generate $D(\hat{\sigma} \otimes \phi)$ as an $F$-vector space (see the remarks preceding Lemma \ref{LEMsubformoftensorproduct}), Lemma \ref{LEMsubformtheorem} then gives that $\anispart{(\hat{\sigma} \otimes \phi)} \subset \hat{\psi}$. Let us now suppose, for the sake of contradiction, that this containment is strict. By (2) and Lemma \ref{LEMboundond}, we then have that $\mydim{\anispart{(\hat{\sigma} \otimes \phi)}} - \mydim{\hat{\sigma}} < d$. Choose a quasi-Pfister subform $\pi' \subset \normform{\phi}$ of largest possible dimension such that $\pi'$ divides $\anispart{(\hat{\sigma} \otimes \phi)}$ and $[\pi'] \in P_{s}(\phi)$, where $s = \mathrm{lndeg}(\pi')$. Since $\mydim{\anispart{(\hat{\sigma} \otimes \phi)}} - \mydim{\hat{\sigma}} < d$, it follows from Lemma \ref{LEMboundond} and the induction hypothesis that $s>r$. On the other hand, since $\mydim{\anispart{(\hat{\sigma} \otimes \phi)}} < \mydim{\hat{\psi}} < 2^{\mathrm{lndeg}(\phi)}$, we have $s < \mathrm{lndeg}(\phi)$. By our choice of $\pi$, $\sigma$ cannot now be divisible by $\pi$. By (4), the same is then true of $\hat{\sigma}$ (see the proof of (4) above). Consider now the form $\rho : = \anispart{(\pi' \otimes \hat{\sigma})}$. Since $\hat{\sigma}$ is not divisible by $\pi'$, Lemmas \ref{LEMsubformoftensorproduct} and \ref{LEMdivisibilitybyquasiPfister} tell us that $\mydim{\rho} > \mydim{\hat{\sigma}} = \mydim{\eta} - \mydim{\sigma}$. Now, by Lemma \ref{LEMmaximaldividingform} (2), we have $\anispart{(\rho \otimes \phi)} \simeq \anispart{(\hat{\sigma} \otimes \phi)} \subset \hat{\psi} \subset \eta$. If we now apply Corollary \ref{CORinductivepropositiontensorbyPfister} to the triple $(\rho, \hat{\psi}, \eta)$, and then repeat the arguments above, we find anisotropic quasilinear quadratic forms $\hat{\rho}$ and $\hat{\hat{\psi}}$ such that $\rho \perp \hat{\rho} \simeq \eta \simeq \hat{\psi} \perp \hat{\hat{\psi}}$ and $\anispart{(\hat{\hat{\psi}} \otimes \phi)} \subset \hat{\rho}$. In fact, it is clear from the proof of Corollary \ref{CORinductivepropositiontensorbyPfister} that we can take $\hat{\hat{\psi}} = \psi$ here, and so we have that $\sigma = \anispart{(\psi \otimes \phi)} \subset \hat{\rho}$. But $\mydim{\hat{\rho}} = \mydim{\eta} - \mydim{\rho} < \mydim{\eta} - (\mydim{\eta} - \mydim{\sigma}) = \mydim{\sigma}$, so this is impossible. Our assumption was therefore incorrect, and so we indeed have that $\anispart{(\hat{\sigma} \otimes \phi)} \simeq \hat{\psi}$. 

(6) Since (2) and (5) hold, Proposition \ref{PROPlndegliesinP} implies that $[\normform{\phi}] \in P_{\mathrm{lndeg}(\phi)}(\hat{\sigma})$. Since $\mydim{\hat{\sigma}} < 2^{\mathrm{lndeg}(\phi)}$, this means that $\mydim{\anispart{(\normform{\phi} \otimes \hat{\sigma})}} = 2^{\mathrm{lndeg}(\phi)}$ (Remark \ref{REMPrdimension}), and so $\anispart{(\normform{\phi} \otimes \hat{\sigma})}$ is similar to $\normform{\phi}$. Since both $\hat{\sigma}$ and $\hat{\psi} \simeq \anispart{(\hat{\sigma} \otimes \phi)}$ are subforms of $\anispart{(\normform{\phi} \otimes \hat{\sigma})}$ (Lemma \ref{LEMsubformoftensorproduct}), and have dimension $< 2^{\mathrm{lndeg}(\phi)}$, the claim then follows.

(7) Taking $\nu = \phi$ in (3), we get that $d(\hat{\psi}_{F(\phi)}) = d(\psi_{F(\phi)})$ and $d(\hat{\sigma}_{F(\phi)}) = d(\sigma_{F(\phi)})$. By Lemmas \ref{LEMnonzerods} and \ref{LEMbasicfactsontensorproduct} (1), it follows that $d(\hat{\psi}_{F(\phi)})$ and $d(\hat{\sigma}_{F(\phi)})$ are positive integers whose sum is at most $d$. In particular, both are strictly less than $d$. Suppose now that $\mydim{\hat{\psi}} + \mydim{\hat{\sigma}} \leq 2^{\mathrm{lndeg}(\phi)}$. We first note:

\begin{claim} \label{LEMoneofpsihatandsigmahatisclose} In the above situation, $\mydim{\hat{\sigma}} < \mydim{\phi}$.
\begin{proof} Suppose otherwise. By (2) and our standing hypothesis, we then have
$$ 2^{\mathrm{lndeg}(\phi)} \geq \mydim{\hat{\psi}} + \mydim{\hat{\sigma}} = d + 2\mydim{\hat{\sigma}} > 2 \mydim{\phi} > 2^{n+1}, $$
and so $\mathrm{lndeg}(\phi) \geq n+2$. In other words, $\phi$ is not a quasi-Pfister neighbour (Lemma \ref{LEMcharacterizationofquasiPfisterneighbours}). By Corollary \ref{CORatleastonedissmall} one of $d(\hat{\sigma}_{F(\phi)}) = d(\sigma_{F(\phi)})$ and $d(\hat{\psi}_{F(\phi)}) = d(\psi_{F(\phi)})$ is then less than $c(\phi)$. Since both these integers are less than $d$, and since Corollary \ref{CORmainresult} holds when $k<d$ (we remind the reader that this is implied by the induction hypothesis), it then follows that at least one of $\mydim{\hat{\sigma}}$ and $\mydim{\hat{\psi}}$ lies within $c(\phi)$ of an integer multiple of $2^{\mathrm{lndeg}(\phi)}$. But since $\mydim{\phi} > c(\phi)$, our standing assumptions give that
$$ c(\phi) < \mydim{\phi} \leq \mydim{\hat{\sigma}} < \mydim{\hat{\psi}} \leq 2^{\mathrm{lndeg}(\phi)} - \mydim{\hat{\sigma}}  \leq 2^{\mathrm{lndeg}(\phi)} - \mydim{\phi} < 2^{\mathrm{lndeg}(\phi)} - c(\phi), $$
and so this is impossible. \end{proof}
\end{claim}

Now, since $\mydim{\hat{\sigma}} < \mydim{\phi}$, we have 
$$\mydim{\hat{\psi}} - \mydim{\phi} = (\mydim{\hat{\psi}} - \mydim{\hat{\sigma}}) - (\mydim{\phi} - \mydim{\hat{\sigma}}) =  d - (\mydim{\phi} - \mydim{\hat{\sigma}})<d. $$
At the same time, since $d< \mydim{\phi}$, the same inequalities show that $\mydim{\hat{\psi}} - \mydim{\phi} < \mydim{\hat{\sigma}}$. We are therefore in a position to apply the induction hypothesis with $\hat{\sigma}$ replacing $\phi$ and $\phi$ replacing $\psi$. Specifically, let $m\leq n$ be the unique non-negative integer with $2^m < \mydim{\hat{\sigma}} \leq 2^{m+1}$. The induction hypothesis then tells us that $[\normform{\hat{\sigma}}] \in P_{\mathrm{lndeg}(\hat{\sigma})}(\phi)$, and either $\mydim{\phi}$ or $\mydim{\hat{\psi}}$ lies within $2^{m-1}$ of $2^{\mathrm{lndeg}(\hat{\sigma})}$. Note, however, that $2^n < \mydim{\phi} < \mydim{\hat{\psi}} \leq 2^{\mathrm{lndeg}(\phi)} - \mydim{\hat{\sigma}}$ (the second inequality being valid by Lemma \ref{LEMsubformoftensorproduct}), so neither $\mydim{\phi}$ nor $\mydim{\hat{\psi}}$ lie within $2^{m-1}$ of a multiple of $2^{\mathrm{lndeg}(\phi)}$. We must therefore have that $\mathrm{lndeg}(\hat{\sigma}) < \mathrm{lndeg}(\phi)$, and so it now only remains to show that $\mydim{\hat{\sigma}} \leq 2^{n-1}$. If $\mathrm{lndeg}(\hat{\sigma}) \leq n-1$, then this is automatically the case. Note that this covers the case where $\phi$ is not a quasi-Pfister neighbour, since we then have that $P_i(\phi) = \emptyset$ for all $i \in [n,\mathrm{lndeg}(\phi)-1]$ by parts (3) and (4) of Lemma \ref{LEMbasicfactsonPr}. To complete the proof, it therefore only remains to consider the case where $\mathrm{lndeg}(\phi) = n+1$ and $\mathrm{lndeg}(\hat{\sigma}) = n$. Suppose, for the sake of contradiction, that $\mydim{\hat{\sigma}} > 2^{n-1}$. Then $\mydim{\hat{\psi}} \leq 2^{\mathrm{lndeg}(\phi)} - \mydim{\hat{\sigma}} < 2^{n+1} - 2^{n-1} = 2^n + 2^{n-1}$. In particular, $\mydim{\hat{\psi}} - \mydim{\phi} < 2^{n-1}$. If we now apply the full force of the induction hypothesis (again, with $\hat{\sigma}$ replacing $\phi$ and $\phi$ replacing $\psi$), then we get that $\hat{\psi}$ is divisible by $\normform{\hat{\sigma}}$ \big(see Remark \ref{REMStensorproducttheorem} (2)\big). In particular, $\mydim{\psi}$ is divisible by $2^{\mathrm{lndeg}(\hat{\sigma})} = 2^n$. By (6), however, we have $2^n < \mydim{\phi} < \mydim{\hat{\psi}} < 2^{\mathrm{lndeg}(\phi)} = 2^{n+1}$, so this is impossible. \end{proof} \end{proposition}

We can now prove the following proposition, which provides us with what we need to prove the theorem:

\begin{proposition} \label{PROPconstructionofalpha1} Let $n$ be the unique integer for which $2^n < \mydim{\phi} \leq 2^{n+1}$, and let $v$ be the largest integer for which $\mydim{\psi} > v2^{\mathrm{lndeg}(\phi)}$. Then there exists an anisotropic quasilinear quadratic form $\alpha$ over $F$ such that if we set $\beta: = \anispart{(\phi \otimes \alpha)}$, then:

\begin{enumerate} \item $\mydim{\beta} = \mydim{\alpha} + d < \mydim{\phi} + \mydim{\alpha};$
\item $\mydim{\alpha} = \mathrm{min}\lbrace \mydim{\psi} - v2^{\mathrm{lndeg}(\phi)}, (v+1)2^{\mathrm{lndeg}(\phi)} - \mydim{\sigma}\rbrace \leq 2^{n-1}$;
\item $\mathrm{lndeg}(\alpha) < \mathrm{lndeg}(\phi)$;
\item If $\pi'$ is a quasi-Pfister subform of $\normform{\phi}$, then the following are equivalent: 
\begin{enumerate} \item[$\mathrm{(i)}$] $\alpha$ is divisible by $\pi'$;
\item[$\mathrm{(ii)}$] $\beta$ is divisible by $\pi'$;
\item[$\mathrm{(iii)}$] $\psi$ is divisible by $\pi'$;
\item[$\mathrm{(iv)}$] $\sigma$ is divisible by $\pi'$. \end{enumerate}
In particular, $\alpha$ and $\beta$ are divisible by $\pi$;
\item If $\pi'$ is an $s$-fold quasi-Pfister subform of $\normform{\alpha}$ for some integer $r<s$, and $\pi'$ divides $\beta$, then $[\pi'] \notin P_s(\alpha)$;
\item $\beta$ is not divisible by $\normform{\alpha}$ unless $\alpha$ is similar to $\pi$ $($in which case $\normform{\alpha} \simeq \pi)$;
\item $d\big(\beta_{F(\phi)}\big) \leq d-\mydim{\alpha}$.
\end{enumerate}
\begin{proof} Let us first observe that (5) and (6) and (7) follow from the other parts:

(5) Suppose that this is false, and choose a counterexample where $\mydim{\pi'} = 2^s$ is as large as possible. By the proof of Lemma \ref{LEMmaximaldividingform}, we must then have that $[\pi'] \in P_s(\phi)$. At the same time, since $\pi'$ divides $\beta$, (4) tells us that it also divides $\sigma$. Since $s>r$, however, this now contradicts our choice of the pair $(r,\pi)$ (see the third condition in Lemma \ref{LEMmaximaldividingform}). 

(6) By (4), $\alpha$ and $\beta$ are divisible by $\pi$. If $\alpha$ is similar to $\pi$, then $\normform{\alpha} \simeq \pi$ and so $\beta$ is divisible by $\normform{\alpha}$. Conversely, suppose that $\beta$ is divisible by $\normform{\alpha}$. Since $[\normform{\alpha}] \in P_{\mathrm{lndeg}(\alpha)}(\alpha)$ \big(Lemma \ref{LEMbasicfactsonPr} (1)\big), (5) then implies that $\mathrm{lndeg}(\alpha) \leq r$. Since $\alpha$ is divisible by $\pi$ \big(part (4)\big), however, this holds if and only if $\alpha$ is similar to $\pi$. 

(7) By (2), we have $\mydim{\alpha} \leq 2^{n-1} < \mydim{\phi}$. By the separation theorem (or the induction hypothesis), it follows that $\alpha_{F(\phi)}$ is anisotropic. Then $d(\alpha_{F(\phi)}) = \mydim{\alpha}$, and so the claim follows from the first part of Lemma \ref{LEMbasicfactsontensorproduct} (with $\psi$ replaced by $\alpha$). 

It now remains to construct a form $\alpha$ satisfying (1)-(4). Consider the pair $(\hat{\psi},\hat{\sigma})$ constructed in Proposition \ref{PROPinductivestepinproofoftheorem}. By statements (1) and (6) of the latter, the equality in (2) may be rewritten as $\mathrm{dim}(\alpha) = \mathrm{min}\lbrace 2^{\mathrm{lndeg}(\phi)} - \mydim{\hat{\psi}}, \mydim{\hat{\sigma}} \rbrace$. We now separate two cases: \vspace{.5 \baselineskip}

\noindent {\it Case 1}. $\mydim{\hat{\psi}} + \mydim{\hat{\sigma}} \leq 2^{\mathrm{lndeg}(\phi)}$. In this case, the equality in (2) becomes $\mydim{\alpha} = \mydim{\hat{\sigma}}$. Set $\alpha: = \hat{\sigma}$. By part (5) of Proposition \ref{PROPinductivestepinproofoftheorem}, we then have that $\anispart{(\phi \otimes \alpha}) \simeq \hat{\psi}$. Moreover, (1) holds by part (2) of the proposition, and (2) and (3) hold by part (7) of the proposition \ref{PROPinductivestepinproofoftheorem}. For (4), part (4) of the proposition tells us that $\alpha$ is divisible by $\pi'$ if and only if $\sigma$ is, and that $\anispart{(\phi \otimes \alpha)}$ is divisible by $\pi'$ if and only if $\psi$ is. At the same time, it follows from Lemma \ref{LEManisotropicpartofformsdivisiblebyaquasiPfister} that any quasi-Pfister divisor of $\alpha$ divides $\anispart{(\phi \otimes \alpha)}$, and any quasi-Pfister divisor of $\psi$ divisor of $\sigma$. Statements (i)-(iv) are therefore equivalent. \vspace{.5 \baselineskip} 

\noindent {\it Case 2}. $\mydim{\hat{\psi}} + \mydim{\hat{\sigma}} > 2^{\mathrm{lndeg}(\phi)}$. In this case, the equality in (2) becomes $\mydim{\alpha}  = 2^{\mathrm{lndeg}(\phi)} - \mydim{\hat{\psi}}$. To achieve this, we apply the construction of Proposition \ref{PROPinductivestepinproofoftheorem} to the pair $(\hat{\sigma},\hat{\psi})$ to obtain anisotropic quasilinear quadratic forms $\hat{\hat{\sigma}}$ and $\hat{\hat{\psi}}$ satisfying the following:
\begin{itemize} \item[(a)] $\hat{\sigma} \perp \hat{\hat{\sigma}}$ and $\hat{\psi} \perp \hat{\hat{\psi}}$ are similar to $\normform{\phi}$;
\item[(b)] $\mydim{\hat{\hat{\sigma}}} - \mydim{\hat{\hat{\psi}}} = d$;
\item[(c)] If $\pi'$ is a quasi-Pfister subform of $\normform{\phi}$, then $\pi'$ divides $\hat{\hat{\sigma}}$ $($resp. $\hat{\hat{\psi}})$ if and only if $\pi'$ divides $\hat{\sigma}$ $($resp. $\hat{\psi})$. 
\item[(d)] $\anispart{(\hat{\hat{\psi}} \otimes \phi)} \simeq \hat{\hat{\sigma}}$;
\item[(e)] $\mydim{\hat{\hat{\psi}}} \leq 2^{n-1}$ and $\mathrm{lndeg}(\hat{\hat{\psi}}) < \mathrm{lndeg}(\phi)$.
\end{itemize}
Indeed, the only thing to be remarked on here is (e). But since $\mydim{\hat{\psi}} + \mydim{\hat{\sigma}} > 2^{\mathrm{lndeg}(\phi)}$, (a) gives that
$$ \mydim{\hat{\hat{\sigma}}} + \mydim{\hat{\hat{\psi}}} = 2^{\mathrm{lndeg}(\phi)+1} - (\mydim{\hat{\psi}} + \mydim{\hat{\sigma}}) < 2^{\mathrm{lndeg}(\phi)}, $$
and so (e) is just part (7) of Proposition \ref{PROPinductivestepinproofoftheorem} applied to the pair $(\hat{\sigma}, \hat{\psi})$. We now set $\alpha: = \hat{\hat{\psi}}$. Then (1) holds by (b) and (d), (2) and (3) hold by (e), and (4) holds by (c) together with part (4) of Proposition \ref{PROPinductivestepinproofoftheorem} (the argument here is essentially identical to that given in Case 1, but with a couple of additional steps to account for the double application of Proposition \ref{PROPinductivestepinproofoftheorem}).  \end{proof}
\end{proposition}

We are now ready to finish the proof. First:

\begin{corollary} \label{CORdvalue} $d = \mydim{\phi} - 2^r$.
\begin{proof} Suppose that this is not the case, and let $\alpha$ be the form constructed in Proposition \ref{PROPconstructionofalpha1}. By part (1) of the latter, we then have that $\mydim{\anispart{(\phi \otimes \alpha)}} = \mydim{\alpha} + d < \mydim{\phi} + (\mydim{\alpha} - 2^r)$. By the induction hypothesis (with $\alpha$ replacing $\phi$ and $\phi$ replacing $\psi$), it follows that there exists an integer $r<s \leq \mathrm{lndeg}(\alpha)$ and an anisotropic $s$-fold quasi-Pfister form $\pi'$ over $F$ such that $\pi'$ divides $\anispart{(\alpha \otimes \phi)}$ and $[\pi'] \in P_s(\alpha)$. Since $s \leq \mathrm{lndeg}(\alpha)$, Lemma \ref{LEMbasicfactsonPr} tells us that $\pi' \subset \normform{\alpha}$. But the preceding statement then contradicts one of the properties of $\alpha$ (see part (5) of Proposition \ref{PROPconstructionofalpha1}) so we must in fact have that $d = \mydim{\phi} - 2^r$. 
\end{proof}
\end{corollary}

We now conclude as follows: Set $\alpha_{-1} = \psi$, $\alpha_0 : = \phi$, and $\alpha_1: = \alpha$, where $\alpha$ is the form constructed in Proposition \ref{PROPconstructionofalpha1}. If $\alpha_1$ is similar to $\pi$, then we stop here. If not, then parts (1), (4), (5) and (6) of Proposition \ref{PROPconstructionofalpha1} show that the pair $(\phi,\alpha_1)$ enjoys all the properties of the pair $(\psi,\phi)$ that were used to construct $\alpha_1$ (note, in particular, that (4) and (5) tell us that $\pi$ plays the same role for both pairs). We can therefore repeat the procedure to pass from $(\phi,\alpha_1)$ to a new pair $(\alpha_1,\alpha_2)$. If $\alpha_2$ is similar to $\pi$, then we again stop. Otherwise we repeat. Since the $\alpha_i$ are of strictly decreasing dimension (part (2) of the proposition) and divisible by $\pi$ (part (4) of the proposition), this process must eventually lead us to a form $\alpha_m$ which is similar to $\pi$. Replacing $\alpha_m$ with $\pi$, the forms $\alpha_1,\hdots,\alpha_m$ then satisfy all the conditions in the statement of the theorem. Indeed, (2) holds since $\alpha_1,\hdots,\alpha_m,\psi,\phi$ and $\sigma$ are all divisible by $\alpha_m = \pi$, and:
\begin{itemize} \item (3)(i) holds trivially;
\item (3)(ii) holds by applying parts (2) and (3) of Proposition \ref{PROPconstructionofalpha1} to the pairs $(\alpha_{i-1},\alpha_i)$;
\item (3)(iii) holds by applying Proposition \ref{PROPlndegliesinP} to the pairs $(\alpha_{i-1},\alpha_i)$;
\item The first statement of (3)(iv) holds by applying Proposition \ref{PROPconstructionofalpha1} (1) and Corollary \ref{CORdvalue} to the pairs $(\alpha_{i-1},\alpha_i)$;
\item The second statement of (3)(iv) holds by applying Proposition \ref{PROPconstructionofalpha1} (2) to the pairs $(\alpha_{i-1},\alpha_i)$;
\item (3)(v) holds by applying Proposition \ref{PROPconstructionofalpha1} (7) to the pairs $(\alpha_{i-1},\alpha_i)$. 
\end{itemize}
This completes the proof. \end{proof}

\section{Optimality of Theorem \ref{THMmaintheoremcorollary} and Examples} \label{SECoptimality}

While Theorem \ref{THMmainresult} can surely be refined (see Remark \ref{REMsuboptimalityoftensorproducttheorem}), the improvements to be made here should be somewhat marginal. In fact, the proof of the theorem shows that the main issue at play is to understand the possible dimensions of anisotropic parts of tensor products of quasilinear quadratic forms, and a quick analysis here reveals that many of the scenarios allowed by Theorem \ref{THMmainresult} are in fact realizable. In this final section, we illustrate this by demonstrating the optimality of Theorem \ref{THMmaintheoremcorollary} (recall that this result takes the Izhboldin dimension into account, but not $\Delta$). For completeness, we also provide some examples illustrating what the latter result gives us beyond the original Conjecture \ref{CONJfirstconjecture}.

\subsection{Optimality} \label{SUBSECoptimality} In this subsection, we fix non-negative integers $s$ and $k$, and an integer $d \in [2^s,2^{s+1})$. If $n \in k + 2\mathbb{N}$, then we shall say that $n$ is \emph{realizable} if there exist a field $F$ of characteristic 2 and anisotropic quasilinear quadratic forms $p$ and $q$ over $F$ such that $2^s < \mydim{p} \leq 2^{s+1}$, $\Izhdim{p} = d$, $\mydim{q} = n$ and $d(q_{F(p)}) = k$. If such a triple exists with $p$ being a quasi-Pfister neighbour (resp. not a quasi-Pfister neighbour), then we will say that $n$ is \emph{realizable with $p$ a quasi-Pfister neighbour} (resp. \emph{realizable with $p$ not a quasi-Pfister neighbour}). This distinction is only relevant for the case where $d = 2^s$, since otherwise $p$ cannot be a quasi-Pfister neighbour (Lemma \ref{LEMcharacterizationofquasiPfisterneighbours}). For each non-negative integer $r$, let $y_r$ be the largest integer for which $d > y_r2^r$. Our goal is then to show the following:

\begin{proposition} \label{PROPoptimality} Let $a$ be a non-negative integer, and $\epsilon$ an integer with $\epsilon \equiv k \pmod{2}$.
\begin{enumerate} \item If $k \geq d$, then all elements of $k+2\mathbb{N}$ are realizable. Moreover, if $d = 2^s$ and $s \geq 2$, then all elements of $k+2\mathbb{N}$ are realizable with $p$ a quasi-Pfister neighbour or $p$ not a quasi-Pfister neighbour.
\item If $k<d = 2^s$, $a \geq 1$ and $\epsilon \in [-k,k]$, then $a2^{s+1} + \epsilon$ is realizable with $p$ a quasi-Pfister neighbour.
\item If $d > \mathrm{max}\lbrace k,2^s \rbrace$, $a \geq 1$ and $\epsilon \in [-k,k]$, then $a2^{s+2} + \epsilon$ is realizable.
\item Let $r \leq s-2$ be a non-negative integer, and $x$ a positive integer $\leq 2^{s-2-r}$. If $d = 2^s$, $k \in [2^s-2^r,2^s)$ and $\epsilon \in [(x-1)2^{r+1} + 2^{s+1} - k,x2^{r+1} + k]$, then $a2^{s+2} + \epsilon$ and $a2^{s+2} - \epsilon$ are realizable with $p$ not a quasi-Pfister neighbour.
\item Let $r \leq s-1$ be a non-negative integer, and $x$ a positive integer $<2^{s+1-r}-y_r$. If $d > 2^s$, $k \in [y_r2^r,d)$ and $\epsilon \in [(x+y_r)2^{r+1} - k,x2^{r+1} +k]$, then $a2^{s+2} + \epsilon$ is realizable.\end{enumerate}
\end{proposition}

Concretely, parts (2)-(5) show the statement of Theorem \ref{THMmaintheoremcorollary} cannot be improved without taking further information into account, while (1) shows that there is nothing to be gained by relaxing the condition that $k<\Izhdim{p}$. To prove the proposition, we shall make use of the following lemma which was already established in the course of proving Proposition \ref{PROPinductivestepinproofoftheorem} (the proof of statement (5), specifically):

\begin{lemma} \label{LEMconstructionofformsforoptimality} Let $E$ be a field of characteristic $2$, and let $\tau, \phi$ and $\sigma$ be anisotropic quasilinear quadratic forms over $E$ such that $\anispart{(\tau \otimes \phi)} \subset \sigma$. Set $F: = E(X)$, where $X$ is an indeterminate, and set $p : = \form{X} \perp \phi_F$ and $q: = \sigma_F \perp X\tau_F$. Then $p$ and $q$ are anisotropic and $d\big(q_{F(p)}\big) = \mydim{\sigma} - \mydim{\tau}$.
\end{lemma}

This gives us the following:

\begin{corollary} \label{CORrealizableintegers} Let $i$ be a positive integer. Suppose there exist a field $E$ of characteristic 2 and anisotropic quasilinear quadratic forms $\tau$ and $\phi$ over $E$ such that $\mydim{\tau} = i$, $\mydim{\phi} = d$ and $\mydim{\anispart{(\tau \otimes \phi)}} \leq k+i$. Then:
\begin{enumerate} \item $k+2i$ is realizable;
\item If $k+i <\mathrm{min}\lbrace 2d, 2^{\mathrm{lndeg}(\phi)} \rbrace$, then $a2^{\mathrm{lndeg}(\phi)+1} + (k+2i)$ and $a2^{\mathrm{lndeg}(\phi)+1} - (k+2i)$ are realizable for every positive integer $a$. 
\end{enumerate}
Moreover, if $\mathrm{lndeg}(\phi) = s$ $($resp. $\mathrm{lndeg}(\phi)>s)$, then ``realizable'' may be replaced in all cases by ``realizable with $p$ a quasi-Pfister neighbour'' $($resp. ``realizable with $p$ not a quasi-Pfister neighbour''$)$.
\begin{proof} We can assume that $1 \in D(\tau) \cap D(\phi)$. Replace $E$ with $E(X_1,\hdots,X_j)$, where $j$ is such that $\mydim{\anispart{(\tau \otimes \phi)}} = k+i-j$ and $X_1,\hdots,X_j$ are indeterminates. By Lemma \ref{LEManisotropyoverseparable}, this doesn't affect the anisotropy of the forms $\tau$ and $\phi$, nor the integer $\mathrm{lndeg}(\phi)$. Moreover, there exists an anisotropic quasilinear quadratic form of dimension $k+i$ over $E$ containing $\anispart{(\tau \otimes \phi)}$ as a subform (e.g., $\form{X_1,\hdots,X_j} \perp \anispart{(\tau \otimes \phi)}$). Let $\sigma$ be any such form. Set $F:= E(X)$, where $X$ is an indeterminate, and set $p: = \form{X} \perp \phi_F$ and $q: = \sigma_F \perp X\tau_F$. Then $\mydim{p} = d+1 \in (2^s,2^{s+1}]$, $\mydim{q} = k+2i$, and Lemma \ref{LEMconstructionofformsforoptimality} gives that $d\big(q_{F(p)}\big) = k$. At the same time, since $F(p)$ is a purely transcendental extension of $E$ (see the proof of Proposition \ref{PROPinductivestepinproofoftheorem}), $\phi_{F(p)}$ is anisotropic by Lemma \ref{LEManisotropyoverseparable}, and so $\Izhdim{p} = d$. This proves that $k+2i$ is realizable. At the same time, since $X$ is transcendental over $E$, we have $\normform{p} \simeq \pfister{X} \otimes (\normform{\phi})_F$, and so $\mathrm{lndeg}(p) = \mathrm{lndeg}(\phi) + 1$. Since $\mydim{\phi} =d \in [2^s, 2^{s+1})$, it follows that $p$ is a quasi-Pfister neighbour if and only if $\mathrm{lndeg}(\phi) = s$. We have now proven statement (1) with the additional qualification on $p$. Going further, suppose now that $k+i < 2^{\mathrm{lndeg}(\phi)}$ and that $\tau \subset \normform{\phi}$. By the second condition, $\anispart{(\tau \otimes \phi)}$ is a subform of $\anispart{(\normform{\phi} \otimes \normform{\phi})} \simeq \normform{\phi}$ (Lemma \ref{LEMdivisibilitybyquasiPfister}), and the first condition then implies that the form $\sigma$ considered above may also be chosen to be a subform of $\normform{\phi}$. We then have that $q = \sigma_F \perp X\tau_F \subset \pfister{X} \otimes (\normform{\phi})_F \simeq \normform{p}$. Choose now a positive integer $a$, and replace $F$ with $F(X_1,\hdots,X_a)$, where $X_1,\hdots,X_a$ are indeterminates (again, this changes nothing in the preceding discussion by Lemma \ref{LEManisotropyoverseparable}). Since $q \subset \normform{p}$, we may view $q$ as a subform of $\form{1,X_1,\hdots,X_{a-1}} \otimes \normform{p}$, which is anisotropic of dimension $a2^{\mathrm{lndeg}(\phi)+1}$. Applying Corollary \ref{CORinductivepropositionweak} with $\psi = q$, $\eta = \form{1,X_1,\hdots,X_{a-1}} \otimes \normform{p}$ and $\pi = \normform{p}$, we obtain a subform $q' \subset \form{1,X_1,\hdots,X_{a-1}} \otimes \normform{p}$ such that $\mydim{q'} = a2^{\mathrm{lndeg}(\phi)+1} - (k+2i)$ and $d\big(q'_{F(p)}\big) = k$. At the same time, $q'$ may in turn be viewed as a subform of the form $\form{1,X_1,\hdots,X_a} \otimes \normform{p}$, which is anisotropic of dimension $(a+1)2^{\mathrm{lndeg}(\phi)+1}$. If we now apply Corollary \ref{CORinductivepropositionweak} with $\psi = q'$, $\eta = \form{1,X_1,\hdots,X_a} \otimes \normform{p}$ and $\pi = \normform{p}$, then we get an anisotropic form $q''$ with $\mydim{q''} = a2^{\mathrm{lndeg}(\phi)+1} + (k+2i)$ and $d\big(q''_{F(p)}\big) = k$. Thus, in this case, both $a2^{\mathrm{lndeg}(\phi)+1} +(k+2i)$ and $a2^{\mathrm{lndeg}(\phi)+1} +(k+2i)$ are realizable, and the additional qualification on $p$ still applies. To complete the proof, it now only remains to show that $\tau \subset \normform{\phi}$ in the case where $k+i<2d$. We may assume here that $\mydim{\tau} \geq 2$. Let $\tau' \subset \tau$ be such that $\tau \simeq \form{1} \perp \tau'$, let $K = E(V_{\tau'})$, and let $\tau'(X) \in K$ be the generic value of $\tau$. Then $\rho: = \pfister{\tau'(X)} \otimes \phi_K$ is a $2d$-dimensional subform of $(\tau \otimes \phi)_K$. Since $\mydim{\anispart{(\tau \otimes \phi)}} \leq k+i < 2d$, $\rho$ must be isotropic. By Lemma \ref{LEMisotropyoverfunctionfieldsofquadrics}, this means that $\phi_{E(\tau)}$ is isotropic, and we then have that $\tau \subset \normform{\phi}$ by Lemma \ref{LEMdropingofnormdegree}, as desired. \end{proof}
\end{corollary}

Now, to apply this, we will need two lemmas. First:

\begin{lemma} \label{LEMexistenceoftensorproducts} Let $r$, $u$ and $v$ be non-negative integers with $1 \leq u \leq v$. There then exist a field $E$ of characteristic $2$ and anisotropic quasilinear quadratic forms $\pi$, $\tau$ and $\phi$ over $E$ such that:
\begin{enumerate} \item $\pi$ is an $r$-fold quasi-Pfister form that divides $\tau$ and $\phi$;
\item $\mydim{\tau} = u2^r$ and $\mydim{\phi} = v2^r$;
\item $\mydim{\anispart{(\tau \otimes \phi)}} \leq (u+v-1)2^r$;
\item If $n$ is the unique integer for which $2^n \leq v2^r < 2^{n+1}$, then $\mathrm{lndeg}(\phi) \leq n+1$, with equality holding unless $v2^r = 2^n$ and $u2^r > 2^{n-2}$ $($in which case $\mathrm{lndeg}(\phi)=n)$.
\end{enumerate}
\begin{proof} Let $t$ be the unique integer for which $2^{t-1} < u2^r \leq 2^t$. 

\begin{claim} \label{CLAIMtensorclaim} The lemma holds when $ v2^r \leq 2^t$.
\begin{proof} Set $E:=\mathbb{F}_2(X_1,\hdots,X_t)$, where $X_1,\hdots,X_t$ are indeterminates, and set $\pi: = \pfister{X_1,\hdots,X_r}$ and $\sigma: = \pfister{X_1,\hdots,X_t}$. Let $\tau$ and $\phi$ be subforms of $\sigma$ which are divisible by $\pi$ and have dimensions $u2^r$ and $v2^r$, respectively. Then (1) and (2) are satisfied, and the same is true of (4) since $\mathrm{lndeg}(\phi) = t$ ($\phi$ is a neighbour of $\sigma$). At the same time, $\anispart{(\tau \otimes \phi)}  \subset \anispart{(\sigma \otimes \sigma)} \simeq \sigma$ by Lemma \ref{LEMdivisibilitybyquasiPfister}, and so $\mydim{\anispart{(\tau \otimes \sigma)}} \leq 2^t$. But since $v2^r \geq u2^r > 2^{t-1}$, we have $(u+v-1)2^r \geq 2^t$, so (3) is also satisfied.
\end{proof} \end{claim}

In general, we argue by induction on $v$. The case where $v = 1$ is covered by Claim \ref{CLAIMtensorclaim}. Suppose now that $v \geq 2$, and that the statement holds for smaller values of $v$. Let $x$ be the largest integer for which $v2^r > x2^t$. If $x = 0$, we are done by Claim \ref{CLAIMtensorclaim}. Suppose now that $x \geq 1$, and set $m : = v2^r - x2^t \in [1,2^t]$. There then exist a field $E$ of characteristic 2 and anisotropic quasilinear quadratic forms $\pi$, $\tau$ and $\psi$ over $F$ such that:
\begin{itemize} \item[$\mathrm{(i)}$] $\pi$ is an $r$-fold quasi-Pfister form that divides $\tau$ and $\psi$;
\item[$\mathrm{(ii)}$] $\mydim{\tau} = u2^r$ and $\mydim{\psi} = m$;
\item[$\mathrm{(iii)}$] $\mydim{\anispart{(\tau \otimes \psi)}} \leq m + (u-1)2^r$. \end{itemize}
Indeed, if $m \geq u2^r$, then this follows from Claim \ref{CLAIMtensorclaim}, while if $m < u2^r$, then it follows from the induction hypothesis. Scaling if necessary, we can assume that $1 \in D(\tau)$. Now, if $\mathrm{lndeg}(\tau) \geq t+1$, then the separation theorem (or Theorem \ref{THMmaintheoremcorollary}) tells us that $\pi$, $\tau$ and $\psi$ all remain anisotropic over $E(\normform{\tau})$. At the same time, we have $\mathrm{lndeg}(\tau_{E(\normform{\tau})}) = \mathrm{lndeg}(\tau) - 1$ (see the remarks preceding Lemma \ref{LEMdropingofnormdegree}). Thus, by replacing $E$ with $E(\normform{\tau})$, then we can lower the value of $\mathrm{lndeg}(\tau)$ by $1$ without changing anything else. Repeating this a finite number of times, we eventually come to the case where $\mathrm{lndeg}(\tau) = t$, i.e., $\tau$ is a quasi-Pfister neighbour. Let us now replace $E$ with $E(X_1,\hdots,X_x)$ for indeterminates $X_1,\hdots,X_x$ (this doesn't affect the anisotropy of any of our forms by Lemma \ref{LEManisotropyoverseparable}), and set $\phi: = (\form{X_1,\hdots,X_x} \otimes \normform{\tau}) \perp \psi$. Since $\tau$ is a quasi-Pfister neighbour, $\phi$ is an anisotropic form of dimension $x2^t + m = v2^r$, and is divisible by $\pi$ by (i) (since $\tau$ is divisible by $\pi$, the same is true of $\normform{\tau}$). Now $\anispart{(\tau \otimes \normform{\tau})} \simeq \normform{\tau}$ by Lemma \ref{LEMdivisibilitybyquasiPfister}, and so
$$ \anispart{(\tau \otimes \phi)} \subset \big(\form{X_1,\hdots,X_x} \otimes \normform{\tau} \big) \perp \anispart{(\tau \otimes \psi)}. $$
By (iii), it then follows that $\mydim{\anispart{(\tau \otimes \phi)}} \leq x2^t + m + (u-1)2^r = (u+v-1)2^r$. The triple $(\pi, \tau, \phi)$ therefore satisfies (1), (2) and (3). Finally, let $n$ be as in (4). Since $x \neq 0$, we have $n>t$. Now the same argument as that used above to make $\tau$ a quasi-Pfister neighbour allows to us to make $\mathrm{lndeg}(\phi)$ at most $n+1$ without affecting the anisotropy of $\pi$, $\tau$ and $\phi$ (replace $E$ with the appropriate field in the Knebusch splitting tower of $\normform{\phi}$). If $v2^r>2^n$, we must then have that $\mathrm{lndeg}(\phi) = n+1$ (since $\mydim{\phi} = v2^r$). If not, then $v2^r = 2^n$, and so $x = 2^{n-t} - 1$. Since the $X_i$ are algebraically independent over the original field of definition of $\normform{\tau}$, we then have that $\mathrm{lndeg}(\phi) \geq \mathrm{lndeg}(\tau) + 2^{n-t}-1 = t + 2^{n-t}-1$. Observe now that $t + 2^{n-t} - 1$ can equal $n$ only if $t>n-2$. Thus, unless $v2^r = 2^n$ and $u2^r > 2^{n-2}$, we must have that $\mathrm{lndeg}(\phi) = n+1$. This completes the proof. \end{proof} \end{lemma}

Second:

\begin{lemma} \label{LEMexistenceoftensorproducts2} Let $i$ and $a$ be positive integers.
\begin{enumerate} \item If $i \leq a2^s$, then there exist a field $E$ of characteristic $2$, an anisotropic quasilinear quadratic form $\tau$ of dimension $i$ over $E$, and an anisotropic $s$-fold quasi-Pfister form $\pi$ over $E$ such that $\mydim{\anispart{(\tau \otimes \pi)}} \leq a2^s$. 
\item If $i \leq a2^{s+1}$, then there exist a field $E$ of characteristic $2$, and anisotropic quasilinear quadratic forms $\tau$ and $\phi$ over $E$ such that $\mydim{\tau} = i$, $\mydim{\phi} = d$, $\mathrm{lndeg}(\phi) = s+1$, and $\mydim{\anispart{(\tau \otimes \phi)}} \leq a2^{s+1}$. Moreover, if $d = 2^s$ and $s \geq 2$, then we can also arrange it so that $\mydim{\anispart{(\tau \otimes \phi)}} \leq 2^s +i $.  
\end{enumerate}

\begin{proof} By the $r=s$ case of Lemma \ref{LEMexistenceoftensorproducts}, there exists a field $E$ of characteristic $2$ and an anisotropic $s$-fold quasi-Pfister form $\pi$ over $F$. Replace $E$ with $E(X_1,\hdots,X_a)$, where $X_1,\hdots,X_a$ are indeterminates. By Lemma \ref{LEManisotropyoverseparable}, this doesn't affect the anisotropy of $\pi$.

(1) The form $\eta: = \form{X_1,\hdots,X_a} \otimes \pi$ is anisotropic of dimension $a2^s$. Since $i\leq a2^s$, we can choose an $i$-dimensional subform $\tau$ of $\eta$. Since $\eta$ is divisible by $\pi$, we have $\anispart{(\eta \otimes \pi)} \simeq \eta$ by Lemma \ref{LEMdivisibilitybyquasiPfister}, and so $\anispart{(\tau \otimes \pi)} \subset \eta$. In particular, $\mydim{\anispart{(\tau \otimes \pi)}} \leq \mydim{\eta} = a2^s$. 

(2) Let us now consider the (anisotropic) $(s+1)$-fold quasi-Pfister form $\sigma := \pfister{X_a} \otimes \pi$. To prove the desired assertion, we can assume that $i>(a-1)2^{s+1}$. Let $j \in [1,2^{s+1}]$ be such that $i = (a-1)2^{s+1} + j$, and set $\tau: = (\form{X_1,\hdots,X_{a-1}} \otimes \sigma) \perp \psi$, where $\psi$ is a $j$-dimensional subform of $\sigma$ such that $\psi \subset \pi$ in the case where $j \leq 2^s$. Then $\tau$ is anisotropic of dimension $i$. Now, let $\pi' \subset \pi$ be such that $\pi \simeq \form{1} \perp \pi'$, and let $\phi$ be any $d$-dimensional subform of $\sigma$ containing $\form{X_a} \perp \pi'$ as a subform. By Lemma \ref{LEMdivisibilitybyquasiPfister}, we have $\anispart{(\sigma \otimes \phi)} \simeq \sigma \simeq \anispart{(\psi \otimes \sigma)}$, and so
\begin{equation} \label{eq3} \anispart{(\tau \otimes \phi)} \simeq (\form{X_1,\hdots,X_{a-1}} \otimes \sigma) \perp \anispart{(\psi \otimes \phi)} \subset \form{1,X_1,\hdots,X_{a-1}} \otimes \sigma. \end{equation}
In particular, $\mydim{\anispart{(\tau \otimes \phi)}} \leq a2^{s+1}$. If $d>2^s$, then $\mathrm{lndeg}(\phi) = s+1$ (because $\phi \subset \sigma$), and we are done. Suppose now that $d = 2^s$ and $s \geq 2$. Since $d = 2^s$, we have $\phi = \form{X_a} \perp \pi'$. Since $d \geq 2$, $\pi'$ is a neighbour of $\pi$, and so $\normform{\pi'} \simeq \pi $ by Lemma \ref{LEMknownstablebirationalinvariants} (1). As $X_a$ is transcendental over the original field of definition of $\pi$, we then have that $\normform{\phi} \simeq \pfister{X_a} \otimes \pi$, and hence $\mathrm{lndeg}(\phi) = s+1$. It remains to check that $\mydim{\anispart{(\tau \otimes \phi)}} \leq 2^s +i$ in this case. If $j \geq 2^s$, then this follows from the inequality $\mydim{\anispart{(\tau \otimes \phi)}} \leq a2^{s+1}$. If $j < 2^s$, then we have $\psi \subset \pi$ by construction, and so $\anispart{(\psi \otimes \pi')} \subset \anispart{(\pi \otimes \pi)} \simeq \pi$ by Lemma \ref{LEMdivisibilitybyquasiPfister}. Since $\phi = \form{X_a} \perp \pi'$, we then have that $\anispart{(\psi \otimes \phi)} \subset (X_a\psi) \perp \pi$, and so
$$ \mydim{\anispart{(\psi \otimes \phi)}} \leq (a-1)2^{s+1} +  \mydim{\psi} + \mydim{\pi} = (a-1)2^{s+1} + j + 2^s = 2^s + i $$
by \eqref{eq3}. This proves the lemma.
\end{proof}
\end{lemma}

We can now prove Proposition \ref{PROPoptimality}:

\begin{proof}[Proof of Proposition \ref{PROPoptimality}] (1) Suppose that $k\geq d$, and let $i$ be a positive integer. By the $r=0$ case of Lemma \ref{LEMexistenceoftensorproducts}, there exists a field $E$ of characteristic 2 and anisotropic quasilinear quadratic forms $\tau$ and $\phi$ over $E$ such that $\mydim{\tau} = i$, $\mydim{\phi} = d$ and $\mydim{\anispart{(\tau \otimes \phi)}} \leq d+i$ (note that when $i>d$, we are switching the roles of $\tau$ and $\phi$ in the statement of Lemma \ref{LEMexistenceoftensorproducts}). Since $k \geq d$, Corollary \ref{CORrealizableintegers} then implies that $k+2i$ is realizable. Moreover, to justify the additional remarks regarding the case where $d = 2^s$, we just have to show that $\phi$ can be chosen so that $\mathrm{lndeg}(\phi) = s$, or $\mathrm{lndeg}(\phi) = s+1$ when $s \geq 2$. But this is covered by Lemma \ref{LEMexistenceoftensorproducts2}, and so the result holds.

(2) Suppose that $k<2^s = d$ and $\epsilon \in [-k,k]$. Set $i: = a2^s + \frac{\epsilon - k}{2}$. Then $i< a2^s$, so Lemma \ref{LEMexistenceoftensorproducts2} (1) tells us that there exist a field $E$ of characteristic $2$, an anisotropic quasilinear quadratic form $\tau$ of dimension $i$ over $E$, and an anisotropic $s$-fold quasi-Pfister form $\phi$ over $E$ such that $\mydim{\anispart{(\tau \otimes \phi)}} \leq a2^s$. Now $k+i = a2^s + \frac{\epsilon + k}{2} \geq a2^s$, and since $\mathrm{lndeg}(\phi) = s$, Corollary \ref{CORrealizableintegers} then tells us that $k+2i = a2^{s+1} + \epsilon$ is realizable with $p$ a quasi-Pfister neighbour.

(3) Suppose that $d> \mathrm{max}\lbrace k,2^s \rbrace$, $a \geq 1$ and $\epsilon \in [-k,k]$. Set $i: = a2^{s+1} + \frac{\epsilon - k}{2}$. Then $i < a2^{s+1}$, so Lemma \ref{LEMexistenceoftensorproducts2} (2) tells us that there exist a field $E$ of characteristic $2$, and anisotropic quasilinear quadratic forms $\tau$ and $\phi$ over $E$ such that $\mydim{\tau} = i$, $\mydim{\phi} = d$ and $\mydim{\anispart{(\tau \otimes \phi)}} \leq a2^{s+1}$. Since $k+i = a2^{s+1} + \frac{\epsilon +k}{2} \geq a2^{s+1}$, Corollary \ref{CORrealizableintegers} then tells us that $k+2i = a2^{s+2} + \epsilon$ is realizable.

(4) Let $r$ be a non-negative integer $\leq s-2$, and $x$ a positive integer $\leq 2^{s-2-r}$. Suppose that $d = 2^s$, $k \in [2^s-2^r,2^s)$ and $\epsilon \in [(x-1)2^{r+1} + 2^{s+1} - k,x2^{r+1} + k]$.  Set $i: = \frac{\epsilon - k}{2} \in [(x-1)2^r + 2^s - k, x2^r]$. Since $x2^r \leq 2^{s-2}$, we can apply Lemma \ref{LEMexistenceoftensorproducts} with $u = x$ and $v = 2^{s-r}$ to get a field $E$ of characteristic 2 and anisotropic quasilinear quadratic forms $\tau$ and $\phi$ over $E$ such that $\mydim{\tau} = i$, $\mydim{\phi} = 2^s$, $\mathrm{lndeg}(\phi) = s+1$ and $\mydim{\anispart{(\tau \otimes \phi)}} \leq (x-1)2^r + 2^s$. The second and third points tell us that $\phi$ is not similar to a quasi-Pfister form, and since $\epsilon \geq (x-1)2^{r+1} + 2^{s+1} - k$, the fourth tells us that $\mydim{\anispart{(\tau \otimes \phi)}} \leq \frac{\epsilon + k}{2} = k+i$. By Corollary \ref{CORrealizableintegers} (1), it follows that $k + 2i = \epsilon$ is realizable with $p$ not a quasi-Pfister neighbour. At the same time, since $\epsilon \leq x2^{r+1} + k$, we have 
$$ k+i = \frac{\epsilon + k}{2} \leq x2^r + k < 2^{s-2} + 2^s < 2^{s+1}= \mathrm{min}\lbrace 2d, 2^{\mathrm{lndeg}(\phi)} \rbrace,$$
and so Corollary \ref{CORrealizableintegers} (2) tells us that $a2^{s+2} + \epsilon$ and $a2^{s+2} - \epsilon$ are also realizable with $p$ not a quasi-Pfister neighbour when $a \geq 1$. 

(5) Let $r$ be a non-negative integer $\leq s-1$, and let $x$ be a positive integer $< 2^{s+1-r}- y_r$. Suppose that $d>2^s$, $k \in [y_r2^r,d)$ and $\epsilon \in [(x+y_r)2^{r+1} - k, x2^{r+1} + k]$. As in (4), set $i: = \frac{\epsilon -k}{2} \in [(x+y_r)2^r-k,x2^r$]. By hypothesis, we have $x2^r \leq 2^{s+1} - (y_r+1)2^{r}$. Since $r \leq s-1$ and $d>2^s$, however, $y_r2^r \geq 2^s$, and so $x2^r \leq 2^s - 2^r$. We can therefore apply Lemma \ref{LEMexistenceoftensorproducts} with $u = x$ and $v = y_r+1$ to get a field $E$ of characteristic 2 and anisotropic quasilinear quadratic forms $\tau$ and $\phi$ over $E$ such that $\mydim{\tau} = i$, $\mydim{\phi} = d$, $\mathrm{lndeg}(\phi)= s+1$ and $\mydim{\anispart{(\tau \otimes \phi)}} \leq (x+y_r)2^r$. As in (4), $\phi$ is then not similar to a quasi-Pfister form, and the inequality $\epsilon \geq (x+y_r)2^{r+1} - k$ gives that $\mydim{\anispart{(\tau \otimes \phi)}} \leq k+i$. By Corollary \ref{CORrealizableintegers} (1), this shows that $k + 2i = \epsilon$ is realizable with $p$ not a quasi-Pfister neighbour. At the same time, the inequality $\epsilon \leq x2^{r+1} +k$ again gives that $k+i \leq x2^r + k$. Since $x2^r \leq 2^{s+1} - (y_r+1)2^{r} \leq 2^{s+1} - d < 2^{s+1} - k$, it follows that $k+i < 2^{s+1} = \mathrm{min}\lbrace 2d,2^{\mathrm{lndeg}(\phi)}\rbrace$, and so Corollary \ref{CORrealizableintegers} (2) tells that $a2^{s+2}  + \epsilon$ is also realizable with $p$ not a quasi-Pfister neighbour when $a \geq 1$. This completes the proof. \end{proof}

\subsection{Examples} \label{SUBSECexamples} To conclude, we give some examples illustrating what Theorem \ref{THMmaintheoremcorollary} gives beyond Corollary \ref{CORmaintheoremcorollarycritical}. In all cases, $p$ will denote an anisotropic quasilinear quadratic form over a field $F$ of characteristic $2$ which is not a quasi-Pfister neighbour and has dimension in the interval $[17,32]$. We consider an anisotropic quasilinear quadratic form $q$ of dimension $\geq 17$ over $F$ such that $q_{F(p)}$ is isotropic, and consider the possible values of $\mydim{q}$ modulo 64 in terms of the integer $k = \mydim{q} - 2\witti{0}{q_{F(p)}}$. For any given $k$, the values $-k,-k+2,\hdots,k-2,k$ are always possible (subject to the requirement that $\mydim{q}>k$). The tables below list the other possible values in the given situations.

\begin{example} If $\Izhdim{p} = 16$, we have: \vspace{.5 \baselineskip}

\begin{center}
\begin{tabular}{c|c} $k$ & Additional possible values of $\mydim{q}$ modulo 64 \\ \hline
$<12$ & None \\
$ 12$ & $\pm 20$ \\
$13$ & $\pm 19, \pm 21$ \\
$14$ & $\pm 18, \pm 20, \pm 22$ \\
$15$ & $\pm 17, \pm 19, \pm 21, \pm 23$ \\
$\geq 16$ & Any additional value $\equiv k \pmod{2}$ \\
\end{tabular} \vspace{.5 \baselineskip}
\end{center}
Note that in this case, we must have $\mydim{p} \leq 20$, since otherwise Proposition \ref{PROPdivisibilityofp1} would tell us that the $16$-dimensional anisotropic form $p_1$ would be divisible by a quasi-Pfister form of foldness $\geq 3$, forcing it to be similar to a quasi-Pfister form. But $p$ would then be a quasi-Pfister neighbour by Lemma \ref{LEMcharacterizationofquasiPfisterneighbours}, contradicting our hypothesis.
\end{example}

\begin{example} If $\Izhdim{p} = 23$, we have: \vspace{.5 \baselineskip}

\begin{center}
\begin{tabular}{c|c} $k$ & Additional possible values of $\mydim{q}$ modulo 64 \\ \hline
$<16$ & None \\
$ 16$ & $32$ \\
$17$ & $\pm 31$ \\
$18$ & $\pm 30, 32$ \\
$19$ & $\pm 29, \pm 31$ \\
$20$ & $\pm 28, \pm 30, 32$ \\
$21$ & $\pm 27, \pm 29, \pm 31$ \\
$22$ & $\pm 26, \pm 28, \pm 30, 32$ \\
$\geq 23$ & Any additional value $\equiv k \pmod{2}$
\end{tabular} \vspace{.5 \baselineskip}
\end{center}
Note that in this case, we must have $\mydim{p} = 24$ by Corollary \ref{CORintro} (1).
\end{example}

\begin{example} If $\Izhdim{p} = 28$, then we have: \vspace{.5 \baselineskip}

\begin{center}
\begin{tabular}{c|c} $k$ & Additional possible values of $\mydim{q}$ modulo 64 \\ \hline
$< 24$ & None \\
$24$ &  $32$ \\
$25$ & $\pm 31$ \\
$26$ & $\pm 30$, $32$ \\
$ \geq 27$ & Any additional value $\equiv k \pmod{2}$
\end{tabular} \vspace{.5 \baselineskip}
\end{center}
In this case, $\mydim{p}$ can take any value between $29$ and $32$.
\end{example}
\vspace{.5 \baselineskip}

\noindent {\bf Acknowledgements.} This work was supported by NSERC Discovery Grant No. RGPIN-2019-05607.

\bibliographystyle{alphaurl}

\end{document}